\documentclass[11pt]{article}
\usepackage{amsmath,amssymb,amsthm,amscd}

\providecommand{\scr}{\mathcal}

\newtheorem{theorem}{Theorem}[section]
\newtheorem{lemma}[theorem]{Lemma}
\newtheorem{proposition}[theorem]{Proposition}
\newtheorem{corollary}[theorem]{Corollary}
\newtheorem{conjecture}[theorem]{Conjecture}

\newtheorem*{thmA}{Theorem A.}
\newtheorem*{thmB}{Theorem B.}
\newtheorem*{thmC}{Theorem C.}

\theoremstyle{definition}
\newtheorem{definition}[theorem]{Definition}
\newtheorem{remark}[theorem]{Remark}
\newtheorem{example}[theorem]{Example}

\numberwithin{equation}{section}

\renewcommand{\phi}{\varphi}

\newcommand{\ep}{\varepsilon}

\newcommand{\AC}{\operatorname{AC}}

\newcommand{\Ad}{\operatorname{Ad}}
\newcommand{\Aut}{\operatorname{Aut}}
\newcommand{\diag}{\operatorname{diag}}
\newcommand{\Hom}{\operatorname{Hom}}
\newcommand{\id}{\operatorname{id}}
\newcommand{\Ima}{\operatorname{Ima}}
\newcommand{\Lip}{\operatorname{Lip}}

\newcommand{\N}{\mathbb{N}}
\newcommand{\Z}{\mathbb{Z}}

\newcommand{\R}{\mathbb{R}}
\newcommand{\C}{\mathbb{C}}
\newcommand{\T}{\mathbb{T}}

\newcommand{\K}{\mathbb{K}}

\title{Poly-$\Z$ group actions on Kirchberg algebras I}
\author{Masaki Izumi
\thanks{Supported in part by JSPS KAKENHI Grant Number JP15H03623}\\
Graduate School of Science \\
Kyoto University \\
Sakyo-ku, Kyoto 606-8502, Japan \\
izumi@math.kyoto-u.ac.jp
\and
Hiroki Matui
\thanks{Supported in part by JSPS KAKENHI Grant Number JP18K03321}\\
Graduate School of Science \\
Chiba University \\
Inage-ku, Chiba 263-8522, Japan \\
matui@math.s.chiba-u.ac.jp}
\date{}

\begin{document}
\maketitle

\begin{abstract}
Toward the complete classification of poly-$\Z$ group actions 
on Kirchberg algebras, 
we prove several fundamental theorems 
that are used in the classification. 
In addition, as an application of them, 
we classify outer actions of poly-$\Z$ groups 
of Hirsch length not greater than three on unital Kirchberg algebras 
up to $KK$-trivial cocycle conjugacy. 
\end{abstract}

\tableofcontents

\section{Introduction}

The present paper is a continuation of the work in \cite{IM10Adv}. 
Our aim is to classify group actions on $C^*$-algebras 
up to cocycle conjugacy. 
In the setting of von Neumann algebras, 
a complete classification is known 
for actions of countable amenable groups on injective factors 
(\cite{MR1621416,Ma13crelle}). 
However, classification of group actions
on $C^*$-algebras is still a far less developed subject, 
partly because of $K$-theoretical difficulties. 
The present technology in the realm of $C^*$-algebras 
does not allow us to attempt to classify actions 
of all amenable groups on all simple nuclear $C^*$-algebras. 
Thus, we need to restrict our attention to 
certain reasonable classes of groups and algebras. 

Kirchberg algebras form one of the most prominent classes 
of $C^*$-algebras from the viewpoint of the Elliott program, 
which aims to classify all separable simple nuclear $C^*$-algebras 
up to isomorphism by invariants coming from $K$-theory. 
In fact, Kirchberg algebras are completely classified 
up to (stable) isomorphism by $KK$-theory (\cite{Ki94prep,Ph00DocM,Ro_text}). 
It is then natural to study group actions on Kirchberg algebras. 
The first result in this direction was obtained 
by H. Nakamura \cite{Na00ETDS}. 
He showed that outer $\Z$-actions are completely classified 
by their $KK$-classes up to $KK$-trivial cocycle conjugacy, 
following the strategy developed by A. Kishimoto 
in his pioneering works \cite{Ki95crelle,Ki96JFA,Ki98JOP,Ki98JFA}. 
On the one hand, the first-named author \cite{Iz04Duke,Iz04Adv} 
completely classified finite group actions with the Rohlin property 
on Kirchberg algebras. 
However, unlike the $\Z$ case where the Rohlin property is automatic, 
there exist several outer finite group actions without the Rohlin property. 
For example, the Cuntz algebra $\mathcal{O}_2$, 
which is $KK$-equivalent to $0$, 
admits uncountably many outer actions of $\Z_2$ 
that are not cocycle conjugate to each other, 
while the action with the Rohlin property is unique. 
Such a phenomenon does not appear in the context of von Neumann algebras, 
and arises from the difficulties of topological nature of finite groups. 
Indeed the classifying space of a non-trivial finite group is 
never finite dimensional. 
On the other hand, 
after the second-named author's prior work \cite{Ma08Adv}, 
we proved that 
outer $\Z^N$-actions on strongly self-absorbing Kirchberg algebras 
satisfying the UCT are unique up to cocycle conjugacy, 
and also succeeded in classifying locally $KK$-trivial outer $\Z^2$-actions 
on Kirchberg algebras in \cite{IM10Adv}. 
The classifying space of $\Z^N$ is $\T^N$, 
a very nice finite dimensional topological space, 
and the results obtained in \cite{IM10Adv} suggest 
the possibility of generalizing them to a larger class of 
discrete groups whose classifying spaces are well-behaved enough. 
From this perspective we make the following conjectures, 
which are slightly strengthened versions of those 
given by the first-named author \cite{Iz10ProcICM}. 

\begin{conjecture}\label{conj_ssa}
Let $G$ be a countable torsion-free amenable group and 
let $\mathcal{D}$ be a strongly self-absorbing $C^*$-algebra. 
There exists a unique strongly outer action of $G$ on $\mathcal{D}$ 
up to cocycle conjugacy. 
\end{conjecture}

\begin{conjecture}\label{classification}
Let $G$ be a countable torsion-free amenable group and 
let $A$ be a unital Kirchberg algebra. 
For outer actions $\alpha,\beta$ of $G$ on $A$, 
the following conditions are equivalent. 
\begin{enumerate}
\item $\alpha$ and $\beta$ are KK-trivially cocycle conjugate. 
\item There exists a base point preserving isomorphism 
between $\mathcal{P}_{\alpha^s}$ and $\mathcal{P}_{\beta^s}$. 
\end{enumerate}
\end{conjecture}

Here $\mathcal{P}_{\alpha^s}$ denotes 
the principal $\Aut(A\otimes\K)$-bundle over the classifying space of $G$ 
associated with the stabilization $\alpha^s:G\to\Aut(A\otimes\K)$. 
Our goal is 
to give an affirmative answer for Conjecture \ref{classification} 
in the case that $G$ is a poly-$\Z$ group 
(see Section 2.2 for its definition), 
and it will be carried out in our forthcoming paper \cite{IMinprep}. 
In this paper, as a preliminary step toward this goal, 
we establish a sufficient condition 
implying $KK$-trivial cocycle conjugacy between poly-$\Z$ group actions. 
Namely we prove the following. 

\begin{thmA}[Theorem \ref{KKtrivcc}]
Let $\mu^G:G\curvearrowright\mathcal{O}_\infty$ be an outer action 
of a poly-$\Z$ group $G$ 
and let $A$ be a unital separable $C^*$-algebra. 
Let $(\alpha,u):G\curvearrowright A$ and $(\beta,v):G\curvearrowright A$ be 
cocycle actions belonging to $\AC(\mathcal{O}_\infty,\mu^G)$. 
Suppose that 
there exists a family $(x_g)_g$ of unitaries in $C^b([0,\infty),A)$ such that 
\[
\lim_{t\to\infty}
(\Ad x_g(t)\circ\alpha_g)(a)=\beta_g(a)\quad\forall g\in G,\ \forall a\in A, 
\]
\[
\lim_{t\to\infty}
x_g(t)\alpha_g(x_h(t))u(g,h)x_{gh}^*(t)=v(g,h)\quad\forall g,h\in G. 
\]
Then $(\alpha,u)$ and $(\beta,v)$ are cocycle conjugate 
via an asymptotically inner automorphism. 
\end{thmA}

The class $\AC(\mathcal{O}_\infty,\mu^G)$ of cocycle actions 
is introduced in Definition \ref{defofAC}, 
and any outer cocycle actions of $G$ on a unital Kirchberg algebra 
belong to $\AC(\mathcal{O}_\infty,\mu^G)$ (Lemma \ref{AC} (3)). 
Next, by using this theorem, 
we classify outer actions of poly-$\Z$ groups 
of Hirsch length not greater than three on a unital Kirchberg algebra 
(Theorem \ref{uni_Hirsch2} and Theorem \ref{uni_Hirsch3}). 
More precisely, for given actions $\alpha,\beta:G\curvearrowright A$, 
we introduce obstruction classes 
$\mathfrak{o}^2(\alpha,\beta)$ and $\mathfrak{o}^3(\alpha,\beta)$ 
living in $H^2(G,KK^1(A,A))$ and $H^3(G,KK(A,A))$ respectively, 
and prove that $\alpha$ and $\beta$ are $KK$-trivially cocycle conjugate 
if and only if these obstruction classes vanish. 

\begin{thmB}[Theorem \ref{uni_Hirsch2}]
Let $A$ be a unital Kirchberg algebra and 
let $G$ be a poly-$\Z$ group of Hirsch length two. 
Let $\alpha,\beta:G\curvearrowright A$ be outer actions. 
The following are equivalent. 
\begin{enumerate}
\item $\alpha$ and $\beta$ are $KK$-trivially cocycle conjugate. 
\item $KK(\alpha_g)=KK(\beta_g)$ for all $g\in G$ and 
$\mathfrak{o}^2(\alpha,\beta)=0$. 
\end{enumerate}
\end{thmB}

\begin{thmC}[Theorem \ref{uni_Hirsch3}]
Let $A$ be a unital Kirchberg algebra and 
let $G$ be a poly-$\Z$ group of Hirsch length three. 
Let $\alpha,\beta:G\curvearrowright A$ be outer actions. 
The following are equivalent. 
\begin{enumerate}
\item $\alpha$ and $\beta$ are $KK$-trivially cocycle conjugate. 
\item $KK(\alpha_g)=KK(\beta_g)$ for all $g\in G$, 
$\mathfrak{o}^2(\alpha,\beta)=0$ and $\mathfrak{o}^3(\alpha,\beta)=0$. 
\end{enumerate}
\end{thmC}

Moreover, we discuss the existence type results of outer cocycle actions 
of poly-$\Z$ groups of Hirsch length not greater than three 
(Theorem \ref{exi_Hirsch2} and Theorem \ref{exi_Hirsch3}). 

The paper is organized as follows. 
In Section 2, we collect for reference basic notation, terminology 
and definitions. 
In Section 3, it is shown that any countable discrete amenable group 
admits an asymptotically representable outer action 
on the Cuntz algebra $\mathcal{O}_\infty$ (Theorem \ref{exist_asymprepre}). 
In Section 4, we give an affirmative answer for Conjecture \ref{conj_ssa} 
when $G$ is a poly-$\Z$ group and 
$\mathcal{D}$ is a strongly self-absorbing Kirchberg algebra 
satisfying the UCT 
(Theorem \ref{Oinfty_unique}). 
Furthermore, we show that 
any outer cocycle action of a poly-$\Z$ group $G$ 
on a unital Kirchberg algebra 
absorbs an outer action of $G$ on $\mathcal{O}_\infty$ tensorially 
(Theorem \ref{Oinfty_absorb}). 
Our proof is along the same line as that of \cite{IM10Adv}. 
We remark that G. Szab\'o \cite{Sz1807arXiv} has recently proved 
Conjecture \ref{conj_ssa} affirmatively 
for a large class of groups containing poly-$\Z$ groups and 
all strongly self-absorbing $C^*$-algebras (see Remark \ref{workofSzabo}). 
Indeed, Conjecture \ref{conj_ssa} is 
the same as \cite[Conjecture A]{Sz1807arXiv}. 
In Section 5, it is shown that 
every poly-$\Z$ group has the properties of 
asymptotic $H^1$-stability and $H^2$-stability (Theorem \ref{H1H2stable}). 
Loosely speaking, asymptotic $H^1$-stability says that 
any $1$-cocycle in the unitary group close to $1$ can be written 
as a continuous limit of coboundaries, and 
$H^2$-stability says that 
any $2$-cocycle in the unitary group close to $1$ is 
actually a coboundary (see Definition \ref{defofH1H2}). 
The proof is by induction on the Hirsch length of poly-$\Z$ groups. 
In Section 6, 
by using the properties of stability obtained in the previous section, 
we show that 
two outer actions $\alpha,\beta:G\curvearrowright A$ are 
$KK$-trivially cocycle conjugate 
if there exists a continuous family of (almost) $\alpha$-$1$-cocycles 
such that the perturbation of $\alpha$ by the cocycle equals $\beta$ 
in the limit (Theorem \ref{KKtrivcc}). 
The proof uses the Evans-Kishimoto intertwining argument. 
But, in contrast to the setting of von Neumann algebras 
(see the work \cite{Ma13crelle} of T. Masuda for instance), 
we could not find a way of proof applicable to 
all countable torsion-free groups. 
Thus, the proof is again by (a bit tricky) induction on the Hirsch length. 
It is also interesting to point out that 
the assumption for $\alpha$ and $\beta$ in Theorem \ref{KKtrivcc} is 
not symmetric, while the conclusion is symmetric. 
We do not know whether this is a special feature of poly-$\Z$ groups. 
In the case of general torsion-free groups, 
it is perhaps natural to assume further that 
$\beta$ is an asymptotic cocycle perturbation of $\alpha$. 
In Section 7, we discuss actions of poly-$\Z$ groups of Hirsch length two. 
It is shown that $\alpha$ and $\beta$ are $KK$-trivially cocycle conjugate 
if and only if the obstruction class $\mathfrak{o}^2(\alpha,\beta)$ is trivial 
(Theorem \ref{uni_Hirsch2}). 
Moreover, we prove the existence of outer cocycle actions 
with prescribed $K$-theoretic data (Theorem \ref{exi_Hirsch2}). 
In Section 8, we discuss actions of poly-$\Z$ groups of Hirsch length three. 
It is shown that $\alpha$ and $\beta$ are $KK$-trivially cocycle conjugate 
if and only if the obstruction classes 
$\mathfrak{o}^2(\alpha,\beta)$ and $\mathfrak{o}^3(\alpha,\beta)$ are trivial 
(Theorem \ref{uni_Hirsch3}). 
The existence of outer cocycle actions is also discussed 
(Theorem \ref{exi_Hirsch3} and Theorem \ref{exi_Hirsch3'}). 
As corollaries, 
we give a complete classification of outer (cocycle) actions 
of a poly-$\Z$ group of Hirsch length three 
on the Cuntz algebra $\mathcal{O}_n$ 
(Corollary \ref{On} and Corollary \ref{On'}). 
A topological interpretation of 
$\mathfrak{o}^2(\alpha,\beta)$ and $\mathfrak{o}^3(\alpha,\beta)$ 
in terms of $\mathcal{P}_{\alpha^s}$ and $\mathcal{P}_{\beta^s}$ 
is discussed in the companion paper \cite{IMinprep}; 
that is, $\mathfrak{o}^2(\alpha,\beta)$ will be identified with 
the primary obstruction for the existence of a continuous section 
for a certain fiber bundle over $BG$. 
Thus, the results of this paper can be thought of 
as a positive solution of a special case of Conjecture \ref{classification}. 

\bigskip

\textbf{Acknowledgement. }
Masaki Izumi would like to thank 
Isaac Newton Institute for Mathematical Sciences 
for its hospitality. 
Masaki Izumi and Hiroki Matui thank the referees for careful reading and 
for giving several useful comments 
that made this manuscript much more accessible to readers.

\section{Preliminaries}

\subsection{Notation and terminology}

For a Lipschitz continuous map $f$ between metric spaces, 
$\Lip(f)$ denotes the Lipschitz constant of $f$. 

Let $A$ be a $C^*$-algebra. 
For $a,b\in A$, we mean by $[a,b]$ the commutator $ab-ba$. 
When $A$ is a unital $C^*$-algebra, 
we let $U(A)$ denote the set of all unitaries of $A$. 
When $A$ is a non-unital $C^*$-algebra, 
we let $U(A)$ be the set of all unitaries $u$ in the unitization of $A$ 
satisfying $u-1\in A$. 
The connected component of the unit is written as $U(A)_0$. 
When $B$ is a non-unital subalgebra of a unital $C^*$-algebra $A$, 
we always regard $U(B)$ as a subgroup of $U(A)$. 
For $u\in U(A)$, 
the inner automorphism induced by $u$ is written by $\Ad u$. 
An automorphism $\alpha\in\Aut(A)$ is called outer, 
when it is not inner. 

Let $A$, $B$ and $C$ be $C^*$-algebras. 
For a homomorphism $\rho:A\to B$, 
$K_0(\rho)$ and $K_1(\rho)$ mean the induced homomorphisms on $K$-groups, 
and $KK(\rho)$ means the induced element in $KK(A,B)$. 
We write $KK(\id_A)=1_A$. 
For $x\in KK(A,B)$ and $y\in KK(B,C)$, 
we denote the Kasparov product by $y\circ x\in KK(A,C)$. 
(The product is usually written in the opposite way, 
e.g.\ $x\hat{\otimes}y$, 
but we adopt $y\circ x$ to ensure consistency 
with compositions of homomorphisms between $C^*$-algebras.) 
When $p\in A$ is a projection, its $K_0$-class is written 
as $K_0(p)\in K_0(A)$. 
Similarly, when $u\in A$ is a unitary, 
$K_1(u)\in K_1(A)$ is the $K_1$-class of $u$. 
A standard reference for $K$-theory and $KK$-theory of $C^*$-algebras 
is \cite{Bl_Ktheory}. 

A simple purely infinite nuclear separable $C^*$-algebra 
is called a Kirchberg algebra. 
We denote by $\mathcal{O}$ the unital Kirchberg algebra 
which is strongly Morita equivalent to $\mathcal{O}_\infty$ 
and is in the Cuntz standard form, 
i.e.\ the $K_0$-class of the unit is zero in $K_0(\mathcal{O})$. 

Two unital homomorphisms $\rho,\sigma$ from $A$ to $B$ are said to be 
asymptotically unitarily equivalent, 
if there exists a continuous family of unitaries 
$(u_t)_{t\in[0,\infty)}$ in $B$ such that 
\[
\rho(a)=\lim_{t\to\infty}\Ad u_t(\sigma(a))
\]
for all $a\in A$. 
Phillips's theorem \cite[Theorem 4.1.1]{Ph00DocM} says that, 
when $A$ is unital separable nuclear simple and 
$B$ is unital separable $\mathcal{O}_\infty$-stable, 
$\rho,\sigma:A\to B$ are asymptotically unitarily equivalent 
if and only if $KK(\rho)=KK(\sigma)$. 
When there exists a sequence of unitaries $(u_n)_{n\in\N}$ in $B$ 
such that 
\[
\rho(a)=\lim_{n\to\infty}\Ad u_n(\sigma(a))
\]
for all $a\in A$, 
$\rho$ and $\sigma$ are said to be approximately unitarily equivalent. 
An automorphism $\alpha\in\Aut(A)$ is said to be 
asymptotically (resp. approximately) inner 
if $\alpha$ is asymptotically (resp. approximately) 
unitarily equivalent to the identity map. 

Let $A$ be a $C^*$-algebra and 
let $\omega\in\beta\N\setminus\N$ be a free ultrafilter. 
Set 
\[
c_\omega(A)=\{(a_n)\in\ell^\infty(\N,A)\mid
\lim_{n\to\omega}\lVert a_n\rVert=0\},\quad 
A^\omega=\ell^\infty(\N,A)/c_\omega(A). 
\]
We identify $A$ with the $C^*$-subalgebra of $A^\omega$ 
consisting of equivalence classes of constant sequences. 
We let 
\[
A_\omega=A^\omega\cap A'
\]
and call it the central sequence algebra of $A$. 

We also need a continuous analogue of $A^\omega$ and $A_\omega$. 
Let $C^b([0,\infty),A)$ denote the $C^*$-algebra consisting of 
bounded continuous functions $[0,\infty)\to A$. 
Set 
\[
A^\flat=C^b([0,\infty),A)/C_0([0,\infty),A). 
\]
We identify $A$ with the $C^*$-subalgebra of $A^\flat$ 
consisting of equivalence classes of constant functions and let 
\[
A_\flat=A^\flat\cap A'. 
\]
We call it the continuous asymptotic centralizer algebra. 
When $\alpha$ is an automorphism of $A$, 
we can consider its natural extension 
on $A^\omega$, $A_\omega$, $A^\flat$ and $A_\flat$. 
We denote it by the same symbol $\alpha$.

\subsection{Group actions and cocycle conjugacy}

We set up some terminology for group actions. 
For a discrete group $G$, 
the neutral element is denoted by $1\in G$. 

\begin{definition}
Let $A$ be a unital $C^*$-algebra and 
let $G$ be a countable discrete group. 
\begin{enumerate}
\item A pair $(\alpha,u)$ of 
a map $\alpha:G \to\Aut(A)$ and a map $u:G\times G\to U(A)$ 
is called a cocycle action of $G$ on $A$ 
if 
\[
\alpha_g\circ\alpha_h=\Ad u(g,h)\circ\alpha_{gh}
\]
and 
\[
u(g,h)u(gh,k)=\alpha_g(u(h,k))u(g,hk)
\]
hold for any $g,h,k\in G$. 
We always assume $\alpha_1=\id$, $u(g,1)=u(1,g)=1$ for all $g\in G$. 
We denote the cocycle action by $(\alpha,u):G\curvearrowright A$. 
Notice that $\alpha$ gives rise to 
genuine actions of $G$ on $A_\omega$ and $A_\flat$, 
which are denoted by the same symbol $\alpha$. 
When $u(g,h)=1$ for all $g,h\in G$, $\alpha$ becomes a genuine action. 
We write $\alpha:G\curvearrowright A$ 
instead of $(\alpha,1):G\curvearrowright A$. 
\item A cocycle action $(\alpha,u)$ is said to be outer 
if $\alpha_g$ is outer for every $g\in G$ except for the neutral element. 
\item Two cocycle actions $(\alpha,u):G\curvearrowright A$ and 
$(\beta,v):G\curvearrowright B$ are said to be cocycle conjugate 
if there exist a family of unitaries $(w_g)_{g\in G}$ in $B$ and 
an isomorphism $\theta:A\to B$ such that 
\[
\theta\circ\alpha_g\circ\theta^{-1}=\Ad w_g\circ\beta_g
\]
and 
\[
\theta(u(g,h))=w_g\beta_g(w_h)v(g,h)w_{gh}^*
\]
for every $g,h\in G$. 
Furthermore, when there exists a sequence $(x_n)_n$ of unitaries in $B$ 
such that $x_n\beta_g(x_n^*)\to w_g$ as $n\to\infty$ for every $g\in G$, 
we say that $(\alpha,u)$ and $(\beta,v)$ are strongly cocycle conjugate. 
Furthermore, when we can find a continuous family 
$(x_t)_{t\in[0,\infty)}$ of unitaries in $U(B)_0$ with the same property, 
we say that $(\alpha,u)$ and $(\beta,v)$ are 
very strongly cocycle conjugate (see \cite[Definition 2.4]{Sz17Adv}). 
\item Two cocycle actions $(\alpha,u)$ and $(\beta,v)$ of $G$ on $A$ are 
said to be $KK$-trivially cocycle conjugate, 
if they are cocycle conjugate 
via an isomorphism $\theta:A\to A$ such that $KK(\theta)=1_A$. 
\item Let $(\alpha,u):G\curvearrowright A$ be a cocycle action and 
let $(v_g)_g$ be a family of unitaries in $A$. 
The cocycle perturbation $(\alpha^v,u^v)$ of $(\alpha,u)$ by $(v_g)_g$ 
is the cocycle action determined by 
\[
\alpha^v_g=\Ad v_g\circ\alpha_g\quad\text{and}\quad 
u^v(g,h)=v_g\alpha_g(v_h)u(g,h)v_{gh}^*. 
\]
Thus, $(\alpha,u)$ and $(\beta,v)$ are cocycle conjugate if and only if 
a cocycle perturbation of $(\alpha,u)$ and $(\beta,v)$ are conjugate. 
\item Let $\alpha:G\curvearrowright A$ be an action. 
A family of unitaries $(x_g)_{g\in G}$ in $A$ is called 
an $\alpha$-cocycle 
if one has $x_g\alpha_g(x_h)=x_{gh}$ for all $g,h\in G$. 
In this case, $\alpha^x_g=\Ad x_g\circ\alpha_g$ form 
an action $\alpha^x:G\curvearrowright A$. 
\item Let $\alpha:G\curvearrowright A$ be an action. 
The fixed point subalgebra 
$\{a\in A\mid\alpha_g(a)=a\quad\text{for all $g\in G$}\}$ is 
written $A^\alpha$. 
\end{enumerate}
\end{definition}

Let $G$ be a countable discrete group. 
The canonical generators in $C^*_r(G)$ are denoted 
by $(\lambda_g)_{g\in G}$. 
For a cocycle action $(\alpha,u):G\curvearrowright A$, 
the reduced twisted crossed product is written $A\rtimes_{(\alpha,u)}G$. 
We denote the canonical implementing unitaries 
by $(\lambda^\alpha_g)_{g\in G}$. 

\begin{definition}[{\cite[Definition 2.2]{IM10Adv}}]
Let $G$ be a countable discrete group and 
let $A$ be a unital $C^*$-algebra. 
A cocycle action $(\alpha,u):G\curvearrowright A$ is said to be 
approximately representable 
if there exists a family of unitaries $(v_g)_{g\in G}$ in $A^\omega$ 
such that 
\[
v_gv_h=u(g,h)v_{gh},\quad 
\alpha_g(v_h)=u(g,h)u(ghg^{-1},g)^*v_{ghg^{-1}}
\]
and 
\[
v_gav_g^*=\alpha_g(a)
\]
hold for all $g,h\in G$ and $a\in A$. 

Asymptotical representability is defined in an analogous way. 
\end{definition}

It is routine to check that 
approximate (resp. asymptotical) representability is preserved 
under cocycle conjugacy. 

A discrete group $G$ is said to be a poly-$\Z$ group 
if there exists a subnormal series 
\[
\{1\}=G_0\leq G_1\leq G_2\leq\dots\leq G_n=G
\]
such that $G_{i+1}/G_i\cong\Z$. 
The length $n$ of the series does not depend 
on the choice of such a subnormal series 
and is called the Hirsch length of $G$. 
The poly-$\Z$ group of Hirsch length one is $\Z$. 
The poly-$\Z$ group of Hirsch length two is 
either $\Z^2\cong\langle\xi,\zeta\mid\xi\zeta=\zeta\xi\rangle$ or 
the Klein bottle group 
$\langle\xi,\zeta\mid\xi\zeta=\zeta^{-1}\xi\rangle$.

\section{Existence of asymptotically representable actions}

In this section, we prove that 
every countable discrete amenable group admits 
an asymptotically representable outer action on $\mathcal{O}_\infty$ 
(Theorem \ref{exist_asymprepre}). 

\begin{lemma}
For any unitary $u\in\mathcal{O}_2$, 
a finite subset $F\subset\mathcal{O}_2$ and $\ep>0$, 
there exists a continuous path of unitaries $w:[0,1]\to\mathcal{O}_2$ 
such that 
\[
w(0)=1,\quad w(1)=u,\quad \Lip(w)<\frac{8\pi}{3}+\ep
\]
and 
\[
\lVert[w(t),x]\rVert<4\lVert[u,x]\rVert+\ep
\]
for any $x\in F$ and $t\in[0,1]$. 
\end{lemma}
\begin{proof}
This follows immediately from \cite[Lemma 6]{Na00ETDS} 
(which is a slight modification of \cite[Lemma 5.1]{HR95Duke}), 
because $\mathcal{O}_2$ is isomorphic to 
the infinite tensor product of $\mathcal{O}_2$ itself 
(\cite[Corollary 5.2.4]{Ro_text}). 
\end{proof}

\begin{lemma}
Let $G$ be a discrete countable amenable group. 
There exists a unital injective homomorphism 
$\rho:C^*(G)\to\mathcal{O}_\infty$ 
such that the following holds. 
For any finite subset $F\subset G$ and $\ep>0$, 
there exists a continuous path of unitaries $w:[0,1]
\to\mathcal{O}_\infty\otimes\mathcal{O}_\infty\otimes\mathcal{O}_\infty$ 
satisfying 
\[
w(0)=1,\quad \lVert w(1)(\rho(\lambda_g)\otimes1\otimes1)w(1)^*
-\rho(\lambda_g)\otimes\rho(\lambda_g)\otimes1\rVert<\ep\quad \forall g\in F
\]
and 
\[
\lVert[w(t),\rho(\lambda_g)\otimes\rho(\lambda_g)\otimes\rho(\lambda_g)]\rVert
<\ep\quad \forall g\in F,\ t\in[0,1].
\]
\end{lemma}
\begin{proof}
Let $\rho_0:C^*(G)\to\mathcal{O}_2$ be a unital embedding and 
let $\tau:C^*(G)\to\C$ be the unital homomorphism 
such that $\tau(\lambda_g)=1$ for all $g\in G$. 
Let $\iota:\mathcal{O}_2\to\mathcal{O}_\infty$ be a (non-unital) embedding 
and let $p=\iota(1)$. 
Define $\rho:C^*(G)\to\mathcal{O}_\infty$ by 
\[
\rho(x)=\iota(\rho_0(x))+\tau(x)(1-p)
\]
for every $x\in C^*(G)$. 
Define unital homomorphisms $\phi$ and $\psi$ 
from $C^*(G)\otimes C^*(G)$ to 
$\mathcal{O}_\infty\otimes\mathcal{O}_\infty\otimes\mathcal{O}_\infty$ by 
\[
\phi(\lambda_g\otimes\lambda_h)
=\rho(\lambda_g)\otimes\rho(\lambda_h)\otimes\rho(\lambda_h)
\]
and 
\[
\psi(\lambda_g\otimes\lambda_h)
=\rho(\lambda_g)\otimes\rho(\lambda_g)\otimes\rho(\lambda_h). 
\]
Set $q=1\otimes p\otimes1$. 
Clearly we have $[\phi(x),q]=[\psi(x),q]=0$ and $\phi(x)(1-q)=\psi(x)(1-q)$ 
for any $x\in C^*(G)\otimes C^*(G)$. 
The two maps $x\mapsto\phi(x)q$ and $x\mapsto\psi(x)q$ are regarded 
as unital injective homomorphisms from $C^*(G)\otimes C^*(G)$ to 
$\mathcal{O}_\infty\otimes\iota(\mathcal{O}_2)\otimes\mathcal{O}_\infty
\cong\mathcal{O}_2$. 
It follows from \cite[Theorem 6.3.8]{Ro_text} that 
these two homomorphisms are approximately unitarily equivalent. 
Thus there exists a unitary $u_0
\in\mathcal{O}_\infty\otimes\iota(\mathcal{O}_2)\otimes\mathcal{O}_\infty$ 
such that 
\[
\lVert u_0\phi(\lambda_g\otimes\lambda_h)qu_0^*
-\psi(\lambda_g\otimes\lambda_h)q\rVert<\ep/5\quad 
\forall g,h\in F. 
\]
In particular we get 
\[
\lVert[u_0,
(\rho(\lambda_g)\otimes\rho(\lambda_g)\otimes\rho(\lambda_g))q]\rVert<\ep/5
\]
for every $g\in F$. 
The lemma above implies that 
there exists a continuous path of unitaries 
$w_0:[0,1]
\to\mathcal{O}_\infty\otimes\iota(\mathcal{O}_2)\otimes\mathcal{O}_\infty$ 
such that $w_0(0)=q$, $w_0(1)=u_0$ and 
\[
\lVert[w_0(t),
(\rho(\lambda_g)\otimes\rho(\lambda_g)\otimes\rho(\lambda_g))q]\rVert<\ep
\]
for any $g\in F$ and $t\in[0,1]$. 
Define $w:[0,1]
\to\mathcal{O}_\infty\otimes\mathcal{O}_\infty\otimes\mathcal{O}_\infty$ by 
\[
w(t)=w_0(t)+1-q. 
\]
Then 
\begin{align*}
& \lVert w(1)(\rho(\lambda_g)\otimes1\otimes1)w(1)^*
-\rho(\lambda_g)\otimes\rho(\lambda_g)\otimes1\rVert \\
&=\lVert w(1)\phi(\lambda_g\otimes1)w(1)^*
-\psi(\lambda_g\otimes1)\rVert \\
&=\lVert u_0\phi(\lambda_g\otimes1)qu_0^*
-\psi(\lambda_g\otimes1)q\rVert<\ep/5
\end{align*}
holds for every $g\in F$ and 
\[
\lVert[w(t),
\rho(\lambda_g)\otimes\rho(\lambda_g)\otimes\rho(\lambda_g)]\rVert
=\lVert[w_0(t),
(\rho(\lambda_g)\otimes\rho(\lambda_g)\otimes\rho(\lambda_g))q]\rVert<\ep
\]
holds for any $g\in F$ and $t\in[0,1]$. 
The proof is completed. 
\end{proof}

From the lemma above we can deduce the following. 

\begin{theorem}\label{exist_asymprepre}
Let $G$ be a countable discrete amenable group. 
There exists an asymptotically representable outer action 
$\gamma$ of $G$ on $\mathcal{O}_\infty$. 
In particular, for any unital Kirchberg algebra $A$, 
there exists an asymptotically representable outer action of $G$ 
on $A$. 
\end{theorem}
\begin{proof}
Let $\rho:C^*(G)\to\mathcal{O}_\infty$ be the unital homomorphism 
as in the previous lemma. 
For $n\in\N$ and $g\in G$, we define the unitary $u^{(n)}_g$ by 
\[
u^{(n)}_g=\left(\bigotimes_{k=1}^n\rho(\lambda_g)\right)\otimes1
\in\bigotimes_{k=1}^\infty\mathcal{O}_\infty. 
\]
Define $\gamma:G\curvearrowright 
\mathcal{O}_\infty\cong\bigotimes_{k=1}^\infty\mathcal{O}_\infty$ by 
\[
\gamma_g=\lim_{n\to\infty}\Ad u^{(n)}_g. 
\]
Clearly, $\gamma$ is outer and approximately representable. 
It suffices to connect $u^{(n)}_g$ to $u^{(n+1)}_g$ 
by an appropriate path of unitaries. 

Let $\pi_n$ be the unital embedding of $\mathcal{O}_\infty$ 
into the $n$-th tensor product component of 
$\bigotimes_{k=1}^\infty\mathcal{O}_\infty$, 
and let 
\[
\tilde\pi_n=\pi_n\otimes\pi_{n+1}\otimes\pi_{n+2}:
\mathcal{O}_\infty\otimes\mathcal{O}_\infty\otimes\mathcal{O}_\infty
\to\bigotimes_{k=1}^\infty\mathcal{O}_\infty. 
\]
Let $(F_n)_n$ be an increasing sequence of finite subsets of $G$ 
such that $G=\bigcup_nF_n$. 
By the previous lemma, 
there exists a continuous path of unitaries $w_n:[0,1]
\to\mathcal{O}_\infty\otimes\mathcal{O}_\infty\otimes\mathcal{O}_\infty$ 
satisfying 
\[
w_n(0)=1,\quad \lVert w_n(1)(\rho(\lambda_g)\otimes1\otimes1)w_n(1)^*
-\rho(\lambda_g)\otimes\rho(\lambda_g)\otimes1\rVert<1/n\quad \forall g\in F_n
\]
and 
\[
\lVert[w_n(t),\rho(\lambda_g)\otimes\rho(\lambda_g)\otimes\rho(\lambda_g)]\rVert
<1/n\quad \forall g\in F_n,\ t\in[0,1].
\]
Define $u_g:[n,n{+}1]\to\bigotimes_{k=1}^\infty\mathcal{O}_\infty$ by 
\[
u_g(s)=u^{(n-1)}_g
\tilde\pi_n\left(w_n(s{-}n)(\rho(\lambda_g)\otimes1\otimes1)w_n(s{-}n)^*\right). 
\]
Then, we obtain 
\[
u_g(n)=u^{(n)}_g\quad\forall g\in G\quad, 
\lVert u_g(n{+}1)-u^{(n+1)}_g\rVert<1/n\quad\forall g\in F_n
\]
and 
\[
\lVert\gamma_g(u_h(s))-u_{ghg^{-1}}(s)\rVert<2/n
\quad\forall g\in F_n,\ h\in G,\ s\in[n,n{+}1]. 
\]
By a suitable perturbation, we may assume $u_g(n{+}1)=u^{(n{+}1)}_g$. 
Hence $u_g$ determines a continuous map from $[1,\infty)$. 
It is easy to see $\gamma_g=\lim_s\Ad u_g(s)$, 
which implies that $\gamma$ is asymptotically representable. 

When $A$ is a unital Kirchberg algebra, 
$\id\otimes\gamma:G\curvearrowright A\otimes\mathcal{O}_\infty\cong A$ is 
asymptotically representable and outer, as desired. 
\end{proof}

\begin{remark}\label{modelaction}
The action $\gamma:G\curvearrowright\mathcal{O}_\infty$ constructed 
in the theorem above is exactly the same as the actions 
discussed by G. Szab\'o \cite{Sz18CMP} 
(see also \cite[Section 6]{GI11Tohoku} for finite groups). 
Szab\'o proved that 
the strong cocycle conjugacy class of $\gamma$ does not depend 
on the choice of the unitary representations 
$G\ni g\mapsto\rho_0(\lambda_g)\in U(\mathcal{O}_2)$ 
(\cite[Corollary 3.8]{Sz18CMP}). 
Moreover, he proved that 
any outer cocycle action of $G$ on a Kirchberg algebra absorbs 
$\gamma$ tensorially (\cite[Corollary 3.7]{Sz18CMP}). 
\end{remark}

\begin{remark}
Y. Arano and Y. Kubota kindly informed us the following. 
Let $G$ be a countable discrete torsion-free amenable group 
and let $A$ be a unital Kirchberg algebra. 
Then any outer asymptotically representable actions $G\curvearrowright A$ 
are mutually cocycle conjugate. 
See \cite{AK17JFA}. 
\end{remark}

\section{Uniqueness of actions on $\mathcal{O}_\infty$}

In this section, we prove that outer actions of poly-$\Z$ groups 
on $\mathcal{O}_2$, $\mathcal{O}_\infty$ and $\mathcal{O}_\infty\otimes B$ 
with $B$ being a UHF algebra of infinite type are 
unique up to ($KK$-trivial) cocycle conjugacy. 
(We note that during the preparation of this paper 
G. Szab\'o has proved much stronger statement. 
See Remark \ref{workofSzabo}.) 

First, let us recall the notion of dual coactions. 
Let $\alpha:G\curvearrowright A$ be 
an action of countable discrete amenable group $G$ 
on a unital $C^*$-algebra $A$. 
The dual coaction $\hat\alpha$ is the homomorphism 
$A\rtimes_\alpha G\to(A\rtimes_\alpha G)\otimes C^*(G)$ defined by 
\[
\hat\alpha(a)=a\otimes1\quad\text{and}\quad 
\hat\alpha(\lambda^\alpha_g)=\lambda^\alpha_g\otimes\lambda_g
\quad\forall a\in A,\ g\in G. 
\]
Let $\beta:G\curvearrowright B$ be another action. 
We say that 
a homomorphism $\rho:A\rtimes_\alpha G\to B\rtimes_\beta G$ commutes 
with the dual coactions if 
\[
\hat\beta\circ\rho=(\rho\otimes\id)\circ\hat\alpha
\]
holds. 
This is equivalent to saying that 
$\rho(\lambda^\alpha_g)\in B\lambda^\beta_g$ holds for all $g\in G$. 
It is well-known (and easy to see) that 
$\alpha$ is cocycle conjugate to $\beta$ if and only if 
there exists an isomorphism $A\rtimes_\alpha G\to B\rtimes_\beta G$ 
which commutes with the dual coactions. 

The coproduct $\delta_G:C^*(G)\to C^*(G)\otimes C^*(G)$ is 
given by $\delta_G(\lambda_g)=\lambda_g\otimes\lambda_g$ for $g\in G$. 
The coproduct is the dual coaction of 
the trivial action $G\curvearrowright\C$.

\subsection{Equivariant version of Nakamura's theorem}

Nakamura proved that 
if two outer actions $\alpha,\beta:\Z\curvearrowright A$ 
on a unital Kirchberg algebra satisfy $KK(\alpha_g)=KK(\beta_g)$ 
for all $g\in\Z$, 
then $\alpha$ and $\beta$ are $KK$-trivially cocycle conjugate 
(\cite[Theorem 5]{Na00ETDS}). 
In \cite[Theorem 4.11]{IM10Adv}, 
we gave an equivariant version of Nakamura's result. 
In this subsection, we would like to improve the statement 
so that it applies to actions of poly-$\Z$ groups. 

Throughout this subsection, 
we let $G$ be a countable discrete amenable group and 
let $N\subset G$ be a normal subgroup such that $G/N\cong\Z$. 
Choose and fix an element $\xi\in G$ 
so that $G$ is generated by $N$ and $\xi$. 

The following proposition is an equivariant version of 
\cite[Theorem 1]{Na00ETDS}. 

\begin{proposition}\label{equivRohlin}
Let $\alpha:G\curvearrowright A$ be an outer action 
on a unital Kirchberg algebra. 
Suppose that 
$\alpha|N:N\curvearrowright A$ is approximately representable. 
Then for any $M\in\N$, there exist projections 
$e_0,e_1,\dots,e_{M-1}$, $f_0,f_1,\dots,f_M$ in $(A_\omega)^{\alpha|N}$ 
such that 
\[
\sum_{i=0}^{M-1}e_i+\sum_{j=0}^Mf_j=1,\quad 
\alpha_\xi(e_i)=e_{i+1}\quad\text{and}\quad \alpha_\xi(f_j)=f_{j+1}
\]
for all $i=0,1,\dots,M{-}1$ and $j=0,1,\dots,M$, 
where $e_M$ and $f_{M+1}$ mean $e_0$ and $f_0$, respectively. 
\end{proposition}
\begin{proof}
Since $\alpha|N:N\curvearrowright A$ is approximately representable, 
by \cite[Corollary 3.2 (2)]{IM10Adv}, 
$(A_\omega)^{\alpha|N}$ is purely infinite and simple. 
(Actually, the assumption of approximate representability 
is not necessary here. See \cite[Proposition 3.3]{Sz18CMP}.) 
By \cite[Proposition 3.5]{IM10Adv}, 
the restriction of $\alpha_{\xi^n}$ to $(A_\omega)^{\alpha|N}$ is 
a non-trivial automorphism for every $n\neq0$, 
and hence it is outer thanks to \cite[Lemma 2]{Na00ETDS}. 
This means that 
we can choose a Rohlin tower for $\alpha_\xi$ in $(A_\omega)^{\alpha|N}$. 
We omit the detail, 
because the argument is exactly the same as \cite[Theorem 1]{Na00ETDS}. 
\end{proof}

Let $A$ be a unital Kirchberg algebra and 
let $\alpha,\beta:G\curvearrowright A$ be two outer actions. 
We assume the following. 
\begin{itemize}
\item $\alpha|N:N\curvearrowright A$ and $\beta|N:N\curvearrowright A$ are 
asymptotically representable. 
\item $\alpha|N$ is $KK$-trivially cocycle conjugate to $\beta|N$. 
\end{itemize}
Define $\tilde\alpha_\xi\in\Aut(A\rtimes_{\alpha|N}N)$ by 
\[
\tilde\alpha_\xi(a)=\alpha_\xi(a)\quad\text{and}\quad 
\tilde\alpha_\xi(\lambda^\alpha_g)=\lambda^\alpha_{\xi g\xi^{-1}}\quad 
\forall a\in A,\ g\in N. 
\]
We have 
\[
\widehat{\alpha|N}\circ\tilde\alpha_\xi
=(\tilde\alpha_\xi\otimes\Ad\lambda_\xi)\circ\widehat{\alpha|N}. 
\]
It is easy to see that 
the crossed product of $A\rtimes_{\alpha|N}N$ 
by the automorphism $\tilde\alpha_\xi$ is naturally 
isomorphic to $A\rtimes_\alpha G$. 
Similarly, the automorphism $\beta_\xi$ extends to 
$\tilde\beta_\xi\in\Aut(A\rtimes_{\beta|N}N)$. 

The following is the equivariant version of Nakamura's theorem. 

\begin{theorem}\label{equivNakamura}
Let $A$ be a unital Kirchberg algebra. 
Suppose that 
$\alpha,\beta:G\curvearrowright A$ satisfy the assumptions stated above. 
Suppose that there exists an isomorphism 
$\theta:A\rtimes_{\beta|N}N\to A\rtimes_{\alpha|N}N$ 
which commutes with the dual coactions. 
If $KK(\theta|A)=1_A$ and 
$KK(\tilde\alpha_\xi)=KK(\theta\circ\tilde\beta_\xi\circ\theta^{-1})$, 
then $\alpha:G\curvearrowright A$ is 
$KK$-trivially cocycle conjugate to $\beta:G\curvearrowright A$. 
\end{theorem}
\begin{proof}
It is easy to see that 
$\tilde\alpha_\xi\circ(\theta\circ\tilde\beta_\xi\circ\theta^{-1})^{-1}$ 
commutes with the dual coactions. 
Indeed, 
\begin{align*}
\widehat{\alpha|N}
\circ\tilde\alpha_\xi\circ(\theta\circ\tilde\beta_\xi\circ\theta^{-1})^{-1}
&=(\tilde\alpha_\xi\otimes\Ad\lambda_\xi)
\circ\widehat{\alpha|N}\circ\theta\circ\tilde\beta_\xi^{-1}\circ\theta^{-1}\\
&=(\tilde\alpha_\xi\otimes\Ad\lambda_\xi)\circ(\theta\otimes\id)
\circ\widehat{\beta|N}\circ\tilde\beta_\xi^{-1}\circ\theta^{-1}\\
&=(\tilde\alpha_\xi\otimes\Ad\lambda_\xi)\circ(\theta\otimes\id)
\circ(\tilde\beta_\xi^{-1}\otimes\Ad\lambda_\xi^*)
\circ\widehat{\beta|N}\circ\theta^{-1}\\
&=(\tilde\alpha_\xi\otimes\Ad\lambda_\xi)\circ(\theta\otimes\id)
\circ(\tilde\beta_\xi^{-1}\otimes\Ad\lambda_\xi^*)\circ(\theta^{-1}\otimes\id)
\circ\widehat{\alpha|N}\\
&=\bigl((\tilde\alpha_\xi\circ\theta\circ\tilde\beta_\xi^{-1}\circ\theta^{-1})
\otimes\id\bigr)\circ\widehat{\alpha|N}. 
\end{align*}
By assumption, the $KK$-class of 
$\tilde\alpha_\xi\circ(\theta\circ\tilde\beta_\xi\circ\theta^{-1})^{-1}$ is 
trivial, and so it is asymptotically inner by Phillips's theorem. 
Since $\alpha|N$ is asymptotically representable, 
it follows from \cite[Theorem 4.8]{IM10Adv} that 
$\tilde\alpha_\xi\circ(\theta\circ\tilde\beta_\xi\circ\theta^{-1})^{-1}$ is 
asymptotically inner via unitaries of $A$. 
Hence there exists $u:[0,\infty)\to U(A)$ such that 
\[
\lim_{t\to\infty}\Ad u(t)\circ\tilde\alpha_\xi
=\theta\circ\tilde\beta_\xi\circ\theta^{-1}
\]
on $A\rtimes_{\alpha|N}N$. 

On the other hand, by the proposition above, 
the automorphism $\tilde\alpha_\xi$ admits Rohlin projections 
in $(A_\omega)^{\alpha|N}\subset(A\rtimes_{\alpha|N}N)_\omega$. 
The same is true for $\tilde\beta_\xi$. 
Hence, 
by using \cite[Lemma 3.3]{IM10Adv} instead of \cite[Theorem 7]{Na00ETDS}, 
the usual intertwining argument shows that 
there exist an isomorphism 
$\theta':A\rtimes_{\beta|N}N\to A\rtimes_{\alpha|N}N$ 
which commutes with the dual coactions 
and a unitary $v\in U(A)$ such that 
\[
\Ad v\circ\tilde\alpha_\xi
=\theta'\circ\tilde\beta_\xi\circ(\theta')^{-1}
\]
holds true. 
It follows that we can define an isomorphism $\theta''$ 
from 
\[
(A\rtimes_{\beta|N}N)\rtimes_{\tilde\beta_\xi}\Z=A\rtimes_\beta G
\]
to 
\[
(A\rtimes_{\alpha|N}N)\rtimes_{\tilde\alpha_\xi}\Z=A\rtimes_\alpha G
\]
by letting $\theta''(x)=\theta'(x)$ for $x\in A\rtimes_{\beta|N}N$ 
and $\theta''(\lambda^\beta_\xi)=v\lambda^\alpha_\xi$. 
Clearly, $\theta''(\lambda^\beta_g)$ is in $A\lambda^\alpha_g$ 
for every $g\in N\cup\{\xi\}$, and hence for every $g\in G$. 
Therefore, $\alpha:G\curvearrowright A$ is 
cocycle conjugate to $\beta:G\curvearrowright A$ 
via the automorphism $\theta''|A$. 
By the construction, $\theta''|A=\theta'|A$ and 
$\theta'|A$ is asymptotically unitarily equivalent to $\theta|A$. 
Thus, $\alpha$ is $KK$-trivially cocycle conjugate to $\beta$. 
\end{proof}

\subsection{Uniqueness of poly-$\Z$ group actions}

In this subsection, 
we prove the uniqueness of outer actions of poly-$\Z$ groups 
on unital strongly self-absorbing Kirchberg algebras 
satisfying the UCT (Theorem \ref{Oinfty_unique}). 

\begin{lemma}
Let $G$ be a poly-$\Z$ group and 
let $\alpha:G\curvearrowright A$ be an action of $G$ 
on a unital $C^*$-algebra $A$ satisfying the UCT. 
Suppose that 
the inclusion map $a\mapsto a\otimes1$ from $A$ to $A\otimes A$ 
induces a $KK$-equivalence. 
\begin{enumerate}
\item The canonical embedding 
$\rho:A\otimes C^*(G)\to A\otimes(A\rtimes_\alpha G)$ 
sending $a\otimes\lambda_g$ to $a\otimes\lambda^\alpha_g$ 
induces a $KK$-equivalence. 
\item The canonical embedding 
$\sigma:A\rtimes_\alpha G\to A\otimes(A\rtimes_\alpha G)$ 
sending $x$ to $1_A\otimes x$ 
induces a $KK$-equivalence. 
\end{enumerate}
\end{lemma}
\begin{proof}
(1) 
The proof is by induction on the Hirsch length of $G$. 
When $G$ is trivial, the assertion is clear 
because we assumed that 
the inclusion map $a\mapsto a\otimes1$ induces a $KK$-equivalence. 
For a given poly-$\Z$ group $G$, 
let $N$ be a normal poly-$\Z$ subgroup of $G$ such that 
$G/N\cong\Z$. 
Let 
\[
\rho':A\otimes C^*(N)\to A\otimes(A\rtimes_{\alpha|N}N). 
\]
be the canonical unital inclusion. 
By the induction hypothesis $KK(\rho')$ is invertible. 
Choose $\xi\in G$ which generates $G/N\cong\Z$. 
Define $\tilde\alpha_\xi\in\Aut(A\rtimes_{\alpha|N}N)$ 
in the same way as in the previous subsection. 
Clearly $(A\rtimes_{\alpha|N}N)\rtimes_{\tilde\alpha_\xi}\Z$ is 
isomorphic to $A\rtimes_\alpha G$ 
and $C^*(N)\rtimes_{\Ad\lambda_\xi}\Z$ is isomorphic to $C^*(G)$. 
By the naturality of the Pimsner-Voiculescu exact sequence, 
we obtain the following commutative diagram, 
in which the vertical sequences are exact. 
\[
\begin{CD}
K_i(A\otimes C^*(N))@>K_i(\rho')>>
K_i(A\otimes(A\rtimes_{\alpha|N}N)) \\
@V\id-K_i(\id\otimes\Ad\lambda_\xi)VV@VV\id-K_i(\id\otimes\tilde\alpha_\xi)V \\
K_i(A\otimes C^*(N))@>K_i(\rho')>>
K_i(A\otimes(A\rtimes_{\alpha|N}N)) \\
@VVV@VVV \\
K_i(A\otimes C^*(G))@>K_i(\rho)>>
K_i(A\otimes(A\rtimes_\alpha G)) \\
@VVV@VVV \\
K_{1-i}(A\otimes C^*(N))@>K_{1-i}(\rho')>>
K_{1-i}(A\otimes(A\rtimes_{\alpha|N}N)) \\
@V\id-K_{1-i}(\id\otimes\Ad\lambda_\xi)VV
@VV\id-K_{1-i}(\id\otimes\tilde\alpha_\xi)V \\
K_{1-i}(A\otimes C^*(N))@>K_{1-i}(\rho')>>
K_{1-i}(A\otimes(A\rtimes_{\alpha|N}N)) \\
\end{CD}
\]
The five lemma implies that $K_i(\rho)$ is an isomorphism for $i=0,1$. 
Hence $KK(\rho)$ is invertible. 

(2) 
In the same way as (1) we can show that $KK(\sigma)$ is invertible. 
\end{proof}

\begin{theorem}\label{Oinfty_unique}
Let $G$ be a poly-$\Z$ group and 
let $A$ be a unital strongly self-absorbing Kirchberg algebra 
satisfying the UCT. 
Then, outer actions of $G$ on $A$ are unique 
up to $KK$-trivial cocycle conjugacy. 
In particular, any outer actions of $G$ on $A$ are 
asymptotically representable. 
\end{theorem}
\begin{proof}
Let $A$ be a unital strongly self-absorbing Kirchberg algebra 
satisfying the UCT. 
The proof is by induction on the Hirsch length of $G$. 
When $G$ is trivial, the assertion is clear. 

Let $\alpha:G\curvearrowright A$ be 
an outer action of a poly-$\Z$ group $G$ on $A$. 
Let $N\subset G$ be a normal poly-$\Z$ subgroup such that $G/N\cong\Z$. 
Choose and fix an element $\xi\in G$ 
so that $G$ is generated by $N$ and $\xi$. 
Let $\tilde\alpha_\xi$ be as in the previous subsection. 
By the induction hypothesis, 
$\alpha|N:N\curvearrowright A$ is asymptotically representable. 
By the lemma above, the canonical embeddings 
$\rho:A\otimes C^*(N)\to A\otimes(A\rtimes_{\alpha|N}N)$ 
and $\sigma:A\rtimes_{\alpha|N}N\to A\otimes(A\rtimes_{\alpha|N}N)$ 
induce $KK$-equivalences. 
Define 
\[
x=KK(\sigma)^{-1}\circ KK(\rho)\in KK(A\otimes C^*(N),A\rtimes_{\alpha|N}N). 
\]
Clearly $x$ is invertible and $K_0(x)$ is unital. 
Since 
$(\id\otimes\tilde\alpha_\xi)\circ\rho=\rho\circ(\id\otimes\Ad\lambda_\xi)$ 
and $(\id\otimes\tilde\alpha_\xi)\circ\sigma=\sigma\circ\tilde\alpha_\xi$, 
we have 
\begin{equation}
x\circ KK(\id\otimes\Ad\lambda_\xi)=KK(\tilde\alpha_\xi)\circ x. 
\label{eq:x1}
\end{equation}
We also have 
\begin{equation}
(x\otimes1_{C^*(N)})\circ KK(\id\otimes\delta_N)
=KK(\widehat{\alpha|N})\circ x, 
\label{eq:x2}
\end{equation}
because $(\id\otimes\widehat{\alpha|N})\circ\rho
=(\rho\otimes\id)\circ(\id\otimes\delta_N)$ and 
$(\id\otimes\widehat{\alpha|N})\circ\sigma
=(\sigma\otimes\id)\circ\widehat{\alpha|N}$. 

Let $\beta:G\curvearrowright A$ be another outer action. 
In the same way, we obtain 
an invertible element $y\in KK(A\otimes C^*(N),A\rtimes_{\beta|N}N)$ 
satisfying 
\begin{equation}
y\circ KK(\id\otimes\Ad\lambda_\xi)=KK(\tilde\beta_\xi)\circ y
\label{eq:y1}
\end{equation}
and 
\begin{equation}
(y\otimes1_{C^*(N)})\circ KK(\id\otimes\delta_N)=KK(\widehat{\beta|N})\circ y. 
\label{eq:y2}
\end{equation}
There exists an isomorphism $A\rtimes_{\beta|N}N\to A\rtimes_{\alpha|N}N$ 
whose $KK$-class is equal to $x\circ y^{-1}$. 
It follows from \eqref{eq:x2} and \eqref{eq:y2} that 
$x\circ y^{-1}$ commutes with the dual coactions in the category of $KK$. 
Hence, by \cite[Remark 4.6]{IM10Adv}, there exists an isomorphism 
$\theta:A\rtimes_{\beta|N}N\to A\rtimes_{\alpha|N}N$ 
which commutes with the dual coactions and 
whose $KK$-class is equal to $x\circ y^{-1}$. 
Also, \eqref{eq:x1} and \eqref{eq:y1} imply 
\[
KK(\theta\circ\tilde\beta_\xi\circ\theta^{-1})
=x\circ y^{-1}\circ KK(\tilde\beta_\xi)\circ y\cdot x^{-1}
=KK(\tilde\alpha_\xi). 
\]
Therefore, thanks to Theorem \ref{equivNakamura}, 
we can conclude that $\alpha:G\curvearrowright A$ is 
$KK$-trivially cocycle conjugate to $\beta:G\curvearrowright A$. 

It follows from Theorem \ref{exist_asymprepre} that 
any outer action $G\curvearrowright A$ is asymptotically representable. 
\end{proof}

\begin{remark}\label{Oinfty_unique'}
In the theorem above, we can strengthen the conclusion a bit more as follows. 
Let $\alpha,\beta:G\curvearrowright A$ be outer actions. 
Applying the first part of the proof to $\alpha$ and $\beta$, 
we get an isomorphism $\theta:A\rtimes_\beta G\to A\rtimes_\alpha G$ 
which commutes with the dual coactions. 
Letting $u_g=\theta(\lambda^\beta_g)(\lambda^\alpha_g)^*$, 
one has 
\[
\Ad u_g\circ\alpha_g=\theta\circ\beta_g\circ\theta^{-1}\quad\forall g\in G
\]
on $A$. 
From the construction of $\theta$, 
it is easy to see $KK(\theta\circ\iota_\beta)=KK(\iota_\alpha)$, 
where $\iota_\alpha$ and $\iota_\beta$ are the canonical embeddings of 
$C^*(G)$ into $A\rtimes_\alpha G$ and $A\rtimes_\beta G$. 
By virtue of Kirchberg's theorem \cite[Theorem 8.3.3]{Ro_text}, 
$\theta\circ\iota_\beta$ and $\iota_\alpha$ are 
asymptotically unitarily equivalent. 
Hence, by \cite[Theorem 4.8]{IM10Adv}, 
there exists a continuous path of unitaries $(v_t)_{t\in[0,\infty)}$ 
in $U(A)$ such that 
\[
\lim_{t\to\infty}\Ad v_t(\iota_\alpha(a))=\theta(\iota_\beta(a))\quad 
\forall a\in C^*(G). 
\]
Therefore one obtains 
\[
\lim_{t\to\infty}v_t\alpha_g(v_t^*)=u_g\quad 
\forall g\in G. 
\]
Thus, $\alpha$ and $\beta$ are very strongly cocycle conjugate 
(note that $K_1(A)$ is trivial). 
In particular, $\alpha$ is homotopic to $\beta$ as $G$-actions. 
\end{remark}

\begin{remark}
For $A=\mathcal{O}_\infty$, 
the proof of Theorem \ref{Oinfty_unique} can be slightly simplified 
in the following manner. 
One can show that the inclusion $C^*(G)\to A\rtimes_\alpha G$ 
induces a $KK$-equivalence by using induction on the Hirsch length of $G$. 
(Actually, this holds for any torsion-free amenable group $G$, 
see \cite[Corollary 4.18]{Sz18CMP}.) 
Instead of $x=KK(\sigma)^{-1}\circ KK(\rho)$, we let $x$ be 
the $KK$-class of the inclusion map $C^*(N)\to A\rtimes_{\alpha|N}N$. 
Then, the rest of the proof works exactly in the same way. 
\end{remark}

\begin{remark}\label{workofSzabo}
During the preparation of this paper, 
G. Szab\'o \cite{Sz1807arXiv} has obtained the following result, 
which is stronger than Theorem \ref{Oinfty_unique}. 
Let $\mathfrak{C}$ be a bootstrap class of countable discrete groups 
(see \cite{Sz1807arXiv} for its precise definition). 
For any $G\in\mathfrak{C}$ and 
any strongly self-absorbing $C^*$-algebra $\mathcal{D}$ 
(not necessarily satisfying the UCT), 
strongly outer actions $G\curvearrowright\mathcal{D}$ are 
unique up to very strong cocycle conjugacy. 
The class $\mathfrak{C}$ contains 
all torsion-free abelian groups and poly-$\Z$ groups. 
Besides, 
he also proved that any outer action of $G\in\mathfrak{C}$ 
on a unital Kirchberg algebra $A$ absorbs 
every $G$-action on $\mathcal{O}_\infty$. 
This is a strengthened version of Theorem \ref{Oinfty_absorb}. 
\end{remark}

As a corollary of Theorem \ref{Oinfty_unique}, 
we get the following. 

\begin{corollary}\label{embedtoAflat}
Let $G$ be a poly-$\Z$ group and 
let $A$ be a unital strongly self-absorbing Kirchberg algebra 
satisfying the UCT. 
For any outer action $\alpha:G\curvearrowright A$, 
there exists a unital homomorphism $\phi:A\to A_\flat$ 
such that $\phi\circ\alpha_g=\alpha_g\circ\phi$ for all $g\in G$. 
\end{corollary}
\begin{proof}
Let $\rho\in\Aut(A\otimes A)$ be the flip automorphism, 
i.e.\ $\rho(a\otimes b)=b\otimes a$. 
Since $(A,\alpha)$ is cocycle conjugate to 
$(B,\beta)=(\bigotimes_\N A,\bigotimes_\N\alpha)$, 
it suffices to construct a unital homomorphism $\phi:A\to B_\flat$ 
such that $\phi\circ\alpha_g=\beta_g\circ\phi$. 
To this end, it suffices to show the following: 
for any finite subset $F\subset G$ and $\ep>0$, 
there exists a homotopy $(\rho_t)_{t\in[0,1]}$ in $\Aut(A\otimes A)$ 
such that $\rho_0=\id$, $\rho_1=\rho$ and 
\[
\lVert\rho_t\circ(\alpha_g\otimes\alpha_g)
-(\alpha_g\otimes\alpha_g)\circ\rho_t\rVert<\ep
\quad\forall t\in[0,1],\ g\in F. 
\]
By means of Theorem \ref{Oinfty_unique} and Remark \ref{Oinfty_unique'}, 
we may perturb $\alpha$ a little bit and assume that 
$\alpha$ extends to an outer action $\tilde\alpha$ of $G\times\R$. 
Define an outer action $\gamma:G\times\Z\curvearrowright A\otimes A$ by 
\[
\gamma_{(g,n)}=\rho^n\circ(\tilde\alpha_{(g,n)}\otimes\tilde\alpha_{(g,n)}). 
\]
By Theorem \ref{Oinfty_unique} and Remark \ref{Oinfty_unique'}, 
$\tilde\alpha|(G\times\Z)$ and $\gamma$ are cocycle conjugate 
with a cocycle arbitrarily close to $1$. 
Thus there exists a homotopy between $\id$ and 
$\gamma_{(e,1)}=\rho\circ(\tilde\alpha_{(e,1)}\otimes\tilde\alpha_{(e,1)})$ 
almost commuting with $\alpha_g\otimes\alpha_g$ for $g\in G$. 
Since $(\tilde\alpha_{(e,t)}\otimes\tilde\alpha_{(e,t)})_t$ gives 
a homotopy between $\id$ and $\tilde\alpha_{(e,1)}\otimes\tilde\alpha_{(e,1)}$ 
commuting with $\alpha_g\otimes\alpha_g$, 
we get $(\rho_t)_t$ as above. 
\end{proof}

When $A=\mathcal{O}_2$, 
the crossed product $A\rtimes_\alpha G$ is again $\mathcal{O}_2$ 
for any outer action $\alpha$ of a poly-$\Z$ group $G$. 
Hence the following stronger statement holds. 

\begin{theorem}\label{O2_unique}
Let $G$ be a poly-$\Z$ group. 
Any outer cocycle actions of $G$ on $\mathcal{O}_2$ are 
mutually cocycle conjugate. 
\end{theorem}
\begin{proof}
The proof is almost the same as Theorem \ref{Oinfty_unique}. 
Let $A=\mathcal{O}_2$ and 
let $(\alpha,u):G\curvearrowright A$ be an outer cocycle action of 
a poly-$\Z$ group $G$. 
Let $N\subset G$ and $\xi\in G$ be 
as in the proof of Theorem \ref{Oinfty_unique}. 
By induction, we may assume that 
the restriction of $(\alpha,u)$ to $N$ is cocycle conjugate to 
a genuine action, and so 
we may perturb $(\alpha,u)$ and assume 
\[
u(g,h)=1\quad\text{and}\quad u(g\xi^l,\xi^m)=1
\quad\forall g,h\in N,\ l,m\in\Z. 
\]
Define $\tilde\alpha_\xi\in\Aut(A\rtimes_{\alpha|N}N)$ by 
\[
\tilde\alpha_\xi(a)=\alpha_\xi(a)\quad\text{and}\quad 
\tilde\alpha_\xi(\lambda^\alpha_g)=u(\xi,g)\lambda^\alpha_{\xi g\xi^{-1}}\quad 
\forall a\in A,\ g\in N. 
\]
Let $(\beta,v):G\curvearrowright A$ be another outer cocycle action. 
Similarly, one can define 
$\tilde\beta_\xi\in A\rtimes_{\beta|N}N$. 
Since $\alpha|N$ is cocycle conjugate to $\beta|N$, 
there exists an isomorphism 
$\theta:A\rtimes_{\beta|N}N\to A\rtimes_{\alpha|N}N$ 
which commutes with the dual coactions. 
Clearly $KK(\tilde\alpha_\xi)=KK(\theta\circ\tilde\beta_\xi\circ\theta^{-1})$, 
because the crossed product $A\rtimes_{\alpha|N}N$ is $\mathcal{O}_2$. 
Theorem \ref{equivNakamura} applies and yields the desired conclusion. 
\end{proof}

\begin{remark}
Outer cocycle actions of a poly-$\Z$ group $G$ 
on $\mathcal{O}_\infty$ is not unique in general. 
More precisely, when the Hirsch length of $G$ is less than three, 
any outer cocycle action $G\curvearrowright\mathcal{O}_\infty$ is 
cocycle conjugate to a genuine action (see Lemma \ref{2cv_Hirsch2}), 
and hence is unique by Theorem \ref{Oinfty_unique}. 
However, there exist 
infinitely many outer cocycle actions $\Z^3\curvearrowright\mathcal{O}_\infty$ 
which are not mutually cocycle conjugate 
(see Example \ref{On''}). 
\end{remark}

\subsection{Absorption of $\mathcal{O}_\infty$}

In this subsection, we prove that 
any outer cocycle action of a poly-$\Z$ group on a unital Kirchberg algebra 
absorbs the outer action on $\mathcal{O}_\infty$ 
(Theorem \ref{Oinfty_absorb}). 
We begin with a McDuff type theorem for discrete group actions. 
The same argument can be found in \cite[Theorem 6.3]{IM10Adv}, 
\cite[Theorem 2.3]{GI11Tohoku} and \cite[Theorem 4.9]{MS14AJM}. 
We note that 
G. Szab\'o \cite{Sz18TAMS,Sz17Adv} gave a strengthened version 
which can be applied to actions of locally compact groups. 

\begin{theorem}\label{McDuff}
Let $G$ be a countable discrete group and 
let $A,B$ be unital separable $C^*$-algebras. 
Let $(\alpha,u):G\curvearrowright A$ be a cocycle action and 
let $\beta:G\curvearrowright B$ be an action. 
Suppose that the following conditions hold. 
\begin{enumerate}
\item There exists a unital homomorphism $\pi:B\to A_\omega$ 
such that $\pi\circ\beta_g=\alpha_g\circ\pi$ for all $g\in G$. 
\item There exists a sequence $(v_n)_n$ of unitaries in $U(B\otimes B)_0$ 
such that 
\[
\lim_{n\to\infty}v_n(b\otimes1)v_n^*=1\otimes b
\quad\forall b\in B,\quad 
\lim_{n\to\infty}(\beta_g\otimes\beta_g)(v_n)-v_n=0
\quad\forall g\in G. 
\]
\end{enumerate}
Then $(\alpha,u)$ is cocycle conjugate to 
$(\alpha\otimes\beta,u\otimes1):G\curvearrowright A\otimes B$ 
via an isomorphism $\psi:A\to A\otimes B$ 
which is asymptotically unitarily equivalent to 
the embedding $A\ni a\mapsto a\otimes1\in A\otimes B$. 
\end{theorem}
\begin{proof}
In a similar fashion to \cite[Theorem 7.2.2]{Ro_text}, 
we can find a sequence $(w_n)_n$ of unitaries in $A\otimes B$ 
such that the following are satisfied. 
\begin{itemize}
\item There exist continuous maps $\tilde w_n:[0,1]\to U(A\otimes B)$ 
such that $\tilde w_n(0)=1$, $\tilde w_n(1)=w_n$ and 
\[
\lim_{n\to\infty}\sup_{t\in[0,1]}\lVert[\tilde w_n(t),a\otimes1]\rVert=0
\quad\forall a\in A. 
\]
\item The distance from $w_n^*cw_n$ to $A\otimes\C$ tends to zero 
as $n\to\infty$ for any $c\in A\otimes B$. 
\item $(\alpha_g\otimes\beta_g)(w_n)-w_n$ tends to zero as $n\to\infty$ 
for every $g\in G$. 
\end{itemize}
Then the same argument as \cite[Proposition 7.2.1]{Ro_text} yields 
the desired cocycle conjugacy. 
\end{proof}

\begin{lemma}[{\cite[Lemma 6.2]{IM10Adv}}]
Let $G$ be a countable infinite discrete amenable group and 
let $\{s_g\}_{g\in G}$ be the generator of 
the Cuntz algebra $\mathcal{O}_\infty$. 
Define an action $\beta:G\curvearrowright\mathcal{O}_\infty$ 
by $\beta_g(s_h)=s_{gh}$. 
For any outer cocycle action $(\alpha,u)$ of $G$ 
on a unital Kirchberg algebra $A$, 
there exists a unital homomorphism $\pi:\mathcal{O}_\infty\to A_\omega$ 
such that $\pi\circ\beta_g=\alpha_g\circ\pi$ for all $g\in G$. 
\end{lemma}

\begin{theorem}\label{Oinfty_absorb}
Let $G$ be a poly-$\Z$ group and 
let $(\alpha,u):G\curvearrowright A$ be an outer cocycle action 
on a unital Kirchberg algebra $A$. 
Let $\beta:G\curvearrowright\mathcal{O}_\infty$ be an outer action. 
Then $(\alpha,u)$ is cocycle conjugate to 
$(\alpha\otimes\beta,u\otimes1):G\curvearrowright A\otimes\mathcal{O}_\infty$ 
via an isomorphism $\psi:A\to A\otimes\mathcal{O}_\infty$ 
which is asymptotically unitarily equivalent to 
the embedding $A\ni a\mapsto a\otimes1\in A\otimes\mathcal{O}_\infty$. 
\end{theorem}
\begin{proof}
We verify the hypotheses of Theorem \ref{McDuff}. 
By Theorem \ref{Oinfty_unique}, 
we may assume that $\beta$ equals the action defined in the lemma above. 
Then, condition (1) of Theorem \ref{McDuff} immediately follows. 
Theorem \ref{Oinfty_unique} also tells us that the diagonal action 
$\beta\otimes\beta:G\curvearrowright
\mathcal{O}_\infty\otimes\mathcal{O}_\infty$ 
is approximately representable, and so 
in the same way as \cite[Lemma 6.1]{IM10Adv} 
we obtain condition (2) of Theorem \ref{McDuff}, 
thus completing the proof. 
\end{proof}

Notice that 
a strengthened version of the theorem above has been obtained 
by G. Szab\'o \cite{Sz1807arXiv}. 
See Remark \ref{workofSzabo}. 

\begin{remark}
One may appeal to Szab\'o's result \cite[Corollary 3.7]{Sz18CMP} 
in order to obtain the $\mathcal{O}_\infty$-absorption theorem above. 
Namely, by Theorem \ref{Oinfty_unique}, 
any outer action $\beta:G\curvearrowright\mathcal{O}_\infty$ is 
(very strongly) cocycle conjugate to 
the model action $\gamma:G\curvearrowright\mathcal{O}_\infty$, 
which was constructed in Theorem \ref{exist_asymprepre}. 
It follows from \cite[Corollary 3.7]{Sz18CMP} that 
any outer cocycle action of $G$ on a unital Kirchberg algebra $A$ 
absorbs $\gamma$, and hence $\beta$. 
See also Remark \ref{modelaction}. 
\end{remark}

\begin{remark}\label{Oinfty_absorb'}
Let $G$ be a poly-$\Z$ group and 
let $(\alpha,u):G\curvearrowright A$ be an outer cocycle action 
on a unital Kirchberg algebra $A$. 
Let $\beta:G\curvearrowright\mathcal{O}_\infty$ be an outer action. 
It follows from Theorem \ref{Oinfty_absorb} that 
$(\alpha,u)$ is also cocycle conjugate to 
$(\alpha\otimes\beta\otimes\id,u\otimes1\otimes1):G\curvearrowright 
A\otimes\mathcal{O}_\infty\otimes\mathcal{O}_\infty$, 
and hence is cocycle conjugate to 
$(\alpha\otimes\id,u\otimes1):G\curvearrowright A\otimes\mathcal{O}_\infty$. 
It is easy to see that 
the isomorphism $A\to A\otimes\mathcal{O}_\infty$ can be taken 
to be asymptotically unitarily equivalent to 
the embedding $a\mapsto a\otimes1$. 
Szab\'o proved that the same statement holds true 
for any countable discrete amenable group $G$ 
(\cite[Theorem 3.4]{Sz18CMP}). 
\end{remark}

\subsection{Actions on algebras in the Cuntz standard form}

In this subsection, we prove that 
any outer cocycle action of a poly-$\Z$ group 
on a unital Kirchberg algebra in the Cuntz standard form is 
cocycle conjugate to a genuine action (Theorem \ref{Cuntzstandard1}). 

Let $A$ be a unital $C^*$-algebra such that the following holds: 
two non-zero projections in $A$ are Murray-von Neumann equivalent 
whenever they have the same $K_0$-class. 
Let $(\alpha,u):G\curvearrowright A$ be 
a cocycle action of a countable discrete group $G$. 
Let $p\in A\setminus\{0\}$ be a projection such that $K_0(p)\in K_0(A)^G$. 
For each $g\in G$, we choose a partial isometry $x_g\in A$ 
so that $x_gx_g^*=p$ and $x_g^*x_g=\alpha_g(p)$. 
Set $\alpha^x_g(a)=x_g\alpha_g(a)x_g^*$ for $a\in pAp$ and 
$u^x(g,h)=x_g\alpha_g(x_h)u(g,h)x_{gh}^*$. 
Then $(\alpha^x,u^x)$ is a cocycle action of $G$ on $pAp$. 
If $(y_g)_g$ is also a family of partial isometries satisfying 
$y_gy_g^*=p$ and $y_g^*y_g=\alpha_g(p)$, then 
$(\alpha^y,u^y)$ is the perturbation of $(\alpha^x,u^x)$ by $(y_gx_g^*)_g$. 
Thus, the equivalence class of $(\alpha^x,u^x)$ does not 
depend on the choice of $(x_g)_g$. 
We write one of those cocycle actions $(\alpha,u)^p$. 

It is routine to check the following. 

\begin{lemma}\label{reduction}
Let $G$ be a countable discrete group and 
let $A$ be a unital Kirchberg algebra. 
\begin{enumerate}
\item Let $(\alpha,u):G\curvearrowright A$ be a cocycle action and 
let $p,q\in A\setminus\{0\}$ be projections 
such that $K_0(p),K_0(q)\in K_0(A)^G$. 
If $K_0(p)=K_0(q)$ in $K_0(A)$, then 
$(\alpha,u)^p$ is cocycle conjugate to $(\alpha,u)^q$. 
\item Suppose that two cocycle actions $(\alpha,u)$ and $(\beta,v)$ 
of $G$ on $A$ are $KK$-trivially cocycle conjugate. 
Let $p\in A\setminus\{0\}$ be a projection such that $K_0(p)\in K_0(A)^G$. 
Then $(\alpha,u)^p$ and $(\beta,v)^p$ are $KK$-trivially cocycle conjugate. 
\end{enumerate}
\end{lemma}
\begin{proof}
(1) is immediate from the definition. 
Let us prove (2). 
There exist $\gamma\in\Aut(A)$ and a family $(c_g)_g$ of unitaries in $A$ 
such that 
\[
KK(\gamma)=1_A,\quad 
\gamma\circ\beta_g\circ\gamma^{-1}=\Ad c_g\circ\alpha_g\quad\text{and}\quad 
c_g\alpha_g(c_h)u(g,h)c_{gh}^*=\gamma(v(g,h))
\]
for all $g,h\in G$. 
For $g\in G$, we choose partial isometries $x_g,y_g\in A$ so that 
$x_gx_g^*=p=y_gy_g^*$, $x_g^*x_g=\alpha_g(p)$ and $y_g^*y_g=\beta_g(p)$. 
Also, take a partial isometry $w\in A$ 
such that $ww^*=p$ and $w^*w=\gamma(p)$. 
Define $\tilde\gamma\in\Aut(pAp)$ by $\tilde\gamma(a)=w\gamma(a)w^*$. 
Clearly $KK(\tilde\gamma)=1_{pAp}$ in $KK(pAp,pAp)$. 
For $a\in pAp$, we have 
\begin{align*}
(\tilde\gamma\circ\beta^y_g\circ\tilde\gamma^{-1})(a)
&=w\gamma(y_g\beta_g(\gamma^{-1}(w^*aw))y_g^*)w^*\\
&=w\gamma(y_g)c_g\alpha_g(w^*aw)c_g^*\gamma(y_g^*)w^*\\
&=(\Ad(w\gamma(y_g)c_g\alpha_g(w^*)x_g^*)\circ\alpha^x_g)(a), 
\end{align*}
where $z_g=w\gamma(y_g)c_g\alpha_g(w^*)x_g^*$ is a unitary in $pAp$. 
Besides, for every $g,h\in G$, we can check 
\begin{align*}
z_g\alpha^x_g(z_h)u^x(g,h)z_{gh}^*
&=w\gamma(y_g)c_g\alpha_g(w^*)x_g^*
\cdot x_g\alpha_g(w\gamma(y_h)c_h\alpha_h(w^*)x_h^*)x_g^*\\
&\qquad\cdot x_g\alpha_g(x_h)u(g,h)x_{gh}^*
\cdot x_{gh}\alpha_{gh}(w)c_{gh}^*\gamma(y_{gh}^*)w^*\\
&=w\gamma(y_g)c_g\alpha_g(\gamma(y_h)c_h\alpha_h(w^*))
u(g,h)\alpha_{gh}(w)c_{gh}^*\gamma(y_{gh}^*)w^*\\
&=w\gamma(y_g)\gamma(\beta_g(y_h))c_g\alpha_g(c_h)u(g,h)
\alpha_{gh}(\gamma(p))c_{gh}^*\gamma(y_{gh}^*)w^*\\
&=w\gamma(y_g\beta_g(y_h))c_g\alpha_g(c_h)u(g,h)c_{gh}^*
\gamma(y_{gh}^*)w^*\\
&=w\gamma(y_g\beta_g(y_h)v(g,h)y_{gh}^*)w\\
&=\tilde\gamma(v^y(g,h)). 
\end{align*}
It follows that 
$(\alpha^x,u^x)$ is $KK$-trivially cocycle conjugate to $(\beta^y,v^y)$. 
\end{proof}

We denote by $\mathcal{O}$ the Kirchberg algebra 
which is strongly Morita equivalent to $\mathcal{O}_\infty$ 
and is in the Cuntz standard form. 

The following theorem is a generalization of \cite[Corollary 7.12]{IM10Adv}. 

\begin{theorem}\label{Cuntzstandard1}
Let $G$ be a poly-$\Z$ group. 
Let $(\alpha,u):G\curvearrowright A$ be an outer cocycle action of $G$ 
on a unital Kirchberg algebra $A$ in the Cuntz standard form. 
Then $(\alpha,u)$ is cocycle conjugate to a genuine action. 
\end{theorem}
\begin{proof}
Since $\mathcal{O}$ contains a unital copy of $\mathcal{O}_2$, 
from Theorem \ref{O2_unique}, 
$(\alpha\otimes\id,u\otimes1):G\curvearrowright A\otimes\mathcal{O}$ is 
cocycle conjugate to a genuine action. 
Let $p\in\mathcal{O}_\infty$ be a non-zero projection such that $K_0(p)=0$. 
By Lemma \ref{reduction} (1), 
$(\alpha\otimes\id,u\otimes1):G\curvearrowright A\otimes\mathcal{O}_\infty$ 
is cocycle conjugate to 
$(\alpha\otimes\id,u\otimes1):
G\curvearrowright A\otimes p\mathcal{O}_\infty p\cong A\otimes\mathcal{O}$, 
because $K_0(1)=0=K_0(1\otimes p)$ in $K_0(A\otimes\mathcal{O}_\infty)$. 
By Remark \ref{Oinfty_absorb'}, $(\alpha,u)$ is cocycle conjugate to 
$(\alpha\otimes\id,u\otimes1):G\curvearrowright A\otimes\mathcal{O}_\infty$. 
Therefore $(\alpha,u)$ is cocycle conjugate to a genuine action. 
\end{proof}

\begin{theorem}\label{Cuntzstandard2}
Let $(\alpha,u)$ and $(\beta,v)$ be outer cocycle actions 
of a poly-$\Z$ group $G$ on a unital Kirchberg algebra $A$. 
The following conditions are equivalent. 
\begin{enumerate}
\item $(\alpha,u)$ and $(\beta,v)$ are $KK$-trivially cocycle conjugate. 
\item $(\alpha\otimes\id,u\otimes1):G\curvearrowright A\otimes\mathcal{O}$ 
and $(\beta\otimes\id,v\otimes1):G\curvearrowright A\otimes\mathcal{O}$ are 
$KK$-trivially cocycle conjugate. 
\end{enumerate}
\end{theorem}
\begin{proof}
(1)$\Rightarrow$(2) is obvious. 
Let us assume (2). 
Take a projection $p\in\mathcal{O}$ 
such that $p\mathcal{O}p\cong\mathcal{O}_\infty$. 
It follows from Lemma \ref{reduction} (2) that 
$(\alpha\otimes\id,u\otimes1)^{1\otimes p}$ and 
$(\beta\otimes\id,v\otimes1)^{1\otimes p}$ are 
$KK$-trivially cocycle conjugate. 
Thus, 
$(\alpha\otimes\id,u\otimes1):G\curvearrowright A\otimes\mathcal{O}_\infty$ 
and $(\beta\otimes\id,v\otimes1):G\curvearrowright A\otimes\mathcal{O}_\infty$ 
are $KK$-trivially cocycle conjugate. 
By Remark \ref{Oinfty_absorb'}, we can conclude that 
$(\alpha,u)$ and $(\beta,v)$ are $KK$-trivially cocycle conjugate. 
\end{proof}

The theorems above say that 
classification of outer cocycle actions on $A$ 
up to $KK$-trivial cocycle conjugacy is exactly equivalent to 
classification of outer (genuine) actions on $A\otimes\mathcal{O}$ 
up to $KK$-trivial cocycle conjugacy.

\section{Stability}

The main purpose of this section is 
to show stability type properties for actions of poly-$\Z$ groups 
(Theorem \ref{H1H2stable}).

\subsection{Approximately central embedding of $\mathcal{O}_\infty$}

In this subsection, we introduce the notion of 
approximately central embedding of 
an action $\mu:G\curvearrowright\mathcal{O}_\infty$. 

\begin{definition}\label{defofAC}
Let $G$ be a countable discrete group and 
let $\mu:G\curvearrowright\mathcal{O}_\infty$ be an action. 
Let $(\alpha,u)$ be a cocycle action of $G$ 
on a unital (not-necessarily separable) $C^*$-algebra $A$. 
We say that $(\alpha,u)$ admits 
an approximately central embedding of $(\mathcal{O}_\infty,\mu)$ 
if for any separable subset $S\subset A$ there exists 
a unital homomorphism $\phi:\mathcal{O}_\infty\to A^\omega\cap S'$ 
such that $\phi\circ\mu_g=\alpha_g\circ\phi$ for any $g\in G$. 
We denote by $\AC(\mathcal{O}_\infty,\mu)$ 
the class of such cocycle actions. 
\end{definition}

\begin{lemma}\label{AC}
Let $\mu:G\curvearrowright\mathcal{O}_\infty$ be an outer action 
of a poly-$\Z$ group $G$ and 
let $(\alpha,u):G\curvearrowright A$ be a cocycle action 
on a unital $C^*$-algebra $A$. 
\begin{enumerate}
\item Let $\nu:G\curvearrowright\mathcal{O}_\infty$ be another outer action. 
Then, $(\alpha,u)$ is in $\AC(\mathcal{O}_\infty,\mu)$ 
if and only if $(\alpha,u)$ is in $\AC(\mathcal{O}_\infty,\nu)$. 
\item Let $(\beta,v):G\curvearrowright B$ be a cocycle action. 
If $(\alpha,u)$ is in $\AC(\mathcal{O}_\infty,\mu)$, 
then so is $(\alpha\otimes\beta,u\otimes v)$. 
\item If $A$ is a unital Kirchberg algebra and $(\alpha,u)$ is outer, 
then $(\alpha,u)$ is in $\AC(\mathcal{O}_\infty,\mu)$. 
\end{enumerate}
\end{lemma}
\begin{proof}
(1)
By Theorem \ref{Oinfty_unique} and Remark \ref{Oinfty_unique'}, 
$\mu$ and $\nu$ are cocycle conjugate 
with a cocycle arbitrarily close to $1$. 
Hence, the conclusion follows immediately. 

(2)
This is clear from the definition. 

(3)
It follows from Theorem \ref{Oinfty_absorb} that 
$(\alpha,u):G\curvearrowright A$ absorbs 
$(\bigotimes_\N\mathcal{O}_\infty,\bigotimes_\N\mu)$ tensorially, 
and so $(\alpha,u)$ belongs to $\AC(\mathcal{O}_\infty,\mu)$. 
\end{proof}

By Lemma \ref{AC} (1), when $G$ is a poly-$\Z$ group, 
the class $\AC(\mathcal{O}_\infty,\mu)$ does not depend 
on the choice of the outer action $\mu:G\curvearrowright\mathcal{O}_\infty$. 

\begin{lemma}\label{AC'}
Let $\mu:G\curvearrowright\mathcal{O}_\infty$ be an outer action 
of a poly-$\Z$ group $G$ and 
let $(\alpha,u):G\curvearrowright A$ be a cocycle action 
on a unital separable $C^*$-algebra $A$. 
Suppose that $(\alpha,u)$ is in $\AC(\mathcal{O}_\infty,\mu)$. 
\begin{enumerate}
\item $(\alpha,u)$ is cocycle conjugate to $(\alpha\otimes\mu,u\otimes1)$ 
via an isomorphism asymptotically unitarily equivalent to 
the embedding $a\mapsto a\otimes1$. 
\item $\alpha:G\curvearrowright A^\omega$ and 
$\alpha:G\curvearrowright A_\omega$ 
are also in $\AC(\mathcal{O}_\infty,\mu)$. 
\item $\alpha:G\curvearrowright A^\flat$ and 
$\alpha:G\curvearrowright A_\flat$ 
are also in $\AC(\mathcal{O}_\infty,\mu)$. 
\end{enumerate}
\end{lemma}
\begin{proof}
(1)
Since $A$ is separable and $(\alpha,u)$ is in $\AC(\mathcal{O}_\infty,\mu)$, 
there exists a unital homomorphism $\pi:\mathcal{O}_\infty\to A_\omega$ 
such that $\pi\circ\mu_g=\alpha_g\circ\pi$ for any $g\in G$. 
Thus, condition (1) of Theorem \ref{McDuff} is satisfied. 
In the proof of Theorem \ref{Oinfty_absorb}, 
we observed that the outer action $\mu:G\curvearrowright\mathcal{O}_\infty$ 
satisfies condition (2) of Theorem \ref{McDuff}. 
Therefore we obtain the desired conclusion. 

(2)
By definition, 
there exists a unital homomorphism 
$\phi:\mathcal{O}_\infty\to A^\omega\cap A'=A_\omega$ such that 
$\phi\circ\mu_g=\alpha_g\circ\phi$ for any $g\in G$. 
For any given separable subset $S\subset A^\omega$, 
by the usual reindexation trick, 
we may modify the homomorphism $\phi$ and 
obtain $\psi:\mathcal{O}_\infty\to A_\omega\cap S'$ 
such that $\phi\circ\mu_g=\alpha_g\circ\phi$. 

(3)
By (1), $\alpha:G\curvearrowright A_\flat$ is conjugate to 
$\alpha\otimes\mu:G\curvearrowright(A\otimes\mathcal{O}_\infty)_\flat$. 
Evidently, $((\mathcal{O}_\infty)_\flat,\mu)$ is embeddable 
into $((A\otimes\mathcal{O}_\infty)_\flat,\alpha\otimes\mu)$. 
By Corollary \ref{embedtoAflat}, 
$(\mathcal{O}_\infty,\mu)$ is embeddable 
into $((\mathcal{O}_\infty)_\flat,\mu)$. 
Hence there exists a unital homomorphism 
$\phi:\mathcal{O}_\infty\to A_\flat$ such that 
$\phi\circ\mu_g=\alpha_g\circ\phi$ for any $g\in G$. 
In the same way as (2), we get the conclusion. 
\end{proof}

\subsection{$H^1$-stability and $H^2$-stability}

In this subsection, we prove that 
any poly-$\Z$ group is asymptotically $H^1$-stable and $H^2$-stable 
(Theorem \ref{H1H2stable}). 
For every poly-$\Z$ group $G$, 
we choose and fix an outer action 
$\mu^G:G\curvearrowright\mathcal{O}_\infty$. 

\begin{lemma}\label{check}
Let $G$ be a countable discrete group and 
let $N\subset G$ be a normal subgroup such that $G/N\cong\Z$. 
Choose $\xi\in G$ so that $N$ and $\xi$ generate $G$. 
Suppose that $(\alpha,u):G\curvearrowright A$ is a cocycle action 
on a unital $C^*$-algebra $A$. 
If $u(g,h)=1$ for all $g,h\in N$, then 
the unitaries 
\[
\check u_g=u(\xi,\xi^{-1}g\xi)u(g,\xi)^*
\]
form 
an $\alpha|N$-cocycle satisfying 
\[
\alpha_\xi\circ\alpha_{\xi^{-1}g\xi}\circ\alpha_\xi^{-1}
=\Ad\check u_g\circ \alpha_g\quad\forall g\in N. 
\]
\end{lemma}
\begin{proof}
Straightforward computations. 
\end{proof}

\begin{lemma}\label{1cv_xi_homo1}
Let $G$ be a countable discrete amenable group and 
let $N\subset G$ be a normal subgroup such that $G/N\cong\Z$. 
Choose $\xi\in G$ so that $N$ and $\xi$ generate $G$. 
For any finite subset $K\subset N$ and $\ep>0$, 
there exist a finite subset $L\subset N$ and $\delta>0$ 
such that the following holds. 

Let $\mu:G\curvearrowright\mathcal{O}_\infty$ be an outer action 
such that $\mu|N$ is approximately representable. 
Suppose that 
a cocycle action $(\alpha,w)$ of $G$ on a unital $C^*$-algebra $A$ 
belongs to $\AC(\mathcal{O}_\infty,\mu)$ 
and that $J\subset A$ is a globally $\alpha$-invariant ideal. 
Assume further that $w(g,h)=1$ for all $g,h\in N$ and 
the $\alpha|N$-cocycle $\check w_g=w(\xi,\xi^{-1}g\xi)w(g,\xi)^*$ 
satisfies 
\[
\lVert\check w_g-1\rVert<\delta\quad\forall g\in L. 
\]
If a unitary $u\in U(C([0,1])\otimes J)$ satisfies 
\[
u(0)=1,\quad 
\lVert u(t)-\alpha_g(u(t))\rVert<\delta\quad\forall t\in[0,1],\ g\in L, 
\]
then there exists a unitary $v\in U(C([0,1])\otimes J)$ such that 
\[
v(0)=1,\quad 
\lVert v(t)-\alpha_g(v(t))\rVert<\ep\quad \forall t\in[0,1],\ g\in K
\]
and 
\[
\lVert u(t)-v(t)\alpha_\xi(v(t)^*)\rVert<\ep\quad \forall t\in[0,1]. 
\]
\end{lemma}
\begin{proof}
Choose $m\in\N$ and $\delta>0$ so that $6\pi/m<\ep$ and $(84m+81)\delta<\ep$. 
Define $L_0,L\subset N$ by 
\[
L_0=\{\xi^{-i}g\xi^i\mid g\in K,\ i=0,1,\dots,m\},\quad 
L=\{\xi^{-i}g\xi^i\mid g\in L_0,\ i=0,1,\dots,m\}. 
\]
Let $u\in U(C([0,1])\otimes J)$ be a unitary satisfying 
\[
u(0)=1,\quad 
\lVert u(t)-\alpha_g(u(t))\rVert<\delta\quad\forall t\in[0,1],\ g\in L. 
\]
To simplify notation, 
we denote the cocycle action 
$(\id\otimes\alpha,1\otimes w):G\curvearrowright C([0,1])\otimes A$ 
by $(\alpha,w)$. 
Define $u_i\in U(C([0,1])\otimes J)$ for $i=0,1,2,\dots$ inductively 
by $u_0=1$ and $u_{i+1}=u\alpha_\xi(u_i)$. 
By an elementary estimate, we obtain 
\[
\lVert u_i-\alpha_g(u_i)\rVert<3i\delta,\quad 
\forall g\in\{\xi^{-j}h\xi^j\mid h\in L_0,\ 0\leq j\leq m{+}1{-}i\},\ 
i=0,1,\dots,m{+}1. 
\]
Indeed, by induction, we can verify 
\begin{align*}
\alpha_g(u_{i+1})&=\alpha_g(u)(\alpha_g\circ\alpha_\xi)(u_i)
\approx_\delta
u(\Ad\check w_g\circ\alpha_\xi\circ\alpha_{\xi^{-1}g\xi})(u_i)\\
&\approx_{2\delta}u(\alpha_\xi\circ\alpha_{\xi^{-1}g\xi})(u_i)
\approx_{3i\delta}u\alpha_\xi(u_i)=u_{i+1}
\end{align*}
for any $g\in\{\xi^{-j}h\xi^j\mid h\in L_0,\ 0\leq j\leq m{-}i\}$. 
By \cite[Corollary 3.2 (1)]{IM10Adv}, 
$((\mathcal{O}_\infty)^\omega)^{\mu|N}$ contains 
a unital copy of $\mathcal{O}_\infty$. 
Hence, by the assumption, $(A^\omega)^{\alpha|N}$ also contains 
a unital copy of $\mathcal{O}_\infty$. 
Therefore, by using the same method as \cite[Theorem 7]{Na00ETDS} 
(see also \cite[Lemma 3.3]{IM10Adv}), 
we can find a unitary $x\in U(C([0,1]\times [0,1])\otimes J)$ such that 
\[
x(0,t)=1,\quad x(s,0)=1,\quad x(1,t)=u_m(t),\quad 
\Lip(x(\cdot,t))<6\pi\quad \forall s,t\in[0,1]
\]
and 
\[
\lVert x(s,t)-\alpha_g(x(s,t))\rVert<81m\delta\quad\forall g\in L_0. 
\]
Similarly, 
we get a unitary $y\in U(C([0,1]\times [0,1])\otimes J)$ for $u_{m+1}$. 

Since $\mu:G\curvearrowright\mathcal{O}_\infty$ is outer and 
$\mu|N:N\curvearrowright\mathcal{O}_\infty$ is approximately representable, 
it follows from Proposition \ref{equivRohlin} that 
$\mu_\xi$ has Rohlin towers of height $m,m{+}1$ 
in $((\mathcal{O}_\infty)_\omega)^{\mu|N}$. 
Let 
\[
S=\{u_i(t)\mid 0\leq t\leq1,\ i=0,1,\dots,m\}
\cup\{x(s,t),y(s,t)\mid s,t\in[0,1]\}. 
\]
By assumption there exists a unital equivariant embedding of 
$\mathcal{O}_\infty$ to $A^\omega\cap S'$, and so 
we can find projections $e,f\in (A^\omega)^{\alpha|N}\cap S'$ such that 
\[
\sum_{i=0}^{m-1}\alpha_\xi^i(e)+\sum_{j=0}^m\alpha_\xi^j(f)=1,\quad 
\alpha_\xi^m(e)=e,\quad \alpha_\xi^{m+1}(f)=f. 
\]
Define a unitary $v\in U(C([0,1])\otimes J^\omega)$ by 
\[
v(t)=\sum_{i=0}^{m-1}u_i(t)\alpha_\xi^i(x(1-i/m,t))\alpha_\xi^i(e)
+\sum_{j=0}^mu_j(t)\alpha_\xi^j(y(1-j/m,t))\alpha_\xi^j(f). 
\]
We can easily see $v(0)=1$ and $\lVert u-v\alpha_\xi(v^*)\rVert<6\pi/m<\ep$. 
Furthermore, we get 
\[
\lVert v-\alpha_g(v)\rVert<3m\delta+81(m+1)\delta
=(84m+81)\delta<\ep\quad \forall g\in K, 
\]
which completes the proof. 
\end{proof}

\begin{lemma}\label{1cv_G_homo1}
Suppose that an action $\alpha:G\curvearrowright A$ of a poly-$\Z$ group 
is in $\AC(\mathcal{O}_\infty,\mu^G)$ 
and that $J\subset A$ is a globally $\alpha$-invariant ideal. 
Let $(u_g)_{g\in G}$ be 
an $\id\otimes\alpha$-cocycle in $U(C([0,1])\otimes J)$ 
satisfying $u_g(0)=1$ for every $g\in G$. 
Then, for any finite subset $K\subset G$ and $\ep>0$, 
there exists a unitary $v\in U(C([0,1])\otimes J)$ such that 
\[
v(0)=1,\quad 
\sup_{t\in[0,1]}\lVert u_g(t)-v(t)\alpha_g(v(t)^*)\rVert<\ep
\]
holds for any $g\in K$. 
\end{lemma}
\begin{proof}
Notice that 
any outer action of any poly-$\Z$ group on $\mathcal{O}_\infty$ is 
approximately representable by Theorem \ref{Oinfty_unique}. 
The proof is by induction on the Hirsch length of $G$. 
When $G=\{1\}$, we have nothing to do. 
For a non-trivial poly-$\Z$ group $G$, 
there exists a normal poly-$\Z$ subgroup $N\subset G$ such that $G/N\cong\Z$. 
Choose and fix $\xi\in G$ so that $N$ and $\xi$ generate $G$. 
From the induction hypothesis we may assume that the theorem is known for $N$. 
To simplify notation, 
we denote the $G$-action 
$\id\otimes\alpha:G\curvearrowright C([0,1])\otimes A$ by $\alpha$. 

Suppose that 
we are given a finite subset $K\subset G$ and $\ep>0$. 
Without loss of generality, we may assume that 
there exist a finite subset $K_0\subset N$ 
such that $K=K_0\cup\{\xi\}$. 
By applying Lemma \ref{1cv_xi_homo1} 
to $K_0\subset N$ and $\ep/2>0$, 
we obtain $L\subset N$ and $\delta>0$. 
By using the induction hypothesis to 
$\alpha|N:N\curvearrowright A$, 
we can find a unitary $v_1\in U(C([0,1])\otimes J)$ such that 
\[
v_1(0)=1,\quad 
\lVert u_g-v_1\alpha_g(v_1^*)\rVert<\min\{\ep/2,\delta/2\}
\]
for any $g\in K_0\cup L\cup\xi^{-1}L\xi$. 
It is not so hard to see that 
\begin{align*}
\alpha_g(v_1^*u_\xi\alpha_\xi(v_1))
&\approx_{\delta/2}v_1^*u_g\alpha_g(u_\xi)\alpha_{g\xi}(v_1)\\
&=v_1^*u_{g\xi}\alpha_{g\xi}(v_1)\\
&=v_1^*u_\xi\alpha_\xi(u_{\xi^{-1}g\xi}\alpha_{\xi^{-1}g\xi}(v_1))
\approx_{\delta/2}v_1^*u_\xi\alpha_\xi(v_1)
\end{align*}
holds for every $g\in L$. 
It follows from Lemma \ref{1cv_xi_homo1} that 
there exists a unitary $v_2\in U(C([0,1])\otimes J)$ such that 
\[
v_2(0)=1,\quad 
\lVert v_2-\alpha_g(v_2)\rVert<\ep/2\quad \forall g\in K_0
\]
and 
\[
\lVert v_1^*u_\xi\alpha_\xi(v_1)-v_2\alpha_\xi(v_2^*)\rVert<\ep/2. 
\]
Put $v=v_1v_2$. 
It is routine to verify 
\[
\lVert u_g-v\alpha_g(v^*)\rVert<\ep/2+\ep/2=\ep\quad \forall g\in K_0
\]
and 
\[
\lVert u_\xi-v\alpha_\xi(v^*)\rVert<\ep/2, 
\]
thereby completing the proof. 
\end{proof}

We introduce the notion of 
approximate (or asymptotic) $H^1$-stability and $H^2$-stability. 
In what follows, we always assume that 
a generating subset $S$ of a group $G$ is symmetric (i.e.\ $S=S^{-1}$) 
and contains $1$. 

\begin{definition}\label{defofH1H2}
Let $G$ be a poly-$\Z$ group and 
let $S\subset G$ be a finite generating subset. 
\begin{enumerate}
\item We say that $(G,S)$ is approximately $H^1$-stable 
if there exists $\delta>0$ satisfying the following property. 
Suppose that an action $\alpha:G\curvearrowright A$ 
is in $\AC(\mathcal{O}_\infty,\mu^G)$ 
and that $J\subset A$ is a globally $\alpha$-invariant ideal. 
If $(u_g)_g$ is an $\alpha$-cocycle in $U(J)_0$ 
satisfying $\lVert u_g-1\rVert<\delta$ for all $g\in S$, 
then there exists a sequence $(v_n)_n$ of unitaries in $U(J)_0$ 
such that 
\[
\lim_{n\to\infty}v_n\alpha_g(v_n)^*=u_g\quad\forall g\in S. 
\]
\item We say that $(G,S)$ is asymptotically $H^1$-stable 
if $G$ has the property above 
with a continuous family $(v_t)_{t\in[0,\infty)}$ 
instead of the sequence $(v_n)_n$. 
\item We say that $(G,S)$ is $H^2$-stable 
if for any $\ep>0$ there exists $\delta>0$ satisfying the following property. 
Suppose that a cocycle action $(\alpha,u):G\curvearrowright A$ 
is in $\AC(\mathcal{O}_\infty,\mu^G)$ 
and that $J\subset A$ is a globally $\alpha$-invariant ideal. 
If $(u(g,h))_{g,h}$ is contained in $U(J)$ and 
$\lVert u(g,h)-1\rVert<\delta$ for all $g,h\in S$, then 
there exists a family of unitaries $(v_g)_g$ in $U(J)$ such that 
\[
u(g,h)=\alpha_g(v_h)^*v_g^*v_{gh}\quad\forall g,h\in G
\]
and 
\[
\lVert v_g-1\rVert<\ep\quad\forall g\in S. 
\]
In addition, when the unitaries $(u(g,h))_{g,h}$ are in $U(J)_0$, 
we require that the unitaries $(v_g)_g$ can be taken from $U(J)_0$. 
\end{enumerate}
\end{definition}

The approximate $H^1$-stability of $\Z$ was (essentially) proved 
by Nakamura \cite[Lemma 8]{Na00ETDS}. 
Obviously, $\Z$ is $H^2$-stable. 
The $H^2$-stability of $\Z^2$ was (essentially) proved 
in \cite[Proposition 7.10]{IM10Adv}. 

\begin{lemma}\label{rem_Hstable}
Let $G$ be a poly-$\Z$ group and 
let $S\subset G$ be a finite generating subset. 
\begin{enumerate}
\item If $(G,S)$ is approximately (or asymptotically) $H^1$-stable 
and $S'\subset G$ is another finite generating subset, 
then so is $(G,S')$. 
\item If $(G,S)$ is $H^2$-stable, then 
$(G,S^m)$ is $H^2$-stable for any $m\in\N$. 
\end{enumerate}
\end{lemma}
\begin{proof}
(1)
Let $\alpha:G\curvearrowright A$ be an action and 
let $(u_g)_g$ be an $\alpha$-cocycle. 
For any $g\in S'$ and $h\in(S')^n$, one has 
\begin{align*}
\lVert u_{gh}-1\rVert&=\lVert u_g\alpha_g(u_h)-1\rVert\\
&\leq\lVert u_g-1\rVert+\lVert u_h-1\rVert\\
&\leq\max\{\lVert u_k-1\rVert\mid k\in S'\}
+\max\{\lVert u_k-1\rVert\mid k\in(S')^n\}. 
\end{align*}
Thus, $\lVert u_g-1\rVert\leq n\max\{\lVert u_k-1\rVert\mid k\in S'\}$ 
for any $g\in(S')^n$. 
There exists $m\in\N$ such that $S\subset(S')^m$, 
and so $(G,S')$ is approximately (or asymptotically) $H^1$-stable. 

(2)
Suppose that we are given $\ep>0$. 
Let $\delta>0$ be the constant derived from the $H^2$-stability of $G$ 
for $S\subset G$ and $\ep/(2m)>0$. 
We show that $\delta'=\min\{\delta,\ep/(2m)\}$ meets the requirement. 

Suppose that a cocycle action $(\alpha,u):G\curvearrowright A$ 
on a unital $C^*$-algebra $A$ belongs to $\AC(\mathcal{O}_\infty,\mu^G)$ 
and that $J\subset A$ is a globally $\alpha$-invariant ideal. 
Assume that $(u(g,h))_{g,h}$ is contained in $U(J)$ (resp.\ $U(J)_0$) and 
$\lVert u(g,h)-1\rVert<\delta'$ for all $g,h\in S^m$. 
By the $H^2$-stability, 
there exists a family of unitaries $(v_g)_g$ in $U(J)_0$ (resp.\ $U(J)_0$) 
such that $u(g,h)=\alpha_g(v_h)^*v_g^*v_{gh}$ for all $g,h\in G$ and 
\[
\lVert v_g-1\rVert<\ep/(2m)\quad\forall g\in S. 
\]
For any $n\leq m$, $g\in S$ and $h\in S^n$, we have 
\[
\lVert v_{gh}-1\rVert
\leq\lVert u(g,h)-1\rVert+\lVert v_g-1\rVert+\lVert v_h-1\rVert
<\delta'+\ep/(2m)+\lVert v_h-1\rVert. 
\]
Therefore, $\lVert v_g-1\rVert<m\delta'+\ep/2\leq\ep$ for any $g\in S^m$. 
\end{proof}

\begin{lemma}\label{H2stable_pre}
Let $G$ be a poly-$\Z$ group and 
let $S\subset G$ be a finite generating subset. 
Suppose that $(G,S)$ is approximately $H^1$-stable and $H^2$-stable. 
Then, for any $\ep>0$ there exists $\delta>0$ 
such that the following holds. 

Suppose that an action $\alpha:G\curvearrowright A$ 
is in $\AC(\mathcal{O}_\infty,\mu^G)$ 
and that $J\subset A$ is a globally $\alpha$-invariant ideal. 
If $(u_g)_g$ is an $\alpha$-cocycle in $U(J)_0$ 
satisfying $\lVert u_g-1\rVert<\delta$ for all $g\in S$, 
then for any $\ep'>0$, there exists a unitary $u\in U(C([0,1])\otimes J)_0$ 
such that 
\[
u(0)=1,\quad \lVert u(1)u_g\alpha_g(u(1))^*-1\rVert<\ep', 
\]
\[
\lVert\alpha_g(u(1))-u(1)\rVert<\delta,\quad 
\lVert\alpha_g(u(t))-u(t)\rVert<\ep\quad\forall g\in S,\ t\in[0,1]. 
\]
\end{lemma}
\begin{proof}
Let $\delta_1>0$ be the constant 
derived from the approximate $H^1$-stability of $G$. 
Let $\delta_2>0$ be the constant 
derived from the $H^2$-stability of $G$ for $\ep/4>0$. 
Set $\delta=\min\{\ep/8,\delta_1,\delta_2/6\}$. 

Suppose that an action $\alpha:G\curvearrowright A$ 
is in $\AC(\mathcal{O}_\infty,\mu^G)$ 
and that $J\subset A$ is a globally $\alpha$-invariant ideal. 
Let $(u_g)_g$ be an $\alpha$-cocycle in $U(J)_0$ satisfying 
\[
\lVert u_g-1\rVert<\delta\quad\forall g\in S. 
\]
Let $\ep'>0$. 
To simplify notation, 
we denote the $G$-action 
$\id\otimes\alpha:G\curvearrowright C([0,1])\otimes A$ by $\alpha$. 
By the approximate $H^1$-stability of $G$, 
we can find $x_0\in U(J)_0$ such that 
\[
\lVert x_0u_g\alpha_g(x_0)^*-1\rVert
<\min\left\{\frac{\ep'}{2},\ 
\delta{-}\max\{\lVert u_h{-}1\rVert\mid h\in S\}\right\}
\quad\forall g\in S. 
\]
Then we have $\lVert x_0-\alpha_g(x_0)\rVert<\delta$ for all $g\in S$, 
and hence $\lVert x_0-\alpha_g(x_0)\rVert<2\delta$ for all $g\in S^2$. 
Choose $x\in U(C([0,1]\otimes J)_0$ so that $x(0)=1$, $x(1)=x_0$. 
For each $g\in S^2$, 
there exists $v_g\in U(C([0,1])\otimes J)_0$ such that 
$v_g(0)=v_g(1)=1$ and 
\[
\lVert v_g-x\alpha_g(x)^*\rVert<2\delta\quad\forall g\in S^2. 
\]
For $g\in G\setminus S^2$, we let $v_g=1$. 
Then we may think of $(\alpha^v,w)$ as a cocycle action 
on $C([0,1])\otimes A$, where 
\[
w(g,h)=v_g\alpha_g(v_h)v_{gh}^*\in U(C_0((0,1))\otimes J). 
\]
For any $g,h\in S$, we have $\lVert w(g,h)-1\rVert<6\delta\leq\delta_2$. 
Hence, by the $H^2$-stability, 
we obtain $(y_g)_g$ in $U(C_0((0,1))\otimes J)$ such that 
\[
w(g,h)=\alpha^v_g(y_h)^*y_g^*y_{gh}\quad\forall g,h\in G
\]
and 
\[
\lVert y_g-1\rVert<\ep/4\quad\forall g\in S. 
\]
Set $\tilde v_g=y_gv_g$. 
It is easy to check that $(\tilde v_g)_g$ is an $\alpha$-cocycle 
such that $\tilde v_g(0)=\tilde v_g(1)=1$ and 
\[
\lVert\tilde v_g-x\alpha_g(x)^*\rVert
<\ep/4+\lVert v_g-x\alpha_g(x)^*\rVert
<\ep/4+2\delta\leq\ep/2\quad\forall g\in S. 
\]
It follows from Lemma \ref{1cv_G_homo1} that 
there exists $z\in U(C([0,1])\otimes J)$ such that $z(0)=1$ and 
\[
\lVert\tilde v_g-z\alpha_g(z)^*\rVert
<\min\left\{\frac{\ep}{2},\ \frac{\ep'}{2},\ 
\delta{-}\max\{\lVert x_0-\alpha_h(x_0)\rVert\mid h\in S\}\right\}. 
\]
Put $u=z^*x\in U(C([0,1])\otimes J)$. 
Clearly $u(0)=1$. 
For any $g\in S$, one has 
\[
u(1)u_g\alpha_g(u(1))^*
=z(1)^*x(1)u_g\alpha_g(x(1))^*\alpha_g(z(1))
\approx_{\ep'/2}z(1)^*\alpha_g(z(1))\approx_{\ep'/2}1, 
\]
\[
u(1)\alpha_g(u(1))^*=z(1)^*x(1)\alpha_g(x(1))\alpha_g(z(1)^*)
=z(1)^*x_0\alpha_g(x_0)\alpha_g(z(1)^*)\approx_\delta1
\]
and 
\[
u(t)\alpha_g(u(t))^*=z(t)^*x(t)\alpha_g(x(t)^*)\alpha_g(z(t))
\approx_{\ep/2}z(t)^*\tilde v_g(t)\alpha_g(z(t))\approx_{\ep/2}1, 
\]
which complete the proof. 
\end{proof}

\begin{lemma}
Let $G$ be a poly-$\Z$ group and 
let $N\subset G$ be a normal poly-$\Z$ subgroup such that $G/N\cong\Z$. 
Choose and fix $\xi\in G$ so that $G$ is generated by $N$ and $\xi$. 
Let $S\subset N$ be a finite generating subset. 
Suppose that $(N,S)$ is approximately $H^1$-stable and $H^2$-stable. 
For any $\ep>0$ there exists $\delta>0$ such that the following holds. 

Suppose that a cocycle action $(\alpha,u):G\curvearrowright A$ 
is in $\AC(\mathcal{O}_\infty,\mu^G)$ 
and that $J\subset A$ is a globally $\alpha$-invariant ideal. 
Assume further that $u(g,h)=1$ for all $g,h\in N$ and 
the $\alpha|N$-cocycle $\check u_g=u(\xi,\xi^{-1}g\xi)u(g,\xi)^*$ 
is contained in $U(J)_0$. 
If $\lVert\check u_g-1\rVert<\delta$ for all $g\in S$, then 
for any $\ep'>0$, there exist unitaries $u,v\in U(J)_0$ such that 
\[
\lVert u\check u_g\alpha_g(u)^*-1\rVert<\ep',\quad 
\lVert u-\alpha_g(u)\rVert<\delta, 
\]
\[
\lVert u-v\alpha_\xi(v)^*\rVert<\ep,\quad \lVert v-\alpha_g(v)\rVert<\ep
\]
hold for all $g\in S$. 
\end{lemma}
\begin{proof}
Applying Lemma \ref{1cv_xi_homo1} to $S\subset N$ and $\ep>0$, 
we get a finite subset $L\subset N$ and $\delta_1>0$. 
There exists $m\in\N$ such that $L\subset S^m$. 
Applying Lemma \ref{H2stable_pre} to $S\subset N$ and $\delta_1/m>0$, 
we get $\delta>0$. 
We may assume that $\delta$ is not greater than $\delta_1$. 

Suppose that a cocycle action $(\alpha,u):G\curvearrowright A$ 
on a unital $C^*$-algebra $A$ belongs to $\AC(\mathcal{O}_\infty,\mu^G)$ 
and that $J\subset A$ is a globally $\alpha$-invariant ideal. 
Assume that $u(g,h)=1$ for all $g,h\in N$ and 
that the $\alpha|N$-cocycle $\check u_g=u(\xi,\xi^{-1}g\xi)u(g,\xi)^*$ 
belongs to $U(J)_0$ and satisfies 
\[
\lVert\check u_g-1\rVert<\delta\quad\forall g\in S. 
\]
Let $\ep'>0$. 
By Lemma \ref{H2stable_pre}, 
there exists a unitary $u\in U(C([0,1])\otimes J)_0$ such that 
\[
u(0)=1,\quad \lVert u(1)\check u_g\alpha_g(u(1))^*-1\rVert<\ep', 
\]
\[
\lVert\alpha_g(u(1))-u(1)\rVert<\delta,\quad 
\lVert\alpha_g(u(t))-u(t)\rVert<\delta_1/m\quad\forall g\in S,\ t\in[0,1]. 
\]
By the choice of $m$, one has 
\[
\lVert\alpha_g(u(t))-u(t)\rVert<\delta_1\quad\forall g\in L,\ t\in[0,1]. 
\]
It follows from Lemma \ref{1cv_xi_homo1} that 
there exists a unitary $v\in U(J)_0$ such that 
\[
\lVert v-\alpha_g(v)\rVert<\ep\quad \forall g\in S
\]
and 
\[
\lVert u(1)-v\alpha_\xi(v^*)\rVert<\ep, 
\]
which complete the proof. 
\end{proof}

\begin{lemma}\label{H2stable}
Let $G$ be a poly-$\Z$ group and 
let $N\subset G$ be a normal poly-$\Z$ subgroup such that $G/N\cong\Z$. 
Let $S\subset N$ be a finite generating subset (with $S=S^{-1}$ and $1\in S$). 
If $(N,S)$ is approximately $H^1$-stable and $H^2$-stable, 
then there exists a finite generating set $S'\subset G$ such that 
$(G,S')$ is $H^2$-stable. 
\end{lemma}
\begin{proof}
Choose and fix $\xi\in G$ so that $G$ is generated by $N$ and $\xi$. 
There exists $m\in\N$ such that $S\cup\xi^{-1}S\xi\subset S^m$. 
Put $S'=S^m\cup\{\xi,\xi^{-1}\}\subset G$. 
Suppose that we are given $\ep>0$. 
Let $(\ep_n)_n$ be a sequence of positive real numbers 
such that $\sum_n\ep_n<\ep/(2m)$. 
Applying the previous lemma to $\ep_n$, we obtain $0<\delta_n<1$. 
We may assume $\delta_n\to0$ as $n\to\infty$. 
By Lemma \ref{rem_Hstable} (2), $(N,S^m)$ is $H^2$-stable. 
Let $\delta_0>0$ be the constant derived from the $H^2$-stability of $N$ 
for $S^m$ and $\min\{\ep/2,\delta_1/4\}>0$. 
Set $\delta=\min\{\ep/2,\delta_0,\delta_1/4\}$. 

Suppose that a cocycle action $(\alpha,u):G\curvearrowright A$ 
on a unital $C^*$-algebra $A$ is in $\AC(\mathcal{O}_\infty,\mu^G)$ 
and that $J\subset A$ is a globally $\alpha$-invariant ideal. 
Assume that $(u(g,h))_{g,h}$ is contained in $U(J)$ and 
\[
\lVert u(g,h)-1\rVert<\delta\quad\forall g,h\in S'. 
\]
By the $H^2$-stability of $N$, 
we can find unitaries $(w_g)_{g\in N}$ in $U(J)$ 
such that 
\[
\lVert w_g-1\rVert<\min\{\ep/2,\delta_1/4\}\quad\forall g\in S^m
\]
and 
\[
u(g,h)=\alpha_g(w_h)^*w_g^*w_{gh}\quad\forall g,h\in N. 
\]
Put $w_g=1$ for $G\setminus N$. 
We consider the perturbation $(\alpha^w,u^w)$ of $(\alpha,u)$ by $w$. 
Since $u^w(g,h)=1$ for all $g,h\in N$, $\alpha^w|N$ is a genuine action. 
Moreover, one has 
\[
\alpha^w_\xi\circ\alpha^w_{\xi^{-1}g\xi}\circ(\alpha^w_\xi)^{-1}
=\Ad\check u^w_g\circ\alpha^w_g\quad\forall g\in N, 
\]
where $(\check u^w_g)_g$ is an $\alpha^w|N$-cocycle in $U(J)$ 
defined by $\check u^w_g=u^w(\xi,\xi^{-1}g\xi)u^w(g,\xi)^*$. 
For every $g\in S$, we can verify 
\[
\check u^w_g
=\alpha_\xi(w_{\xi^{-1}g\xi})u(\xi,\xi^{-1}g\xi)u(g,\xi)^*w_g^*
\approx_{\delta_1/2}u(\xi,\xi^{-1}g\xi)u(g,\xi)^*\approx_{\delta_1/2}1. 
\]
In particular, $(\check u^w_g)_{g\in N}$ are in $U(J)_0$. 
It follows from the previous lemma that 
there exist unitaries $u_1,v_1\in U(J)_0$ such that 
\[
\lVert u_1\check u^w_g\alpha^w_g(u_1)^*-1\rVert<\delta_2,\quad 
\lVert u_1-\alpha^w_g(u_1)\rVert<\delta_1, 
\]
\[
\lVert u_1-v_1\alpha^w_\xi(v_1)^*\rVert<\ep_1,\quad 
\lVert v_1-\alpha^w_g(v_1)\rVert<\ep_1
\]
hold for all $g\in S$. 
We define unitaries $(x_{1,g})_{g\in G}$ in $U(J)_0$ by 
\[
x_{1,g}=\begin{cases}v_1^*\alpha^w_g(v_1)&g\in N\\
v_1^*u_1\alpha^w_\xi(v_1)&g=\xi\\
1&\text{otherwise. }\end{cases}
\]
It is easy to see 
\[
\lVert x_{1,g}-1\rVert<k\ep_1\quad\forall g\in S^k,\ k\in\N
\]
and $\lVert x_{1,\xi}-1\rVert<\ep_1$. 
Let $(\alpha^{(2)},u^{(2)})$ be the cocycle perturbation 
of $(\alpha^w,u^w)$ by $(x_{1,g})_g$, i.e. 
\[
\alpha^{(2)}_g=\Ad x_{1,g}\circ\alpha^w_g
\]
and 
\[
u^{(2)}(g,h)=x_{1,g}\alpha^w_g(x_{1,h})u^w(g,h)x_{1,gh}^*. 
\]
Then we have 
\[
\alpha^{(2)}_\xi\circ\alpha^{(2)}_{\xi^{-1}g\xi}\circ(\alpha^{(2)}_\xi)^{-1}
=\Ad\check u^{(2)}_g\circ\alpha^{(2)}_g\quad\forall g\in N, 
\]
where $(\check u^{(2)}_g)_g$ is an $\alpha^{(2)}|N$-cocycle in $U(J)$ 
defined by 
\begin{align*}
\check u^{(2)}_g
&=u^{(2)}(\xi,\xi^{-1}g\xi)u^{(2)}(g,\xi)^*\\
&=x_{1,\xi}\alpha^w_\xi(x_{1,\xi^{-1}g\xi})u^w(\xi,\xi^{-1}g\xi)x_{1,g\xi}^*
\cdot\left(x_{1,g}\alpha^w_g(x_{1,\xi})u^w(g,\xi)x_{1,g\xi}^*\right)^*\\
&=v_1^*u_1\alpha^w_\xi(v_1)\alpha^w_\xi(v_1^*\alpha^w_{\xi^{-1}g\xi}(v_1))
u^w(\xi,\xi^{-1}g\xi)u^w(g,\xi)^*\alpha^w_g(v_1^*u_1\alpha^w_\xi(v_1))^*
\alpha^w_g(v_1)^*v_1\\
&=v_1^*u_1u^w(\xi,\xi^{-1}g\xi)\alpha^w_{g\xi}(v_1)
\alpha^w_{g\xi}(v_1)^*u^w(g,\xi)^*\alpha^w_g(u_1)^*v_1\\
&=v_1^*u_1\check u^w_g\alpha^w_g(u_1)^*v_1\quad\forall g\in N. 
\end{align*}
Clearly, 
\[
\lVert\check u^{(2)}_g-1\rVert
=\lVert u_1\check u^w_g\alpha^w_g(u_1)^*-1\rVert<\delta_2. 
\]
holds for every $g\in S$. 
In particular, 
the unitaries $(\check u^{(2)}_g)_{g\in N}$ belong to $U(J)_0$. 
It follows from the previous lemma that 
there exist unitaries $u_2,v_2\in U(J)_0$ such that 
\[
\lVert u_2\check u^{(2)}_g\alpha^{(2)}_g(u_2)^*-1\rVert<\delta_3,\quad 
\lVert u_2-\alpha^{(2)}_g(u_2)\rVert<\delta_2, 
\]
\[
\lVert u_2-v_2\alpha^{(2)}_\xi(v_2)^*\rVert<\ep_2,\quad 
\lVert v_2-\alpha^{(2)}_g(v_2)\rVert<\ep_2
\]
hold for all $g\in S$. 
We define unitaries $(x_{2,g})_{g\in G}$ in $U(J)_0$ by 
\[
x_{2,g}=\begin{cases}v_2^*\alpha^{(2)}_g(v_2)&g\in N\\
v_2^*u_2\alpha^{(2)}_\xi(v_2)&g=\xi\\
1&\text{otherwise. }\end{cases}
\]
It is easy to see 
\[
\lVert x_{2,g}-1\rVert<k\ep_2\quad\forall g\in S^k,\ k\in\N
\]
and $\lVert x_{2,\xi}-1\rVert<\ep_2$. 
Let $(\alpha^{(3)},u^{(3)})$ be the cocycle perturbation 
of $(\alpha^{(2)},u^{(2)})$ by $(x_{2,g})_g$. 
Then we have 
\[
\alpha^{(3)}_\xi\circ\alpha^{(3)}_{\xi^{-1}g\xi}\circ(\alpha^{(3)}_\xi)^{-1}
=\Ad\check u^{(3)}_g\circ\alpha^{(3)}_g\quad\forall g\in N, 
\]
where $(\check u^{(3)}_g)_g$ is an $\alpha^{(3)}|N$-cocycle in $U(J)$ 
defined by 
\[
\check u^{(3)}_g=u^{(3)}(\xi,\xi^{-1}g\xi)u^{(3)}(g,\xi)^*
=v_2^*u_2\check u^{(2)}_g\alpha^{(2)}_g(u_2)^*v_2
\quad\forall g\in N. 
\]
Clearly, 
\[
\lVert\check u^{(3)}_g-1\rVert
=\lVert u_2\check u^{(2)}_g\alpha^{(2)}_g(u_2)^*-1\rVert<\delta_3. 
\]
holds for every $g\in S$, 
and hence $(\check u^{(3)}_g)_{g\in N}$ are contained in $U(J)_0$. 

Repeating the same process, for each $n\geq3$, 
we get a family of unitaries $(x_{n,g})_g$, 
a cocycle action $(\alpha^{(n)},u^{(n)})$ 
and an $\alpha^{(n)}|N$-cocycle $(\check u^{(n)}_g)_g$ 
such that the following hold. 
\begin{itemize}
\item $(\alpha^{(n)},u^{(n)})$ is the cocycle perturbation 
of $(\alpha^{(n-1)},u^{(n-1)})$ by $(x_{n-1,g})_g$. 
\item $u^{(n)}(g,h)=1$ for all $g,h\in N$. 
\item $\alpha^{(n)}_\xi\circ\alpha^{(n)}_{\xi^{-1}g\xi}
\circ(\alpha^{(n)}_\xi)^{-1}=\Ad\check u^{(n)}_g\circ\alpha^{(n)}_g$ 
for all $g\in N$. 
\item $\lVert\check u^{(n)}(\xi,\xi^{-1}g\xi)-\check u^{(n)}(g,\xi)\rVert
<\delta_n$ for all $g\in S$. 
\item $\lVert x_{n,\xi}-1\rVert<\ep_n$ and $\lVert x_{n,g}-1\rVert<k\ep_n$ 
for all $g\in S^k$, $k\in\N$, 
and $x_{n,g}=1$ for all $g\notin N\cup\{\xi\}$. 
\end{itemize}
Moreover, since $\sum_n\ep_n<\ep/(2m)$, the limit 
\[
x_g=\lim_{n\to\infty}x_{n,g}x_{n-1,g}\dots x_{1,g}
\]
exists for all $g\in G$ and 
$\lVert x_g-1\rVert<\ep/(2m)$ for all $g\in S\cup\{\xi\}$. 
Let $(\alpha',u')$ be the cocycle perturbation of $(\alpha^w,u^w)$ 
by $(x_g)_g$. 
Then, we obtain 
$u'(g,h)=1$ and $u'(\xi,\xi^{-1}g\xi)=u'(g,\xi)$ for all $g,h\in N$. 
Hence there uniquely exists 
a family of unitaries $(y_g)_{g\in G}$ in $U(J)$ such that 
$y_g=1$ for all $g\in N\cup\{\xi\}$ and 
the cocycle perturbation of $(\alpha',u')$ by $(y_g)_g$ is 
a genuine action. 
Note that 
\[
y_{\xi^{-1}}=u'(\xi^{-1},\xi)^*
=(\alpha_{\xi^{-1}}(x_\xi)u^w(\xi^{-1},\xi))^*
\approx_{\ep/(2m)}u^w(\xi^{-1},\xi)^*=u(\xi^{-1},\xi)^*\approx_{\ep/2}1. 
\]
Consequently, 
\[
\lVert y_\xi x_\xi w_\xi-1\rVert
=\lVert x_\xi-1\rVert<\ep/(2m), 
\]
\[
\lVert y_{\xi^{-1}}x_{\xi^{-1}}w_{\xi^{-1}}-1\rVert
=\lVert y_{\xi^{-1}}-1\rVert<\ep
\]
and 
\[
\lVert y_gx_gw_g-1\rVert=\lVert x_gw_g-1\rVert
<\ep/2+\ep/2=\ep
\]
for all $g,h\in S^m$. 

When the family of unitaries $(u(g,h))_{g,h\in G}$ is contained in $U(J)_0$, 
we can choose $(w_g)_{g\in N}$ and $(y_g)_{g\in G}$ from $U(J)_0$, 
thereby completing the proof. 
\end{proof}

\begin{lemma}\label{appH1stable}
Let $G$ be a poly-$\Z$ group and 
let $S\subset G$ be a finite generating subset. 
If $(G,S)$ is $H^2$-stable, 
then it is approximately $H^1$-stable. 
\end{lemma}
\begin{proof}
Let $\delta>0$ be the constant derived from the $H^2$-stability 
of $(G,S)$ for $\ep=1$. 

Suppose that an action $\alpha:G\curvearrowright A$ 
is in $\AC(\mathcal{O}_\infty,\mu^G)$ 
and that $J\subset A$ is a globally $\alpha$-invariant ideal. 
Assume that an $\alpha$-cocycle $(u_g)_g$ in $U(J)_0$ satisfies 
$\lVert u_g-1\rVert<\delta/6$ for all $g\in S$. 
Then we have $\lVert u_g-1\rVert<\delta/3$ for all $g\in S^2$. 
To simplify notation, 
we denote the $G$-action 
$\id\otimes\alpha:G\curvearrowright C([0,1])\otimes A$ by $\alpha$. 
For each $g\in G$, 
we choose $v_g\in U(C([0,1])\otimes J)_0$ such that 
$v_g(0)=1$, $v_g(1)=u_g$ for all $g\in G$ and 
$\lVert v_g-1\rVert<\delta/3$ for all $g\in S^2$. 
Let $\sigma_g=\Ad v_g\circ\alpha_g$ and 
let $w(g,h)=v_g\alpha_g(v_h)v_{gh}^*\in U(C_0((0,1))\otimes J)$. 
Then $(\sigma,w)$ is a cocycle action of $G$ on $C([0,1])\otimes A$ 
satisfying $\lVert w(g,h)-1\rVert<\delta$ for all $g,h\in S$. 
Applying the $H^2$-stability of $(G,S)$, 
we get $\tilde u_g\in U(C_0((0,1))\otimes J)$ satisfying 
$w(g,h)=\sigma_g(\tilde u_h)^*\tilde u_g^*\tilde u_{gh}$ for all $g,h\in S$. 
It is routine to check 
that $(\tilde u_gv_g)_g$ is an $\alpha$-cocycle 
satisfying $(\tilde u_gv_g)(0)=1$ and $(\tilde u_gv_g)(1)=u_g$. 
Now the statement follows from Lemma \ref{1cv_G_homo1}. 
\end{proof}

\begin{lemma}\label{asympH1stable}
Let $G$ be a poly-$\Z$ group and 
let $S\subset G$ be a finite generating subset. 
If $(G,S)$ is approximately $H^1$-stable and $H^2$-stable, 
then it is asymptotically $H^1$-stable. 
\end{lemma}
\begin{proof}
Applying Lemma \ref{H2stable_pre} to $\ep=1/n$, 
we get $\delta_n>0$. 
We may assume $\delta_n\to0$ as $n\to\infty$. 

Suppose that an action $\alpha:G\curvearrowright A$ 
on a unital $C^*$-algebra $A$ belongs to $\AC(\mathcal{O}_\infty,\mu^G)$ 
and that $J\subset A$ is a globally $\alpha$-invariant ideal. 
Assume that an $\alpha$-cocycle $(u_g)_g$ in $U(J)_0$ satisfies 
$\lVert u_g-1\rVert<\delta_1$ for all $g\in S$. 
By Lemma \ref{H2stable_pre}, 
we can find a unitary $v_1\in U(C([0,1])\otimes J)_0$ 
such that 
\[
v_1(0)=1,\quad \lVert v_1(1)u_g\alpha_g(v_1(1))^*-1\rVert<\delta_2, 
\]
\[
\lVert\alpha_g(v_1(t))-v_1(t)\rVert<1\quad\forall g\in S,\ t\in[0,1]. 
\]
Then $u^{(1)}_g=v_1(1)u_g\alpha_g(v_1(1))^*$ form an $\alpha$-cocycle 
satisfying $\lVert u^{(1)}_g-1\rVert<\delta_2$ for all $g\in S$. 
Again using Lemma \ref{H2stable_pre}, 
we can find a unitary $v_2\in U(C([0,1])\otimes J)_0$ 
such that 
\[
v_2(0)=1,\quad \lVert v_2(1)u^{(1)}_g\alpha_g(v_2(1))^*-1\rVert<\delta_3, 
\]
\[
\lVert\alpha_g(v_2(t))-v_2(t)\rVert<1/2\quad\forall g\in S,\ t\in[0,1]. 
\]
Then $u^{(2)}_g=v_2(1)u^{(1)}_g\alpha_g(v_2(1))^*$ form an $\alpha$-cocycle 
satisfying $\lVert u^{(2)}_g-1\rVert<\delta_3$ for all $g\in S$. 

Repeating this process, we obtain $v_n$ and $(u^{(n)}_g)_g$. 
We inductively construct a continuous map $w:[0,\infty)\to U(J)_0$ 
by $w(t)=v_1(t)$ for $t\in[0,1]$ and 
$w(t)=v_{n+1}(t{-}n)w(n)$ for $t\in[n,n{+}1]$. 
For $t\in[n,n{+}1]$ and $g\in S$, we have 
\begin{align*}
\lVert w(t)u_g\alpha_g(w(t))^*-1\rVert
&=\lVert v_{n+1}(t{-}n)w(n)u_g\alpha_g(v_{n+1}(t{-}n)w(n))^*-1\rVert\\
&=\lVert v_{n+1}(t{-}n)u^{(n)}_g\alpha_g(v_{n+1}(t{-}n))^*-1\rVert\\
&<\lVert v_{n+1}(t{-}n)u^{(n)}_gv_{n+1}(t{-}n)^*-1\rVert+(n+1)^{-1}\\
&<\delta_{n+1}+(n+1)^{-1}, 
\end{align*}
which shows that the family $(w(t))_t$ has the desired property. 
\end{proof}

\begin{theorem}\label{H1H2stable}
Let $G$ be a poly-$\Z$ group. 
There exists a finite generating subset $S\subset G$ such that 
$(G,S)$ is asymptotically $H^1$-stable and $H^2$-stable. 
\end{theorem}
\begin{proof}
This follows from induction 
using Lemma \ref{H2stable}, Lemma \ref{appH1stable} 
and Lemma \ref{asympH1stable}. 
\end{proof}

\section{$KK$-trivial cocycle conjugacy}

In this section, 
we establish Theorem \ref{KKtrivcc} and Theorem \ref{KKtrivcc2}, 
which will be the main tool for classification theorems 
in Section 7 and Section 8. 

First, let us recall a refinement of Nakamura's homotopy theorem 
from \cite{IMweakhomot}. 

\begin{theorem}[{\cite[Theorem 5.5]{IMweakhomot}}]\label{refinedNakamura}
Let $A$ be a unital separable $C^*$-algebra and let $I\subset A$ be an ideal. 
Let $(\alpha,u):G\curvearrowright A$ be a cocycle action 
of a countable discrete group $G$. 
Suppose that $(A_\flat)^\alpha$ contains 
a unital copy of $\mathcal{O}_\infty$. 
For any continuous map $x:[0,1]\times[0,\infty)\to U(I)$, 
there exists a continuous map 
$y:[0,1]\times[0,\infty)\to U(I)$ such that 
\[
y(0,t)=x(0,t),\quad y(1,t)=x(1,t),\quad \Lip(y(\cdot,t))<25\pi
\quad\forall t\in[0,\infty), 
\]
\[
\limsup_{t\to\infty}\sup_{s\in[0,1]}\lVert[y(s,t),a]\rVert
\leq243\limsup_{t\to\infty}\sup_{s\in[0,1]}\lVert[x(s,t),a]\rVert
\quad\forall a\in A
\]
and 
\[
\limsup_{t\to\infty}\sup_{s\in[0,1]}\lVert\alpha_g(y(s,t))-y(s,t)\rVert
\leq243\limsup_{t\to\infty}\sup_{s\in[0,1]}\lVert\alpha_g(x(s,t))-x(s,t)\rVert
\quad\forall g\in G. 
\]
\end{theorem}

For every poly-$\Z$ group $G$, 
we choose and fix an outer action 
$\mu^G:G\curvearrowright\mathcal{O}_\infty$. 

\begin{lemma}
Let $G$ be a poly-$\Z$ group. 
For any $\ep>0$, 
there exist a finite generating subset $S\subset G$ and $\delta>0$ 
such that the following holds. 

Suppose that an action $\alpha:G\curvearrowright A$ 
on a unital separable $C^*$-algebra $A$ 
belongs to $\AC(\mathcal{O}_\infty,\mu^G)$ 
and that $J\subset A$ is a globally $\alpha$-invariant ideal. 
If $(u_g)_g$ is an $\alpha$-cocycle in $U(J^\flat\cap A')_0$ 
satisfying $\lVert u_g-1\rVert<\delta$ for all $g\in S$, then 
there exists a continuous map $w:[0,1]\to U(J^\flat\cap A')_0$ 
such that $w(0)=1$, $\lVert w(t)-\alpha_g(w(t))\rVert<\ep$ 
for all $t\in[0,1]$ and $g\in S$ and 
$u_g=w(1)\alpha_g(w(1))^*$ for all $g\in G$. 
\end{lemma}
\begin{proof}
By Theorem \ref{H1H2stable}, 
there exists a finite generating subset $S\subset G$ 
such that $(G,S)$ is asymptotically $H^1$-stable and $H^2$-stable. 
Let $\delta_1>0$ be the constant derived 
from the asymptotic $H^1$-stability of $(G,S)$. 
We may and do assume $\delta_1<\ep/243$. 
Let $\delta_2>0$ be the constant derived 
from the $H^2$-stability of $(G,S)$ for $\delta_1/2$. 
Put $\delta=\min\{\delta_1/2,\delta_2/3\}$. 

Suppose that an action $\alpha:G\curvearrowright A$ 
on a unital separable $C^*$-algebra $A$ is in $\AC(\mathcal{O}_\infty,\mu^G)$ 
and that $J\subset A$ is a globally $\alpha$-invariant ideal. 
By Lemma \ref{AC'}, 
$\id\otimes\alpha:G\curvearrowright C([0,1])\otimes A_\flat$ is 
also in $\AC(\mathcal{O}_\infty,\mu^G)$. 
By abuse of notation, 
$\id\otimes\alpha$ is simply denoted by $\alpha$. 
Let $(u_g)_g$ be an $\alpha$-cocycle in $U(J^\flat\cap A')_0$ 
satisfying $\lVert u_g-1\rVert<\delta$ for all $g\in S^2$. 
Choose continuous paths $\tilde u_g:[0,1]\to U(J^\flat\cap A')_0$ 
such that $\tilde u_g(0)=1$, $\tilde u_g(1)=u_g$ for all $g\in G$ and 
$\lVert\tilde u_g-1\rVert<\delta$ for all $g\in S^2$. 
Then 
\[
v(g,h)=\tilde u_g\alpha_g(\tilde u_h)\tilde u_{gh}^*
\in U(C_0((0,1))\otimes(J^\flat\cap A')), 
\]
and $(\alpha^{\tilde u},v)$ is a cocycle action 
on $C([0,1])\otimes(J^\flat\cap A')$. 
Since 
\[
\lVert v(g,h)-1\rVert<3\delta\leq\delta_2\quad\forall g\in S, 
\]
the $H^2$-stability implies that 
there exists a family of unitaries $(\hat u_g)_g$ 
in $U(C_0((0,1))\otimes(J^\flat\cap A'))$ such that 
$v(g,h)=\alpha^{\tilde u}_g(\hat u_h^*)\hat u_g^*\hat u_{gh}$ 
for all $g,h\in G$ and 
$\lVert\hat u_g-1\rVert<\delta_1/2$ for all $g\in S$. 
It is routine to check that $(\hat u_g\tilde u_g)_g$ is 
a path of $\alpha$-cocycles from $1$ to $(u_g)_g$ satisfying 
\[
\lVert\hat u_g\tilde u_g-1\rVert<\delta_1/2+\delta\leq\delta_1
\quad\forall g\in S. 
\]
To simplify notation, 
we write $(\tilde u_g)_g$ instead of $(\hat u_g\tilde u_g)_g$. 
The asymptotic $H^1$-stability yields 
a continuous map $x:[0,\infty)\to U(C_0((0,1])\otimes(J^\flat\cap A'))_0$ 
such that $x(0)=1$ and $x(s)\alpha_g(x(s)^*)\to\tilde u_g$ as $s\to\infty$. 

We would like to find a lift 
$\hat x:[0,\infty)\to U(C_0((0,1])\otimes C_b([0,\infty),J))$ 
of $x$ satisfying $\hat x(0)=1$. 
For each $n\in\N\cup\{0\}$, 
the restriction of $x$ to $[n,n{+}1]$ is regarded 
as an element of $U(C([n,n{+}1])\otimes C_0((0,1])\otimes J^\flat)$, 
and so there exists a lift 
$\hat x_n:[n,n{+}1]\to U(C_0((0,1])\otimes C_b([0,\infty),J))$. 
We may assume $\hat x_0(0)=1$. 
Inductively we define $\hat x$ by 
$\hat x(s)=\hat x_0(s)$ for $s\in[0,1]$ and 
$\hat x(s)=\hat x(n)\hat x_n(n)^*\hat x_n(s)$ for $s\in(n,n{+}1]$. 
Then $\hat x:[0,\infty)\to U(C_0((0,1])\otimes C_b([0,\infty),J))$ 
is a lift of $x$. 
We treat $\hat x$ as a continuous function 
in two variables $(s,t)\in[0,\infty)\times[0,\infty)$. 

Let $(F_n)_{n\in\N}$ be an increasing sequence of finite subsets of $A$ 
whose union is dense in $A$. 
By abuse of notation, 
we identify $\tilde u_g\in U(C_0((0,1])\otimes(J^\flat\cap A'))$ 
with its lift $\tilde u_g:[0,\infty)\to U(C_0((0,1])\otimes J)$, 
and may assume $\lVert\tilde u_g(t)-1\rVert<\delta_1$ 
for all $g\in S$ and $t\in[0,\infty)$. 
We choose a strictly increasing sequence $(s_n)_{n\in\N}$ of 
positive real numbers satisfying 
\[
\lVert x(s)\alpha_g(x(s)^*)-\tilde u_g\rVert<1/n
\quad\forall s\geq s_n,\ g\in S, 
\]
and choose a strictly increasing sequence $(t_n)_{n\in\N}$ of 
positive real numbers satisfying 
\[
\lVert\hat x(s,t)\alpha_g(\hat x(s,t)^*)-\tilde u_g(t)\rVert<1/n
\quad\forall s\in[s_n,s_{n+1}],\ t\geq t_n,\ g\in S,
\]
\[
\lVert[\hat x(s,t),1\otimes a]\rVert<1/n
\quad\forall s\in[0,s_{n+1}],\ t\geq t_n,\ a\in F_n. 
\]
We define a continuous piecewise linear function $f:[0,\infty)\to[0,\infty)$ 
as follows: 
for $t\in [0,t_1]$, we set $f(t)=s_1t/t_1$, 
and for $t\in [t_n,t_{n+1}]$ with $n\geq 1$ we set 
\[
f(t)=s_n+\frac{(s_{n+1}-s_n)(t-t_n)}{t_{n+1}-t_n}. 
\]
Then $z(t)=\hat x(f(t),t)$ is a continuous map 
from $[0,\infty)$ to $U(C_0((0,1])\otimes J)$ satisfying 
\[
\lim_{t\to\infty}z(t)\alpha_g(z(t)^*)=\tilde u_g\quad\text{and}\quad 
\lim_{t\to\infty}[z(t),1\otimes a]=0\quad\forall g\in G,\ \forall a\in A. 
\]
Notice that one has 
\[
\limsup_{t\to\infty}\lVert z(t)-\alpha_g(z(t))\rVert\leq\delta_1
\quad\forall g\in S. 
\]
Regard $z$ as a continuous function 
in two variables $(r,t)\in[0,1]\times[0,\infty)$. 
Applying Theorem \ref{refinedNakamura} to $z(r,t)$, 
we can get a continuous map $w:[0,1]\to U(J^\flat\cap A')_0$ 
with the desired properties. 
\end{proof}

\begin{proposition}\label{extendedH1stable}
Let $G$ be a poly-$\Z$ group. 
Suppose that a cocycle action $(\alpha,u):G\curvearrowright A$ 
on a unital separable $C^*$-algebra $A$ is in $\AC(\mathcal{O}_\infty,\mu^G)$ 
and that $I\subset A$ is a globally $\alpha$-invariant ideal. 
Suppose that a family $(x_g)_{g\in G}$ of continuous maps 
from $[0,1]\times[0,\infty)$ to $U(I)$ satisfies 
\[
\lim_{t\to\infty}\max_{s\in[0,1]}
\lVert[x_g(s,t),a]\rVert=0\quad\forall g\in G,\ a\in A, 
\]
\[
\lim_{t\to\infty}\max_{s\in[0,1]}
\lVert x_g(s,t)\alpha_g(x_h(s,t))-x_{gh}(s,t)\rVert=0\quad\forall g,h\in G
\]
and $x_g(0,t)=1$. 
Then there exists a continuous map $y:[0,\infty)\to U(I)$ such that 
\[
\lim_{t\to\infty}\lVert[y(t),a]\rVert=0\quad\forall a\in A
\]
and 
\[
\lim_{t\to\infty}\lVert x_g(1,t)-y(t)\alpha_g(y(t)^*)\rVert=0
\quad\forall g\in G. 
\]
Moreover, the equivalence class of $(y(t))_t$ in $A^\flat$ 
belongs to $U(I^\flat\cap A')_0$. 
\end{proposition}
\begin{proof}
The proof is by induction on the Hirsch length of $G$. 
When $G$ is trivial, we have nothing to do. 
Suppose that the theorem is known for any poly-$\Z$ groups 
with Hirsch length less than $l$. 
Let $G$ be a poly-$\Z$ group with Hirsch length $l$. 
There exists a normal poly-$\Z$ subgroup $N\subset G$ 
whose Hirsch length equals $l-1$. 
Take $\xi\in G$ so that $G$ is generated by $N$ and $\xi$. 

Let $I\subset A$, $(\alpha,u):G\curvearrowright A$ and $(x_g)_{g\in G}$ 
be as in the statement. 
We define an ideal $J$ of $C([0,1],A)$ by 
\[
J=\{f:[0,1]\to A\mid f(s)\in I\ \forall s\in[0,1],\ f(0)=0\}. 
\]
Set $\tilde x_g(r,s,t)=x_g(rs,t)$ 
for $(r,s,t)\in[0,1]\times[0,1]\times[0,\infty)$. 
Then $(r,t)\mapsto \tilde x_g(r,\cdot,t)$ is regarded 
as a continuous map from $[0,1]\times[0,\infty)$ to $U(J)$. 
By Lemma \ref{AC} (2), 
$(\id\otimes\alpha,1\otimes u):G\curvearrowright C([0,1],A)$ 
is also in $\AC(\mathcal{O}_\infty,\mu^G)$. 
By the induction hypothesis, 
there exists a continuous map $y_1:[0,1]\times[0,\infty)\to U(I)$ such that 
\[
\lim_{t\to\infty}\sup_{s\in[0,1]}
\lVert[y_1(s,t),a]\rVert=0\quad\forall a\in A, 
\]
\[
\lim_{t\to\infty}\sup_{s\in[0,1]}
\lVert x_g(s,t)-y_1(s,t)\alpha_g(y_1(s,t)^*)\rVert=0
\quad\forall g\in N
\]
and $y_1(0,t)=1$. 
By Theorem \ref{refinedNakamura}, 
the equivalence class of $(y_1(1,t))_t$ in $A^\flat$ 
belongs to $U(I^\flat\cap A')_0$. 
Put $x'_g(s,t)=y_1(s,t)^*x_g(s,t)\alpha_g(y_1(s,t))$. 
It is easy to see 
\[
\lim_{t\to\infty}\max_{s\in[0,1]}
\lVert[x'_\xi(s,t),a]\rVert=0\quad\forall a\in A, 
\]
\[
\lim_{t\to\infty}\max_{s\in[0,1]}
\lVert\alpha_g(x'_\xi(s,t))-x'_\xi(s,t)\rVert=0\quad\forall g\in N
\]
and $x'_\xi(0,t)=1$. 
Since $\alpha:G\curvearrowright A_\flat$ is in $\AC(\mathcal{O}_\infty,\mu^G)$ 
by Lemma \ref{AC'} (3), 
we can find Rohlin projections for $\alpha_\xi$ in $(A_\flat)^{\alpha|N}$. 
This, together with Theorem \ref{refinedNakamura}, enables us to 
apply the usual stability argument. 
Thus, for any $\ep>0$, there exists $y_2:[0,\infty)\to U(I)$ such that 
\[
\lim_{t\to\infty}\lVert[y_2(t),a]\rVert=0\quad\forall a\in A,\quad 
\lim_{t\to\infty}\lVert\alpha_g(y_2(t))-y_2(t)\rVert=0\quad\forall g\in N
\]
and 
\[
\limsup_{t\to\infty}\lVert x'_\xi(1,t)-y_2(t)\alpha_\xi(y_2(t)^*)\rVert<\ep. 
\]
Moreover, the equivalence class of $(y_2(t))_t$ in $A^\flat$ 
belongs to $U(I^\flat\cap A')_0$. 
Consequently we have 
\[
\limsup_{t\to\infty}
\lVert x_g(1,t)-y_1(1,t)y_2(t)\alpha_g(y_2(t)^*y_1(1,t)^*)\rVert<\ep
\]
for all $g\in N\cup\{\xi\}$. 
Let $(\mathbf{x}_g)_g$ denote the $\alpha$-cocycle in $U(I^\flat\cap A')$ 
consisting of the equivalence classes of $(x_g(1,t))_t$. 
The above argument means that for any $\ep>0$ 
there exists $\mathbf{y}\in U(I^\flat\cap A')_0$ 
such that $\lVert\mathbf{x}_g-\mathbf{y}\alpha_g(\mathbf{y}^*)\rVert<\ep$ 
holds for any $g\in N\cup\{\xi\}$. 
Then the statement follows from the lemma above. 
\end{proof}

\begin{theorem}\label{KKtrivcc}
Let $G$ be a poly-$\Z$ group and 
let $A$ be a unital separable $C^*$-algebra. 
Let $(\alpha,u):G\curvearrowright A$ and $(\beta,v):G\curvearrowright A$ be 
cocycle actions belonging to $\AC(\mathcal{O}_\infty,\mu^G)$. 
Suppose that 
there exists a family $(x_g)_g$ of unitaries in $A^\flat$ such that 
\[
(\Ad x_g\circ\alpha_g)(a)=\beta_g(a)\quad\forall g\in G,\ \forall a\in A, 
\]
\[
x_g\alpha_g(x_h)u(g,h)x_{gh}^*=v(g,h)\quad\forall g,h\in G. 
\]
Then $(\alpha,u)$ and $(\beta,v)$ are cocycle conjugate 
via an asymptotically inner automorphism. 
\end{theorem}

\begin{theorem}\label{KKtrivcc2}
Let $G$ be a poly-$\Z$ group and 
let $(\alpha,u)$ and $(\beta,v)$ be cocycle actions of $G$ 
on a unital separable $C^*$-algebra $A$. 
Suppose that 
there exists a family $(x_g)_g$ of unitaries in $C_b([0,\infty),A)$ such that 
\[
\lim_{s\to\infty}(\Ad x_g(s)\circ\alpha_g)(a)=\beta_g(a)
\quad\forall g\in G,\ \forall a\in A, 
\]
\[
\lim_{s\to\infty}x_g(s)\alpha_g(x_h(s))u(g,h)x_{gh}(s)^*
=v(g,h)\quad\forall g,h\in G. 
\]
Then 
there exist a continuous map $w:[0,\infty)\to U(A\otimes\mathcal{O}_\infty)$, 
$\gamma\in\Aut(A\otimes\mathcal{O}_\infty)$ and 
a family $(c_g)_g$ of unitaries in $A\otimes\mathcal{O}_\infty$ satisfying 
\[
\lim_{s\to\infty}\Ad w(s)(a)=\gamma(a)
\quad\forall a\in A\otimes\mathcal{O}_\infty, 
\]
\[
\lim_{s\to\infty}w(s)(x_g(s)\otimes1)(\alpha_g\otimes\mu^G_g)(w(s))^*=c_g
\quad\forall g\in G
\]
and 
\[
c_g(\alpha_g\otimes\mu^G_g)(c_h)(u(g,h)\otimes1)c_{gh}^*
=\gamma(v(g,h)\otimes1)\quad\forall g,h\in G. 
\]
\end{theorem}

Notice that the conclusion of Theorem \ref{KKtrivcc2} implies 
\begin{align*}
(\gamma\circ(\beta_g\otimes\mu^G_g)\circ\gamma^{-1})(a)
&=(\Ad w\circ\Ad(x_g\otimes1)\circ(\alpha_g\otimes\mu^G_g)\circ\Ad w^*)(a)\\
&=((\Ad w(x_g\otimes1)(\alpha_g\otimes\mu^G_g)(w^*))
\circ(\alpha_g\otimes\mu^G_g))(a)\\
&=(\Ad c_g\circ(\alpha_g\otimes\mu^G_g))(a)
\end{align*}
for all $g\in G$ and $a\in A$. 
Therefore, 
$(\alpha\otimes\mu^G,u\otimes1)$ and $(\beta\otimes\mu^G,v\otimes1)$ 
are $KK$-trivially cocycle conjugate. 

\begin{proof}[Proof of Theorem \ref{KKtrivcc} and Theorem \ref{KKtrivcc2}]
The proof is by induction on the Hirsch length of $G$. 
By \cite[Theorem 5]{Na00ETDS}, Theorem \ref{KKtrivcc} is known for $G=\Z$. 

Let $G$ be a poly-$\Z$ group. 
Assuming that Theorem \ref{KKtrivcc} is known for $G$, 
we would like to prove that 
Theorem \ref{KKtrivcc2} holds for $G$. 

Let $(\alpha,u):G\curvearrowright A$, $(\beta,v):G\curvearrowright A$ 
and $(x_g)_g$ be as in the statement. 
We denote the one point compactification of $[0,\infty)$ 
by $[0,\infty]=[0,\infty)\cup\{\infty\}$. 
Define cocycle actions $(\tilde\alpha,\tilde u)$, $(\tilde\beta,\tilde v)$ 
of $G$ on $C([0,\infty],A\otimes\mathcal{O}_\infty)$ by 
\[
\tilde\alpha_g(f)(s)=(\alpha_g\otimes\mu^G_g)(f(s)),\quad 
\tilde u(g,h)(s)=u(g,h)\otimes1, 
\]
\[
\tilde\beta_g(f)(s)
=\begin{cases}
(\Ad(x_g(s)\otimes1)\circ(\alpha_g\otimes\mu^G_g))(f(s))&s\in[0,\infty)\\
(\beta_g\otimes\mu^G_g)(f(s))&s=\infty\end{cases}
\]
and 
\[
\tilde v(g,h)
=\begin{cases}
(x_g(s)\alpha_g(x_h(s))u(g,h)x_{gh}^*)\otimes1&s\in[0,\infty)\\
v(g,h)\otimes1&s=\infty\end{cases}
\]
It is easy to see that 
$(\tilde\alpha,\tilde u)$ and $(\tilde\beta,\tilde v)$ 
are in $\AC(\mathcal{O}_\infty,\mu^G)$. 
We define a continuous map $\tilde x_g$ 
from $[0,\infty)$ to $U(C([0,\infty],A\otimes\mathcal{O}_\infty))$ by 
\[
\tilde x_g(s,t)=x_g(\min\{s,t\})\otimes1\quad\forall 
(s,t)\in[0,\infty]\times [0,\infty). 
\]
We think of $\tilde x_g$ 
as a unitary in $(C([0,\infty],A\otimes\mathcal{O}_\infty))^\flat$. 
It is not so hard to see 
\[
(\Ad\tilde x_g\circ\tilde\alpha_g)(f)=\tilde\beta_g(f)
\quad\forall f\in C([0,\infty],A\otimes\mathcal{O}_\infty). 
\]
Moreover we can check 
\[
\tilde x_g\tilde\alpha_g(\tilde x_h)\tilde u(g,h)\tilde x_{gh}^*
=\tilde v(g,h)\quad\forall g,h\in G. 
\]
It follows from Theorem \ref{KKtrivcc} that 
there exist an asymptotically inner automorphism 
$\gamma\in\Aut(C([0,\infty],A\otimes\mathcal{O}_\infty))$ and 
a family $(c_g)_g$ of unitaries 
in $C([0,\infty],A\otimes\mathcal{O}_\infty)$ such that 
\[
\Ad c_g\circ\tilde\alpha_g=\gamma\circ\tilde\beta_g\circ\gamma^{-1}
\quad\forall g\in G
\]
and 
\[
c_g\tilde\alpha_g(c_h)\tilde u(g,h)c_{gh}^*=\gamma(\tilde v(g,h))
\quad\forall g,h\in G. 
\]
We write $\gamma=(\gamma_s)_s$ and $c_g=(c_g(s))_s$ for $s\in[0,\infty]$. 
For $s\in[0,\infty)$, we have 
\begin{align*}
&\gamma_0^{-1}\circ\gamma_s\circ\Ad(x_g(s)\otimes1)
\circ(\alpha_g\otimes\mu^G_g)\circ\gamma_s^{-1}\circ\gamma_0\\
&=\gamma_0^{-1}\circ\Ad c_g(s)\circ(\alpha_g\otimes\mu^G_g)\circ\gamma_0\\
&=\Ad\gamma_0^{-1}(c_g(s)c_g(0)^*)
\circ\gamma_0^{-1}\circ\Ad c_g(0)\circ(\alpha_g\otimes\mu^G_g)\circ\gamma_0\\
&=\Ad\gamma_0^{-1}(c_g(s)c_g(0)^*)\circ\Ad(x_g(0)\otimes1)
\circ(\alpha_g\otimes\mu^G_g). 
\end{align*}
Similarly we have 
\[
\gamma_0^{-1}\circ\gamma_\infty\circ(\beta_g\otimes\mu^G_g)
\circ\gamma_\infty^{-1}\circ\gamma_0
=\Ad\gamma_0^{-1}(c_g(\infty)c_g(0)^*)\circ\Ad(x_g(0)\otimes1)
\circ(\alpha_g\otimes\mu^G_g). 
\]
Hence, by replacing $\gamma_s$ and $c_g(s)$ 
with $\gamma_0^{-1}\circ\gamma_s$ and 
$\gamma_0^{-1}(c_g(s)c_g(0)^*)(x_g(0)\otimes1)$, 
we may assume $\gamma_0=\id$ and $c_g(0)=x_g(0)\otimes1$. 
Since $\gamma\in\Aut(C([0,\infty],A\otimes\mathcal{O}_\infty))$ 
is (still) asymptotically inner, 
there exists a continuous map $w$ 
from $[0,\infty]\times[0,\infty)$ to $U(A\otimes\mathcal{O}_\infty)$ 
such that $\gamma=\lim_t\Ad w(\cdot,t)$. 
We may assume $w(0,t)=1$ for all $t\in[0,\infty)$. 
For $(s,t)\in[0,\infty]\times[0,\infty)$, we consider 
\[
y_g(s,t)
=c_g(s)^*w(s,t)\tilde x_g(s,t)(\alpha_g\otimes\mu^G_g)(w(s,t)^*). 
\]
It is not so hard to see 
\[
\lim_{t\to\infty}\max_{s\in[0,\infty]}
\lVert[y_g(s,t),a]\rVert=0
\quad\forall g\in G,\ a\in A\otimes\mathcal{O}_\infty
\]
and $y_g(0,t)=1$. 
In order to show 
\[
\lim_{t\to\infty}\max_{s\in[0,\infty]}
\lVert y_g(s,t)(\alpha_g\otimes\mu^G_g)(y_h(s,t))
-y_{gh}(s,t)\rVert=0\quad\forall g,h\in G, 
\]
we regard $y_g$ as a map 
from $[0,\infty)$ to $U(C([0,\infty],A\otimes\mathcal{O}_\infty))$ and 
write 
\[
y_g(t)=c_g^*w(t)\tilde x_g(t)\tilde\alpha_g(w(t)^*). 
\]
Then we get 
\begin{align*}
y_g(t)\tilde\alpha_g(y_h(t))
&=c_g^*w(t)\tilde x_g(t)\tilde\alpha_g(w(t)^*)
\tilde\alpha_g(c_h^*w(t)\tilde x_h(t)\tilde\alpha_h(w(t)^*))\\
&\approx c_g^*w(t)\tilde x_g(t)\tilde\alpha_g(\gamma^{-1}(c_h^*))
\tilde\alpha_g(\tilde x_h(t)\tilde\alpha_h(w(t)^*))\\
&\approx c_g^*w(t)\tilde\beta_g(\gamma^{-1}(c_h^*))\tilde x_g(t)
\tilde\alpha_g(\tilde x_h(t)\tilde\alpha_h(w(t)^*))\\
&\approx c_g^*\gamma(\tilde\beta_g(\gamma^{-1}(c_h^*)))w(t)
\tilde v(g,h)\tilde x_{gh}(t)\tilde u(g,h)^*
\tilde\alpha_g(\tilde\alpha_h(w(t)^*))\\
&\approx c_g^*c_g\tilde\alpha_g(c_h^*)c_g^*\gamma(\tilde v(g,h))w(t)
\tilde x_{gh}(t)\tilde\alpha_{gh}(w(t)^*)\tilde u(g,h)^*\\
&\approx \tilde u(g,h)c_{gh}^*w(t)
\tilde x_{gh}(t)\tilde\alpha_{gh}(w(t)^*)\tilde u(g,h)^*\\
&\approx y_{gh}(t), 
\end{align*}
when $t$ is sufficiently large, as desired. 
By Proposition \ref{extendedH1stable}, 
we can find a continuous map $z:[0,\infty)\to U(A\otimes\mathcal{O}_\infty)$ 
such that 
\[
\lim_{t\to\infty}\lVert[z(t),a]\rVert=0
\quad\forall a\in A\otimes\mathcal{O}_\infty
\]
and 
\[
\lim_{t\to\infty}
\lVert y_g(\infty,t)-z(t)(\alpha_g\otimes\mu^G_g)(z(t)^*)\rVert=0
\quad\forall g\in G. 
\]
Put $w'(t)=z(t)^*w(\infty,t)$. 
Then one has 
\begin{align*}
&\lim_{t\to\infty}w'(t)(x_g(t)\otimes1)(\alpha_g\otimes\mu^G_g)(w'(t)^*)\\
&=\lim_{t\to\infty}z(t)^*w(\infty,t)(x_g(t)\otimes1)
(\alpha_g\otimes\mu^G_g)(w(\infty,t)^*z(t))\\
&=\lim_{t\to\infty}z(t)^*c_g(\infty)y_g(\infty,t)
(\alpha_g\otimes\mu^G_g)(z(t))\\
&=c_g(\infty)
\end{align*}
and 
\[
\lim_{t\to\infty}\Ad w'(t)(a)=\lim_{t\to\infty}\Ad w(\infty,t)(a)
=\gamma_\infty(a)\quad\forall a\in A\otimes\mathcal{O}_\infty, 
\]
which imply that Theorem \ref{KKtrivcc2} holds for $G$. 

Next, assuming that 
Theorem \ref{KKtrivcc2} is known for any poly-$\Z$ groups 
with Hirsch length less than $l$, we prove that 
Theorem \ref{KKtrivcc} holds for a poly-$\Z$ group $G$ 
whose Hirsch length equals $l$. 
There exists a normal poly-$\Z$ subgroup $N\subset G$ 
whose Hirsch length equals $l-1$. 
Take $\xi\in G$ so that $G$ is generated by $N$ and $\xi$. 

Let $(\alpha,u):G\curvearrowright A$ and $(\beta,v):G\curvearrowright A$ be 
cocycle actions belonging to $\AC(\mathcal{O}_\infty,\mu^G)$. 
Suppose that 
there exists a family $(x_g)_g$ of continuous maps 
from $[0,\infty)$ to $U(A)$ such that 
\[
\lim_{t\to\infty}(\Ad x_g(t)\circ\alpha_g)(a)=\beta_g(a)
\quad\forall g\in G,\ \forall a\in A, 
\]
\[
\lim_{t\to\infty}x_g(t)\alpha_g(x_h(t))u(g,h)x_{gh}(t)^*
=v(g,h)\quad\forall g,h\in G. 
\]
By Theorem \ref{KKtrivcc2} for the poly-$\Z$ group $N$, 
we can find a continuous map $w:[0,\infty)\to U(A\otimes\mathcal{O}_\infty)$, 
$\gamma\in\Aut(A\otimes\mathcal{O}_\infty)$ and 
a family $(c_g)_{g\in N}$ of unitaries in $A\otimes\mathcal{O}_\infty$ 
satisfying 
\[
\lim_{t\to\infty}\Ad w(t)(a)=\gamma(a)
\quad\forall a\in A\otimes\mathcal{O}_\infty, 
\]
\[
\lim_{t\to\infty}w(t)(x_g(t)\otimes1)(\alpha_g\otimes\mu^G_g)(w(t))^*=c_g
\quad\forall g\in N
\]
and 
\[
c_g(\alpha_g\otimes\mu^G_g)(c_h)(u(g,h)\otimes1)c_{gh}^*
=\gamma(v(g,h)\otimes1)\quad\forall g,h\in N. 
\]
Define cocycle actions $(\tilde\alpha,\tilde u)$, $(\tilde\beta,\tilde v)$ 
of $G$ on $A\otimes\mathcal{O}_\infty$ by 
\[
\tilde\alpha_g=\alpha_g\otimes\mu^G_g,\quad \tilde u(g,h)=u(g,h)\otimes1
\]
and 
\[
\tilde\beta_g=\gamma\circ(\beta_g\otimes\mu^G_g)\circ\gamma^{-1},\quad 
\tilde v(g,h)=\gamma(v(g,h)\otimes1). 
\]
We let 
\[
\tilde x_g(t)=w(t)(x_g(t)\otimes1)\tilde\alpha_g(w(t)^*)\quad\forall g\in G
\]
and regard it as an element of $(A\otimes\mathcal{O}_\infty)^\flat$. 
It is routine to verify 
\[
(\Ad\tilde x_g\circ\tilde\alpha_g)(a)=\tilde\beta_g(a)
\quad\forall g\in G,\ a\in A\otimes\mathcal{O}_\infty
\]
and 
\[
\tilde x_g\tilde\alpha_g(\tilde x_h)\tilde u(g,h)\tilde x_{gh}^*=\tilde v(g,h)
\quad\forall g,h\in G. 
\]
By construction, 
we have $\tilde x_g=c_g$ and $\tilde\beta_g=\Ad c_g\circ\tilde\alpha_g$ 
for $g\in N$. 

Let $B_\alpha$ and $B_\beta$ be 
the twisted crossed products of $A\otimes\mathcal{O}_\infty$ 
by $(\tilde\alpha,\tilde u)$ and $(\tilde\beta,\tilde v)$, respectively. 
We denote the implementing unitary representations of $G$ 
in $B_\alpha$ and $B_\beta$ 
by $(\lambda^\alpha_g)_g$ and $(\lambda^\beta_g)_g$, respectively. 
We can define a homomorphism $\pi:B_\beta\to(B_\alpha)^\flat$ by 
\[
\pi(a)=a\quad\forall a\in A\otimes\mathcal{O}_\infty
\quad\text{and}\quad 
\pi(\lambda^\beta_g)=\tilde x_g\lambda^\alpha_g\quad\forall g\in G. 
\]
Let $C_\alpha\subset B_\alpha$ be the subalgebra 
generated by $A\otimes\mathcal{O}_\infty$ and 
$\{\lambda^\alpha_g\mid g\in N\}$. 
The $C^*$-algebra $C_\alpha$ is canonically isomorphic to 
the twisted crossed product of $A\otimes\mathcal{O}_\infty$ 
by the restriction of $(\tilde\alpha,\tilde u)$ to $N$. 
In the same way, we define $C_\beta\subset B_\beta$. 
Then we have $\pi(C_\beta)=C_\alpha$ 
because $\tilde x_g=c_g$ holds for $g\in N$. 
Furthermore, for any $z\in B_\beta$, it is easy to see 
\[
(\Ad\tilde x_\xi\circ\Ad\lambda^\alpha_\xi\circ\pi)(z)
=(\Ad\pi(\lambda^\beta_\xi)\circ\pi)(z)
=(\pi\circ\Ad\lambda^\beta_\xi)(z). 
\]
In particular, for any $z\in C_\alpha$, 
\[
(\Ad\tilde x_\xi\circ\Ad\lambda^\alpha_\xi)(z)
=(\pi\circ\Ad\lambda^\beta_\xi\circ\pi^{-1})(z), 
\]
which means that 
$\Ad\lambda^\alpha_\xi$ and $\pi\circ\Ad\lambda^\beta_\xi\circ\pi^{-1}$ are 
asymptotically unitarily equivalent in $\Aut(C_\alpha)$ 
by $\tilde x_\xi:[0,\infty)\to U(A\otimes\mathcal{O}_\infty)$. 
Then, in the same way as Theorem \ref{equivNakamura} we can show that 
$(\tilde\alpha,\tilde u)$ and $(\tilde\beta,\tilde v)$ are 
cocycle conjugate via an asymptotically inner automorphism. 
Thanks to Theorem \ref{McDuff}, 
we can conclude that $(\alpha,u)$ and $(\beta,v)$ are cocycle conjugate 
via an asymptotically inner automorphism. 
Thus, Theorem \ref{KKtrivcc} is true for $G$. 
\end{proof}

\begin{corollary}
Let $G$ be a poly-$\Z$ group and 
let $A$ be a unital separable $C^*$-algebra. 
Suppose that $(\alpha,u):G\curvearrowright A$ is a cocycle action 
in $\AC(\mathcal{O}_\infty,\mu^G)$ and that 
there exists a family of unitaries $(x_g)_g$ in $A^\flat$ satisfying 
\[
\alpha_g(a)=\Ad x_g(a),\quad 
x_gx_hx_{gh}^*=u(g,h)\quad\forall a\in A,\ g,h\in G. 
\]
Then, $(\alpha,u)$ and 
$\id\otimes\mu^G:G\curvearrowright A\otimes\mathcal{O}_\infty$ are 
cocycle conjugate 
via an isomorphism asymptotically unitarily equivalent to 
the embedding $a\mapsto a\otimes1$. 

In particular, for every unital Kirchberg algebra $A$, 
all asymptotically representable outer cocycle actions of $G$ on $A$ are 
mutually $KK$-trivially cocycle conjugate. 
\end{corollary}
\begin{proof}
By Lemma \ref{AC'} (1), 
$(\alpha,u)$ is cocycle conjugate to 
$(\alpha\otimes\mu^G,u\otimes1):G\curvearrowright A\otimes\mathcal{O}_\infty$ 
via an isomorphism asymptotically unitarily equivalent to 
the embedding $a\mapsto a\otimes1$. 
One has 
\[
((\Ad x_g\otimes1)\circ(\id\otimes\mu^G_g))(a)=(\alpha_g\otimes\mu^G_g)(a)
\quad\forall a\in A\otimes\mathcal{O}_\infty,\ g\in G
\]
and 
\[
(x_g\otimes1)(\id\otimes\mu^G_g)(x_h\otimes1)(x_{gh}\otimes1)^*
=u(g,h)\otimes1\quad\forall g,h\in G. 
\]
It follows from Theorem \ref{KKtrivcc} that 
$(\alpha\otimes\mu^G,u\otimes1)$ and $(\id\otimes\mu^G,1\otimes1)$ are 
cocycle conjugate via an asymptotically inner automorphism. 
This completes the proof. 
\end{proof}

\section{Poly-$\Z$ groups of Hirsch length two}

For every poly-$\Z$ group $G$, 
we choose and fix an outer action 
$\mu^G:G\curvearrowright\mathcal{O}_\infty$. 

Let $A$ be a unital (not-necessarily separable) $C^*$-algebra. 
Suppose that the trivial action $\{1\}\curvearrowright A$ 
admits an approximately central embedding of 
$\{1\}\curvearrowright\mathcal{O}_\infty$. 
The following are well-known facts, 
which will be used repeatedly without mention. 
\begin{enumerate}
\item For any $x\in K_0(A)$, 
there exists a full and properly infinite projection $p\in A$ 
with $K_0(p)=x$. 
If $p,q$ are full and properly infinite projections 
with the same $K_0$-class in $K_0(A)$, then 
they are Murray-von Neumann equivalent. 
\item For any ideal $J\subset A$, 
the canonical map $U(J)/U(J)_0\to K_1(J)$ is an isomorphism. 
Moreover, for any $u\in U(J)_0$, 
there exists a continuous map $\tilde u:[0,1]\to U(J)_0$ 
such that $\tilde u(0)=1$, $\tilde u(1)=u$ and $\Lip(\tilde u)\leq2\pi$. 
\end{enumerate}
In fact, (1) follows from \cite[Proposition 4.1.4]{Ro_text} 
(which was originally proved by J. Cuntz \cite{Cu81Annals}), 
because $1$ is a full and properly infinite projection. 
(2) follows from the proof of \cite[Theorem 3.1]{Ph02JOT}. 

Let $A$ be a unital Kirchberg algebra. 
In \cite[Corollary 2.8]{IMweakhomot}, it is shown that 
$K_i(A_\flat)$ is isomorphic to $KK^i(A,A)$ for $i=0,1$. 
In what follows, we identify these groups. 

\begin{lemma}\label{auto_flat}
For $\alpha\in\Aut(A)$, the following hold. 
\begin{enumerate}
\item Under the identification of $K_0(A_\flat)$ with $KK(A,A)$, 
$K_0(\alpha|A_\flat)$ corresponds to 
the homomorphism $x\mapsto KK(\alpha)\circ x\circ KK(\alpha)^{-1}$. 
\item Under the identification of $K_1(A_\flat)$ with $KK(SA,A)$, 
$K_1(\alpha|A_\flat)$ corresponds to 
the homomorphism $x\mapsto KK(\alpha)\circ x\circ KK(S\alpha)^{-1}$. 
\item If $\alpha$ is asymptotically inner and 
$u\in U(A^\flat)$ satisfies $\alpha=\Ad u$ on $A$, then 
$K_i(\Ad u|A_\flat)=K_i(\alpha|A_\flat)$ for $i=0,1$. 
\end{enumerate}
\end{lemma}
\begin{proof}
(1) and (2) are obvious. See \cite{IMweakhomot}. 
To show (3), take a projection $p\in A_\flat$. 
There exists a unitary $v\in A^\flat$ 
such that $\alpha(a)=vav^*$ holds for every $a\in A\cup\{p\}$. 
Because $uv^*$ is in $A_\flat$, we get 
\[
K_0(upu^*)=K_0(uv^*vpv^*vu^*)=K_0(vpv^*)=K_0(\alpha(p)), 
\]
and so $K_0(\Ad u|A_\flat)=K_0(\alpha|A_\flat)$. 
In the same way we obtain $K_1(\Ad u|A_\flat)=K_1(\alpha|A_\flat)$. 
\end{proof}

Let $p\in A$ and $q\in A_\flat$ be projections and 
let $\tilde q:[0,\infty)\to A$ be a lift of $q$. 
When $t$ is large enough, $p\tilde q(t)$ is close to a projection 
and its $K_0$-class does not depend on $t$. 
This correspondence gives rise to a homomorphism $K_0(A_\flat)\to K_0(A)$, 
which we denote by $p_*$. 
In the same fashion, 
we get the homomorphism $p_*:K_1(A_\flat)\to K_1(A)$. 
By a slight abuse of notation, we will use $p_*$ 
to denote induced homomorphisms $H^n(G,K_i(A_\flat))\to H^n(G,K_i(A))$.

\subsection{Uniqueness}

In this subsection, we determine 
when outer (cocycle) actions of poly-$\Z$ groups of Hirsch length two 
are mutually $KK$-trivially cocycle conjugate (Theorem \ref{uni_Hirsch2}). 

\begin{lemma}
Let $\alpha:\Z\curvearrowright A$ be an action 
in $\AC(\mathcal{O}_\infty,\mu^\Z)$ 
and let $J\subset A$ be a globally $\alpha$-invariant ideal. 
Let $u\in U(J)$ be a unitary satisfying $K_1(u)=0$ in $K_1(J)$. 
For any $\ep>0$, there exists $v\in U(J)_0$ 
such that $\lVert u-v\alpha(v^*)\rVert<\ep$, 
where the $\Z$-action $\alpha$ is identified with a single automorphism. 
\end{lemma}
\begin{proof}
Let $w\in U(J)$ be a unitary satisfying $K_1(w)=0$ in $K_1(J)$. 
By the fact mentioned above, there exists a continuous map 
$\tilde w:[0,1]\to U(J)_0$ such that 
$\tilde w(0)=1$, $\tilde w(1)=w$ and $\Lip(\tilde w)\leq2\pi$. 
Then, by using this, one can prove the statement 
in the same way as \cite[Lemma 8]{Na00ETDS}. 
\end{proof}

\begin{lemma}\label{2cv_Hirsch2}
Let $G$ be a poly-$\Z$ group of Hirsch length two and 
let $(\alpha,u):G\curvearrowright A$ be a cocycle action 
belonging to $\AC(\mathcal{O}_\infty,\mu^G)$. 
Suppose that $J\subset A$ is a globally $\alpha$-invariant ideal. 
If $u(g,h)$ belongs to $U(J)_0$ for all $g,h\in G$, then 
there exists a family of unitaries $(v_g)_{g\in G}$ in $U(J)_0$ 
such that $u(g,h)=\alpha_g(v_h^*)v_g^*v_{gh}$ for all $g,h\in G$. 
\end{lemma}
\begin{proof}
There exists a normal subgroup $N\subset G$ isomorphic to $\Z$ 
and $\xi\in G$ such that $G$ is generated by $N$ and $\xi$. 
Because $N$ is isomorphic to $\Z$, by a cocycle perturbation, 
we may assume that $u(g,h)=1$ for all $g,h\in N$. 
It follows from Lemma \ref{check} that the unitaries 
\[
\check u_g=u(\xi,\xi^{-1}g\xi)u(g,\xi)^*\in U(J)_0
\]
form an $\alpha|N$-cocycle satisfying 
\[
\alpha_\xi\circ\alpha_{\xi^{-1}g\xi}\circ\alpha_\xi^{-1}
=\Ad\check u_g\circ\alpha_g\quad\forall g\in N. 
\]
Then, the lemma above tells us that 
the $\alpha|N$-cocycle $(\check u_g)_g$ can be approximated by coboundaries. 
Therefore, by a suitable perturbation, 
we may further assume that $u(g,h)$ is close to $1$ 
on a finite generating subset of $G$. 
Then, by the $H^2$-stability of $G$ (Theorem \ref{H1H2stable}), 
we can conclude that $(u(g,h))_{g,h\in G}$ is a coboundary. 
\end{proof}

\begin{lemma}\label{realize:1c}
Let $G$ be a poly-$\Z$ group of Hirsch length two and 
let $\alpha:G\curvearrowright A$ be an action 
belonging to $\AC(\mathcal{O}_\infty,\mu^G)$. 
Suppose that $J\subset A$ is a globally $\alpha$-invariant ideal. 
For any $1$-cocycle $\rho:G\to K_1(J)$, 
there exists an $\alpha$-cocycle $(u_g)_{g\in G}$ in $U(J)$ 
such that $K_1(u_g)=\rho(g)$ for all $g\in G$. 
\end{lemma}
\begin{proof}
Choose unitaries $w_g\in U(J)$ 
so that $K_1(w_g)=\rho(g)$ for all $g\in G$. 
Consider the cocycle action $(\alpha^w,1^w):G\curvearrowright A$. 
As $1^w(g,h)=w_g\alpha_g(w_h)w_{gh}^*$ is in $U(J)_0$, 
Lemma \ref{2cv_Hirsch2} applies and yields $(v_g)_{g\in G}$ in $U(J)_0$ 
such that $1^w(g,h)=\alpha^w_g(v_h^*)v_g^*v_{gh}$ for all $g,h\in G$. 
Hence $(v_gw_g)_g$ is an $\alpha$-cocycle 
satisfying $K_1(v_gw_g)=\rho(g)$. 
\end{proof}

We introduce the invariant $\kappa^2(\alpha,u)$ 
for an $\alpha$-cocycle $(u_g)_g$ as follows. 
Let $\alpha:G\curvearrowright A$ be an action of a discrete group $G$ 
on a unital $C^*$-algebra $A$. 
Let $(u_g)_{g\in G}$ be an $\alpha$-cocycle with $u_g\in U(A)_0$. 
We choose a continuous path $\tilde u_g:[0,1]\to U(A)$ from $1$ to $u_g$. 
By abuse of notation, 
$\id\otimes\alpha_g\in\Aut(C([0,1])\otimes A)$ is 
simply denoted by $\alpha_g$. 
Then
\[
\rho(g,h)
=K_1(\tilde u_g\alpha_g(\tilde u_h)\tilde u_{gh}^*)\in K_1(SA)=K_0(A), 
\]
and they form a $2$-cocycle, thanks to the next lemma. 
We denote by $\kappa^2(\alpha,u)$ its cohomology class in $H^2(G,K_0(A))$, 

\begin{lemma}\label{kappa2}
In the setting above, $\rho$ is a $2$-cocycle, 
and its cohomology class does not depend on the choice of 
the continuous paths $(\tilde u_g)_{g\in G}$. 
\end{lemma}
\begin{proof}
For continuous maps $v,w:[0,1]\to U(A)_0$, 
it is easy to verify the following. 
\begin{itemize}
\item If $v(0)=1$ and $w(1)=1$, then the two paths $vw$ and $wv$ are homotopic 
within the paths from $w(0)$ to $v(1)$. 
\item If $v(0)=1$ and $v(1)=w(1)^*$, then 
the two paths $vw$ and $wv$ are homotopic within the paths from $w(0)$ to $1$. 
\end{itemize}
First, we show that $\rho$ satisfies the $2$-cocycle relation. 
For $g,h,k\in G$, we have 
\begin{align*}
& g\cdot\rho(h,k)-\rho(gh,k)+\rho(g,hk)-\rho(g,h)\\
&=K_1(\alpha_g(\tilde u_{hk}^*\tilde u_h\alpha_h(\tilde u_k)))
-K_1(\tilde u_{ghk}^*\tilde u_{gh}\alpha_{gh}(\tilde u_k))
+\rho(g,hk)-\rho(g,h)\\
&=K_1(\alpha_g(\tilde u_{hk}^*\tilde u_h)\tilde u_{gh}^*\tilde u_{ghk})
+K_1(\tilde u_g\alpha_g(\tilde u_{hk})\tilde u_{ghk}^*)-\rho(g,h)\\
&=K_1(\alpha_g(\tilde u_h)\tilde u_{gh}^*\tilde u_{ghk}\alpha_g(\tilde u_{hk}^*))
+K_1(\alpha_g(\tilde u_{hk})\tilde u_{ghk}^*\tilde u_g)-\rho(g,h)\\
&=K_1(\alpha_g(\tilde u_h)\tilde u_{gh}^*\tilde u_g)-\rho(g,h)=0, 
\end{align*}
and so $\rho$ is a $2$-cocycle. 

When $(\hat u_g)_g$ is another family of paths from $1$ to $u_g$ in $U(A)_0$, 
one has 
\begin{align*}
& K_1(\tilde u_g\alpha_g(\tilde u_h)\tilde u_{gh}^*)
-K_1(\hat u_g\alpha_g(\hat u_h)\hat u_{gh}^*)\\
&=K_1(\tilde u_g\alpha_g(\tilde u_h)\tilde u_{gh}^*
\hat u_{gh}\alpha_g(\hat u_h^*)\hat u_g^*)\\
&=K_1(\alpha_g(\tilde u_h)\tilde u_{gh}^*
\hat u_{gh}\alpha_g(\hat u_h^*))+K_1(\hat u_g^*\tilde u_g)\\
&=K_1(\tilde u_{gh}^*\hat u_{gh})+K_1(\alpha_g(\hat u_h^*\tilde u_h))
+K_1(\hat u_g^*\tilde u_g)\\
&=g\cdot K_1(\hat u_h^*\tilde u_h)-K_1(\hat u_{gh}^*\tilde u_{gh})
+K_1(\hat u_g^*\tilde u_g), 
\end{align*}
which is a coboundary. 
\end{proof}

The following lemma can be shown easily. 

\begin{lemma}\label{kappa2additive}
In the setting above, 
for any $\alpha^u$-cocycle $(v_g)_g$ in $U(A)_0$, 
we have $\kappa^2(\alpha,vu)=\kappa^2(\alpha,u)+\kappa^2(\alpha^u,v)$. 
\end{lemma}

\begin{remark}
When $G=\Z^2$ and its action on $K_0(A)$ is trivial, 
the cohomology class $\kappa^2(\alpha,u)$ is determined by 
\[
\rho((1,0),(0,1))-\rho((0,1),(1,0))
=K_1\left(\tilde u_{(1,0)}\alpha_{(1,0)}(\tilde u_{(0,1)})
\alpha_{(0,1)}(\tilde u_{(1,0)}^*)\tilde u_{(0,1)}^*\right), 
\]
which is the $\kappa$-invariant discussed in \cite{KM08Adv}. 
\end{remark}

\begin{lemma}\label{1cv_Hirsch2}
Let $G$ be a poly-$\Z$ group of Hirsch length two and 
let $A$ be a unital $C^*$-algebra. 
Let $\alpha:G\curvearrowright A$ be an action 
in $\AC(\mathcal{O}_\infty,\mu^G)$. 
If $(u_g)_{g\in G}$ is an $\alpha$-cocycle 
with $u_g\in U(A)_0$ and $\kappa^2(\alpha,u)=0$, then 
$(u_g)_{g\in G}$ can be approximated by coboundaries. 
Thus, there exists a sequence of unitaries $(v_n)_n$ in $U(A)_0$ 
such that 
\[
\lim_{n\to\infty}v_n\alpha_g(v_n^*)=u_g\quad\forall g\in G. 
\]
\end{lemma}
\begin{proof}
Since $\kappa^2(\alpha,u)=0$, 
we can choose continuous paths $\tilde u_g:[0,1]\to U(A)$ from $1$ to $u_g$ 
so that 
\[
K_1(\tilde u_g\alpha_g(\tilde u_h)\tilde u_{gh}^*)=0\quad\forall g,h\in G. 
\]
Consider 
\[
\tilde\alpha_g=(\Ad\tilde u_g)\circ\alpha_g
\in\Aut(C([0,1])\otimes A)
\]
and 
\[
w(g,h)=\tilde u_g\alpha_g(\tilde u_h)\tilde u_{gh}^*
\in U(C_0((0,1))\otimes A). 
\]
Then, $(\tilde\alpha,w)$ is a cocycle action on $C([0,1])\otimes A$ 
and $K_1(w(g,h))=0$, 
i.e.\ $w(g,h)\in U(C_0((0,1))\otimes A)_0$ for all $g,h\in G$. 
It follows from Lemma \ref{2cv_Hirsch2} that 
$(w(g,h))_{g,h}$ is a coboundary. 
Hence, we may assume $w(g,h)=1$ for every $g,h\in G$. 
Thus, $(\tilde u_g)_g$ is an $\id\otimes\alpha$-cocycle 
such that $\tilde u_g(0)=1$. 
We get the conclusion from Lemma \ref{1cv_G_homo1}. 
\end{proof}

Let $(\alpha,u)$ be a cocycle action of a countable discrete group $G$ 
on a unital Kirchberg algebra $A$. 
Let $\beta$ be an action of $G$ on $A$. 
Assume $KK(\alpha_g)=KK(\beta_g)$. 
We choose a family $(v_g)_{g\in G}$ of unitaries in $A^\flat$ 
satisfying $(\Ad v_g\circ\alpha_g)(a)=\beta_g(a)$ for every $a\in A$. 
Define a cocycle action $(\sigma,w):G\curvearrowright A_\flat$ 
by $\sigma_g=\Ad v_g\circ\alpha_g$ and 
$w(g,h)=v_g\alpha_g(v_h)u(g,h)v_{gh}^*$. 

\begin{lemma}\label{key}
In the setting above, assume further that $G$ is a poly-$\Z$ group 
and that $(\alpha,u)$ and $\beta$ are outer. 
The following are equivalent. 
\begin{enumerate}
\item $(\alpha,u)$ and $\beta$ are $KK$-trivially cocycle conjugate. 
\item The $2$-cocycle $(w(g,h))_{g,h\in G}$ in $A_\flat$ is 
a coboundary. 
\end{enumerate}
\end{lemma}
\begin{proof}
(1)$\Rightarrow$(2) is obvious. 
Indeed, 
this is true without the assumption about $G$, $(\alpha,u)$ and $\beta$. 

Let us show the converse. 
Since $(w(g,h))_{g,h\in G}$ is a coboundary, 
by replacing the unitaries $(v_g)_g$, 
we may assume $w(g,h)=1$ for all $g,h\in G$. 
Thus, the hypothesis of Theorem \ref{KKtrivcc} is satisfied. 
Hence $(\alpha,u)$ and $\beta$ are $KK$-trivially cocycle conjugate. 
\end{proof}

\begin{definition}
Let $(\alpha,u):G\curvearrowright A$, $\beta:G\curvearrowright A$ and 
$(\sigma,w):G\curvearrowright A_\flat$ be as above. 
We denote by 
\[
\mathfrak{o}^2((\alpha,u),\beta)
\in H^2(G,K_1(A_\flat))=H^2(G,KK^1(A,A))
\]
the cohomology class of the 2-cocycle $(g,h)\mapsto K_1(w(g,h))$. 
Then $\mathfrak{o}^2((\alpha,u),\beta)$ does not depend 
on the choice of the unitaries $(v_g)_g$. 
When $\alpha$ is a genuine action, 
we write $\mathfrak{o}^2(\alpha,\beta)=\mathfrak{o}^2((\alpha,1),\beta)$. 
\end{definition}

It is easy to see that 
$\mathfrak{o}^2((\alpha^x,u^x),\beta)
=\mathfrak{o}^2((\alpha,u),\beta)$ 
holds for any family $(x_g)_g$ of unitaries in $A$. 
Moreover, 
when a projection $p\in A$ satisfies $\beta_g(p)=p$ for every $g\in G$, 
it is straightforward to see that 
$\mathfrak{o}^2((\alpha,u)^p,\beta^p)$ equals 
$\mathfrak{o}^2((\alpha,u),\beta)$ 
under the identification of $KK^1(pAp,pAp)$ with $KK^1(A,A)$. 
When $(v_g(t))_{t\in[0,\infty)}$ in $U(C^b([0,\infty),A))$ are lifts of $v_g$, 
the unitary $w(g,h)\in A_\flat$ is represented by 
\[
w(g,h)(t)=v_g(t)\alpha_g(v_h(t))u(g,h)v_{gh}(t)^*, 
\]
whose $K_1$-class in $K_1(A)$ equals 
\[
K_1(u(g,h))+K_1(v_g(t))+K_1(\alpha_g(v_h(t)))-K_1(v_{gh}(t)). 
\]
This means that 
the homomorphism $1_*:H^2(G,K_1(A_\flat))\to H^2(G,K_1(A))$ sends 
$\mathfrak{o}^2((\alpha,u),\beta)$ to the cohomology class of 
$(g,h)\mapsto K_1(u(g,h))$. 
We also remark that $K_1(A_\flat)$ is isomorphic to $KK^1(A,A)$ 
(\cite[Corollary 2.8]{IMweakhomot}), 
and $K_i(\alpha_g|A_\flat)=K_i(\beta_g|A_\flat)=K_i(\sigma_g)$ 
for $i=0,1$ and $g\in G$ by Lemma \ref{auto_flat}. 
For genuine actions, the chain rule 
\[
\mathfrak{o}^2(\alpha,\beta)+\mathfrak{o}^2(\beta,\gamma)
=\mathfrak{o}^2(\alpha,\gamma)
\]
holds. 

\begin{theorem}\label{uni_Hirsch2}
Let $A$ be a unital Kirchberg algebra and 
let $G$ be a poly-$\Z$ group of Hirsch length two. 
Let $(\alpha,u):G\curvearrowright A$ be an outer cocycle action 
and let $\beta:G\curvearrowright A$ be an outer action. 
The following are equivalent. 
\begin{enumerate}
\item $(\alpha,u)$ and $\beta$ are $KK$-trivially cocycle conjugate. 
\item $KK(\alpha_g)=KK(\beta_g)$ for all $g\in G$ and 
$\mathfrak{o}^2((\alpha,u),\beta)=0$. 
\end{enumerate}
\end{theorem}
\begin{proof}
(1)$\Rightarrow$(2) is obvious from Lemma \ref{key}. 

Let us show the converse. 
Define the cocycle action $(\sigma,w):G\curvearrowright A_\flat$ as above. 
By $\mathfrak{o}^2((\alpha,u),\beta)=0$, 
we may assume $K_1(w(g,h))=0$ in $K_1(A_\flat)$ for every $g,h\in G$. 
It follows from Lemma \ref{2cv_Hirsch2} and Lemma \ref{key} that 
$(\alpha,u)$ and $\beta$ are $KK$-trivially cocycle conjugate. 
\end{proof}

\subsection{The relationship between $H^*(N,M)$ and $H^*(N\rtimes\Z,M)$}

In this subsection, we collect several statements 
which will be used in Section 7.3 and Section 8. 
Throughout this subsection, 
we let $G$ be a countable group and 
let $N\subset G$ be a normal subgroup such that $G/N\cong\Z$. 
Take $\xi\in G$ so that $G$ is generated by $N$ and $\xi$. 

First, let us recall a few basic facts 
about group cohomology of semidirect products by $\Z$. 
Let $M$ be a left $G$-module. 
For each $n\in\N$, 
the Lyndon-Hochschild-Serre spectral sequence 
(see \cite[Chapter XI.10]{MR1344215} for instance) 
gives the short exact sequence 
\begin{equation}
\begin{CD} 
0@>>>H^1(\Z,H^{n-1}(N,M))@>j>>H^n(G,M)@>q>>H^n(N,M)^\Z@>>>0, 
\end{CD}
\label{LHS}
\end{equation}
where $q$ is the restriction map. 
Let $\omega:N^n\to M$ be an $n$-cocycle, that is, 
\[
g_0\omega(g_1,\dots,g_n)
+\sum_{i=1}^n(-1)^i\omega(g_0,\dots,g_{i-1}g_i,\dots,g_n)
+(-1)^{n+1}\omega(g_0,\dots,g_{n-1})=0
\]
holds for any $(g_0,g_1,\dots,g_n)\in G^{n+1}$. 
For $\l\in\Z$, define an $n$-cocycle $\xi^l\omega:N^n\to M$ by 
\[
(\xi^l\omega)(g_1,g_2,\dots,g_n)
=\xi^l\omega(\xi^{-l}g_1\xi^l,\xi^{-l}g_2\xi^l,\dots,\xi^{-l}g_n\xi^l). 
\]
Then, in \eqref{LHS}, 
the $\Z$-action on $H^n(N,M)$ is given by $\xi^l[\omega]=[\xi^l\omega]$. 

We would like to write the homomorphism 
$j:H^1(\Z,H^{n-1}(N,M))\to H^n(G,M)$ explicitly. 
Let $\bar j:H^{n-1}(N,M)\to H^n(G,M)$ be the composition of 
$H^{n-1}(N,M)\to H^1(\Z,H^{n-1}(N,M))$ and $j$. 
Suppose that an $(n{-}1)$-cocycle $\rho:N^{n-1}\to M$ is given. 
We let $(\rho_l)_{l\in\Z}$ be the family of $(n{-}1)$-cocycles 
satisfying $\rho_0=0$, $\rho_1=\rho$ and 
\[
\rho_{l+m}=\rho_l+\xi^l\rho_m\quad\forall l,m\in\Z. 
\]
Define $\omega:G^n\to M$ by 
\begin{align}
&\omega(g_1\xi^{l_1},g_2\xi^{l_2},\dots,g_n\xi^{l_n})\notag\\
&=(-1)^ng_1\rho_{l_1}(\xi^{l_1}g_2\xi^{-l_1},\xi^{l_1+l_2}g_3\xi^{-l_1-l_2},
\dots,\xi^{l_1+l_2+\dots+l_{n-1}}g_n\xi^{-l_1-l_2-\dots-l_{n-1}}). 
\label{defofomega}
\end{align}
Then one can verify the following easily. 

\begin{lemma}\label{barj}
The map $\omega$ is an $n$-cocycle and 
$\bar j([\rho])$ equals $[\omega]$. 
\end{lemma}

Now we turn to group actions on $C^*$-algebras. 
Let $G$, $N$ and $\xi$ be as above. 

\begin{lemma}\label{realize:ka2}
Suppose that $N$ is a poly-$\Z$ group. 
Let $\alpha:G\curvearrowright A$ be an action 
in $\AC(\mathcal{O}_\infty,\mu^G)$. 
Let $(u_g)_{g\in N}$ be an $\alpha|N$-cocycle in $U(SA)$ and 
let $c\in H^1(N,K_0(A))$ be the cohomology class of $g\mapsto K_1(u_g)$. 
Then there exists an $\alpha$-cocycle $(v_g)_{g\in G}$ in $U(A)_0$ 
such that $\bar j(c)=\kappa^2(\alpha,v)$, 
where $\bar j:H^1(N,K_0(A))\to H^2(G,K_0(A))$ is the homomorphism 
introduced above. 
\end{lemma}
\begin{proof}
Define a $2$-cocycle $\omega:G^2\to K_1(SA)=K_0(A)$ by 
\[
\omega(g,h\xi^m)=0\quad\text{and}\quad 
\omega(g\xi,h\xi^m)=g\cdot K_1(u_{\xi h\xi^{-1}})
\quad\forall g,h\in N,\ m\in\Z. 
\]
By Lemma \ref{barj}, $\bar j(c)$ equals $[\omega]$. 

By Lemma \ref{1cv_G_homo1}, 
there exists a sequence $(x_k)_{k\in\N}$ of unitaries 
in $U(C_0((0,1])\otimes A)$ such that 
\[
\lim_{k\to\infty}x_k\alpha_g(x_k^*)=u_g\quad\forall g\in N. 
\]
In particular, 
one has $x_k(1)\alpha_g(x_k(1)^*)\to1$ as $k\to\infty$ for any $g\in N$. 
For each $k\in\N$, we can construct a family of unitaries 
$(y_{k,l})_{l\in\Z}$ in $U(C_0((0,1])\otimes A)$ 
satisfying $y_{k,0}=1$, $y_{k,1}=x_k$ and 
$y_{k,l}\alpha_{\xi^l}(y_{k,m})=y_{k,l+m}$ for every $l,m\in\Z$. 
For $k\in\N$, $g\in N$ and $l\in\Z$, 
we let $z_{k,g\xi^l}=\alpha_g(y_{k,l})\in U(C_0((0,1])\otimes A)$. 
Then, for any $g\xi^l,h\xi^m\in G$, we get 
\begin{align*}
z_{k,g\xi^l}\alpha_{g\xi^l}(z_{k,h\xi^m})z_{k,g\xi^lh\xi^m}^*
&=\alpha_g(y_{k,l})\alpha_{g\xi^l}\left(\alpha_h(y_{k,m})\right)
\alpha_{g\xi^lh\xi^{-l}}(y_{k,l+m}^*)\\
&=\alpha_g\left(y_{k,l}\alpha_{\xi^lh\xi^{-l}}(\alpha_{\xi^l}(y_{k,m}))\right)
\alpha_{g\xi^lh\xi^{-l}}(y_{k,l+m}^*)\\
&=\alpha_g\left(y_{k,l}\alpha_{\xi^lh\xi^{-l}}(y_{k,l}^*)\right), 
\end{align*}
and hence 
\[
\lim_{k\to\infty}
z_{k,g\xi^l}(1)\alpha_{g\xi^l}(z_{k,h\xi^m}(1))z_{k,g\xi^lh\xi^m}(1)^*
=\lim_{k\to\infty}
\alpha_g\left(y_{k,l}(1)\alpha_{\xi^lh\xi^{-l}}(y_{k,l}(1)^*)\right)=1. 
\]
Since $G$ is $H^2$-stable by Theorem \ref{H1H2stable}, 
there exists a sequence of $\alpha$-cocycles $(v_{k,g})_g$ in $U(A)_0$ 
such that 
\[
\lim_{k\to\infty}\lVert v_{k,g}-z_{k,g}(1)\rVert=0\quad\forall g\in G. 
\]
In order to compute $\kappa^2(\alpha,v_k)$, 
we choose a path of unitaries $\tilde v_{k,g}$ so that 
$\tilde v_{k,g}(0)=1$, $\tilde v_{k,g}(1)=v_{k,g}$ and 
\[
\lim_{k\to\infty}
\sup_{t\in[0,1]}\lVert\tilde v_{k,g}(t)-z_{k,g}(t)\rVert=0. 
\]
For $g\xi^l,h\xi^m\in G$, when $k$ is sufficiently large, 
\[
\tilde v_{k,g\xi^l}\alpha_{g\xi^l}(\tilde v_{k,h\xi^m})
\tilde v_{k,g\xi^lh\xi^m}^*
\approx z_{k,g\xi^l}\alpha_{g\xi^l}(z_{k,h\xi^m})z_{k,g\xi^lh\xi^m}^*
=\alpha_g\left(y_{k,l}\alpha_{\xi^lh\xi^{-l}}(y_{k,l}^*)\right). 
\]
When $l=0$, 
$y_{k,l}\alpha_{\xi^lh\xi^{-l}}(y_{k,l}^*)$ equals $1$. 
Moreover, when $\l=1$, 
\[
\tilde v_{k,g\xi}\alpha_{g\xi}(\tilde v_{k,h\xi^m})\tilde v_{k,g\xi h\xi^m}^*
\approx\alpha_g(y_{k,1}\alpha_{\xi h\xi^{-1}}(y_{k,1}^*))
\approx\alpha_g(u_{\xi h\xi^{-1}}). 
\]
Therefore, 
$\kappa^2(\alpha,v_k)$ equals $\bar j(c)$ for sufficiently large $k$. 
\end{proof}

\begin{lemma}\label{realize:ka2'}
Suppose that $N$ is $\Z$ 
(i.e.\ $G$ is a poly-$\Z$ group of Hirsch length two). 
Let $\alpha:G\curvearrowright A$ be an action 
in $\AC(\mathcal{O}_\infty,\mu^G)$. 
Then for any $c\in H^2(G,K_0(A))$, 
there exists an $\alpha$-cocycle $(v_g)_{g\in G}$ in $U(A)_0$ 
such that $\kappa^2(\alpha,v)=c$. 
\end{lemma}
\begin{proof}
By \eqref{LHS}, 
$j:H^1(\Z,H^1(N,K_0(A)))\to H^2(G,K_0(A))$ is an isomorphism. 
Hence there exists a $1$-cocycle $\rho:N\to K_0(A)=K_1(SA)$ 
such that $\bar j([\rho])=c$. 
Since $N$ is $\Z$, 
we can find an $\alpha|N$-cocycle $(u_g)_{g\in N}$ in $U(SA)$ 
such that $K_1(u_g)=\rho(g)$. 
By the lemma above, 
there exists an $\alpha$-cocycle $(v_g)_{g\in G}$ in $A$ 
such that $\kappa^2(\alpha,v)=\bar j([\rho])=c$, 
which completes the proof. 
\end{proof}

\begin{lemma}\label{realization}
Suppose that $N$ is a poly-$\Z$ group. 
Let $\alpha:G\curvearrowright A$ be an outer action of $G$ 
on a unital Kirchberg algebra $A$. 
Let $(u_g)_{g\in N}$ be an $\alpha|N$-cocycle in $A_\flat$. 
Then there exist a cocycle action 
$(\beta,v):G\curvearrowright A\otimes\mathcal{O}_\infty$ 
and a family $(x_g)_{g\in G}$ of unitaries 
in $(A\otimes\mathcal{O}_\infty)^\flat$ 
satisfying the following. 
\begin{enumerate}
\item $\beta_g=\alpha_g\otimes\mu^G_g$, $v(g,h)=1$, $v(g\xi^l,\xi^m)=1$, 
$x_g=1$ for all $g,h\in N$ and $l,m\in\Z$. 
\item $(\alpha_g\otimes\mu^G_g)(a)=(\Ad x_g\circ\beta_g)(a)$ holds 
for all $g\in G$ and $a\in A$. 
\item When we put $w(g,h)=x_g\beta_g(x_h)v(g,h)x_{gh}^*$ for $g,h\in G$, 
\[
w(g,h)=1,\quad w(g\xi^l,\xi^m)=1
\quad\text{and}\quad w(\xi,g)=u_{\xi g\xi^{-1}}\otimes1
\]
hold for any $g,h\in N$ and $l,m\in\Z$. 
\end{enumerate}
\end{lemma}
\begin{proof}
We may assume that $\mu^N$ equals the restriction of $\mu^G$ to $N$. 
By applying Theorem \ref{KKtrivcc2} to the $\alpha|N$-cocycle $(u_g)_g$, 
we can find $y\in U((A\otimes\mathcal{O}_\infty)^\flat)$, 
$\gamma\in\Aut(A\otimes\mathcal{O}_\infty)$ and 
an $\alpha|N$-cocycle $(c_g)_g$ in $A\otimes\mathcal{O}_\infty$ such that 
\[
\Ad y(a)=\gamma(a)\quad\forall a\in A\otimes\mathcal{O}_\infty,\quad 
y(u_g\otimes1)(\alpha_g\otimes\mu^G_g)(y^*)=c_g\quad\forall g\in N. 
\]
In particular, 
one has $\gamma\circ(\alpha_g\otimes\mu^G_g)\circ\gamma^{-1}
=\Ad c_g\circ(\alpha_g\otimes\mu^G_g)$ on $A$ 
for $g\in N$. 
Let us write $\alpha'=\alpha\otimes\mu^G$ for simplicity. 
Then we get 
\[
(\gamma\circ\alpha'_\xi)\circ\alpha'_{\xi^{-1}g\xi}
\circ(\gamma\circ\alpha'_\xi)^{-1}
=\gamma\circ\alpha'_g\circ\gamma^{-1}
=\Ad c_g\circ\alpha'_g, 
\]
and so there exists a cocycle action 
$(\beta,v):G\curvearrowright A\otimes\mathcal{O}_\infty$ 
such that 
\[
\beta_{g\xi^l}=\alpha'_g\circ(\gamma\circ\alpha'_\xi)^l
\quad\forall g\in N,\ l\in\Z
\]
and 
\[
v(g,h)=1,\quad v(g\xi^l,\xi^m)=1,\quad v(\xi,g)=c_{\xi g\xi^{-1}}
\quad\forall g,h\in N,\ l,m\in\Z. 
\]
Let $(y_l)_{l\in\Z}$ be the unitaries 
in $(A\otimes\mathcal{O}_\infty)^\flat$ satisfying 
$y_0=1$, $y_1=y^*$ and $y_l\beta_{\xi^l}(y_m)=y_{l+m}$. 
For $g\in N$ and $l\in\Z$, we put $x_{g\xi^l}=\beta_g(y_l)$. 
Then $\alpha'_g=\Ad x_g\circ\beta_g$ on $A$ for any $g\in G$. 
For $g,h\in G$, we set $w(g,h)=x_g\beta_g(x_h)v(g,h)x_{gh}^*$. 
For $g,h\in N$ and $l,m\in\Z$, it is straightforward to see 
\[
w(g,h)=1,\quad w(g\xi^l,\xi^m)=1
\]
and 
\[
w(\xi,g)=x_\xi\beta_\xi(x_g)v(\xi,g)x_{\xi g}^*
=y^*c_{\xi g\xi^{-1}}\beta_{\xi g\xi^{-1}}(y)
=y^*c_{\xi g\xi^{-1}}\alpha'_{\xi g\xi^{-1}}(y)
=u_{\xi g\xi^{-1}}\otimes1. 
\]
\end{proof}

\subsection{Existence}

In this subsection, 
we discuss existence of outer (cocycle) actions of poly-$\Z$ groups 
of Hirsch length two with prescribed $K$-theoretic data. 
For unital $C^*$-algebras $A,B$, 
we let $KK(A,B)^{-1}_*$ denote the set of 
invertible elements $x\in KK(A,B)$ such that $K_0(x)(K_0(1_A))=K_0(1_B)$. 

\begin{theorem}\label{exi_Hirsch2}
Let $G$ be a poly-$\Z$ group of Hirsch length two. 
Let $A$ be a unital Kirchberg algebra. 
\begin{enumerate}
\item For any homomorphism $\phi:G\to KK(A,A)^{-1}_*$, 
there exists an outer cocycle action $(\alpha,u):G\curvearrowright A$ 
such that $KK(\alpha_g)=\phi(g)$ for all $g\in G$. 
\item Let $\alpha$ be an action of $G$ on $A$. 
For any $c\in H^2(G,KK^1(A,A))$, 
there exists an outer cocycle action $(\beta,v):G\curvearrowright A$ 
such that $KK(\alpha_g)=KK(\beta_g)$ for all $g\in G$ and 
$\mathfrak{o}^2((\beta,v),\alpha)=c$. 
Moreover, $(\beta,v)$ can be chosen to be a genuine action 
if and only if $1_*(c)=0$ in $H^2(G,K_1(A))$. 
\end{enumerate}
\end{theorem}
\begin{proof}
There exists a normal subgroup $N\subset G$ isomorphic to $\Z$ 
and $\xi\in G$ such that $G$ is generated by $N$ and $\xi$. 

(1) 
Since $N$ is isomorphic to $\Z$, 
there exists an outer action $\alpha:N\curvearrowright A$ such that 
$KK(\alpha_g)=\phi(g)$ for all $g\in N$. 
Choose $\gamma\in\Aut(A)$ so that $KK(\gamma)=\phi(\xi)$. 
Then 
\[
KK(\alpha_g)=\phi(g)=KK(\gamma)\circ\phi(\xi^{-1}g\xi)\circ KK(\gamma^{-1})
=KK(\gamma\circ\alpha_{\xi^{-1}g\xi}\circ\gamma^{-1})
\]
holds for $g\in N$. 
It follows from \cite[Theorem 5]{Na00ETDS} that $g\mapsto\alpha_g$ and 
$g\mapsto\gamma\circ\alpha_{\xi^{-1}g\xi}\circ\gamma^{-1}$ 
are $KK$-trivially cocycle conjugate. 
Thus, there exist an $\alpha$-cocycle $(u_g)_{g\in N}$ in $A$ 
and $\tilde\gamma\in\Aut(A)$ such that $KK(\tilde\gamma)=1_A$ and 
\[
\Ad u_g\circ\alpha_g=\tilde\gamma\circ\gamma\circ\alpha_{\xi^{-1}g\xi}
\circ\gamma^{-1}\circ\tilde\gamma^{-1}\quad\forall g\in N. 
\]
Hence there exists a cocycle action $(\beta,v):G\curvearrowright A$ such that 
$\beta_g=\alpha_g$ for $g\in N$ and $\beta_\xi=\tilde\gamma\circ\gamma$. 
Clearly $KK(\beta_g)=\phi(g)$ holds for all $g\in G$. 
We can make $(\beta,v)$ outer 
by tensoring the outer action $\mu^G:G\curvearrowright\mathcal{O}_\infty$. 

(2) 
Suppose that we are given $c\in H^2(G,KK^1(A,A))=H^2(G,K_1(A_\flat))$. 
By \eqref{LHS}, there exists a $1$-cocycle $\rho:N\to K_1(A_\flat)$ 
satisfying $\bar j([\rho])=c$. 
Since $N\cong\Z$, there exists an $\alpha|N$-cocycle $(u_g)_g$ in $A_\flat$ 
such that $K_1(u_g)=\rho(g)$ for $g\in N$. 
By Lemma \ref{realization}, there exist a cocycle action 
$(\beta,v):G\curvearrowright A\otimes\mathcal{O}_\infty$ and 
a family $(x_g)_{g\in G}$ of unitaries in $(A\otimes\mathcal{O}_\infty)^\flat$ 
satisfying the following. 
\begin{itemize}
\item $\beta_g=\alpha_g\otimes\mu^G_g$, $v(g,h)=1$, $v(g\xi^l,\xi^m)=1$, 
$x_g=1$ for all $g,h\in N$ and $l,m\in\Z$. 
\item $(\alpha_g\otimes\mu^G_g)(a)=(\Ad x_g\circ\beta_g)(a)$ holds 
for all $g\in G$ and $a\in A$. 
\item When we put $w(g,h)=x_g\beta_g(x_h)v(g,h)x_{gh}^*$ for $g,h\in G$, 
\[
w(g,h)=1,\quad w(g\xi^l,\xi^m)=1
\quad\text{and}\quad w(\xi,g)=u_{\xi g\xi^{-1}}\otimes1 
\]
hold for any $g,h\in N$ and $l,m\in\Z$. 
\end{itemize}
Define $\omega:G^2\to K_1(A_\flat)$ by \eqref{defofomega}. 
It is easy to see 
\[
\omega(g,h)=0,\quad \omega(g\xi^l,\xi^m)=0
\quad\text{and}\quad\omega(\xi,g)=\rho(\xi g\xi^{-1}) 
\]
hold for any $g,h\in N$ and $l,m\in\Z$. 
Therefore, $\mathfrak{o}^2((\beta,v),\alpha\otimes\mu^G)=[\omega]$. 
Hence, from Lemma \ref{barj}, we can conclude 
$\mathfrak{o}^2((\beta,v),\alpha\otimes\mu^G)=\bar j([\rho])=c$. 
We can make $(\beta,v)$ outer 
by tensoring the outer action $\mu^G:G\curvearrowright\mathcal{O}_\infty$. 

As mentioned before, 
$1_*(c)=1_*(\mathfrak{o}^2((\beta,v),\alpha\otimes\mu^G))$ equals 
the cohomology class of $(g,h)\mapsto v(g,h)$. 
By Lemma \ref{2cv_Hirsch2}, we can conclude that 
$1_*(c)=0$ if and only if 
$(\beta,v)$ is $KK$-trivially cocycle conjugate to a genuine action. 
\end{proof}

\begin{remark}
In Theorem \ref{exi_Hirsch2} (1), 
if $K_1(A)$ is trivial or $A$ is in the Cuntz standard form, 
then $(\alpha,u)$ can be chosen to be a genuine action 
(see Lemma \ref{2cv_Hirsch2} and Theorem \ref{Cuntzstandard1}). 
In general, however, we do not know 
if a given homomorphism $\phi:G\to KK(A,A)^{-1}_*$ is realized 
by a genuine action or not. 
\end{remark}

The following two corollaries are immediate consequences of 
Theorem \ref{uni_Hirsch2} and Theorem \ref{exi_Hirsch2}. 

\begin{corollary}
Let $G$ be a poly-$\Z$ group of Hirsch length two. 
Let $A$ be a unital Kirchberg algebra and 
let $\phi:G\to KK(A,A)^{-1}_*$ be a homomorphism. 
There exists a bijective correspondence 
between the set of $KK$-trivially cocycle conjugacy classes of 
outer cocycle actions $(\alpha,u):G\curvearrowright A$ 
satisfying $KK(\alpha_g)=\phi(g)$ for all $g\in G$ and 
the cohomology group $H^2(G,KK^1(A,A))$. 
\end{corollary}

\begin{corollary}
Let $G$ be a poly-$\Z$ group of Hirsch length two. 
Let $\alpha:G\curvearrowright A$ be an outer action of $G$ 
on a unital Kirchberg algebra $A$. 
There exists a bijective correspondence between the following two sets. 
\begin{enumerate}
\item The set of $KK$-trivial cocycle conjugacy classes of 
outer actions $\beta:G\curvearrowright A$ 
satisfying $KK(\alpha_g)=KK(\beta_g)$ for all $g\in G$. 
\item The set of cohomology classes $c\in H^2(G,KK^1(A,A))$ 
satisfying $1_*(c)=0$. 
\end{enumerate}
\end{corollary}

\begin{example}
Let $A=\mathcal{O}_n$ be the Cuntz algebra. 
We have $KK(A,A)^{-1}_*=\{1\}$. 
\begin{enumerate}
\item Let $G=\langle\xi,\zeta\mid\xi\zeta=\zeta\xi\rangle\cong\Z^2$. 
Then $H^2(G,KK^1(A,A))\cong KK^1(A,A)\cong\Z_{n-1}$, 
where $KK^1(A,A)$ is regarded as a trivial module. 
Hence, there exist $n{-}1$ cocycle conjugacy classes of 
outer $G$-actions on $A$. 
This agrees with \cite[Example 8.7]{IM10Adv}. 
\item Let $G=\langle\xi,\zeta\mid\xi\zeta=\zeta^{-1}\xi\rangle$ be 
the Klein bottle group. 
Then $H^2(G,KK^1(A,A))\cong KK^1(A,A)\otimes\Z_2\cong\Z_{n-1}\otimes\Z_2$, 
where $KK^1(A,A)$ is regarded as a trivial module. 
Hence, 
\[
\#\{\text{cocycle conjugacy classes of outer actions $G\curvearrowright A$}\}
=\begin{cases}1&\text{$n$ is even}\\
2&\text{$n$ is odd. }\end{cases}
\]
\end{enumerate}
\end{example}

\section{Poly-$\Z$ groups of Hirsch length three}

For every poly-$\Z$ group $G$, 
we choose and fix an outer action 
$\mu^G:G\curvearrowright\mathcal{O}_\infty$.

\subsection{Uniqueness}

In this subsection, we determine 
when outer (cocycle) actions of poly-$\Z$ groups of Hirsch length three 
are mutually $KK$-trivially cocycle conjugate (Theorem \ref{uni_Hirsch3}). 

Let $(\alpha,u):G\curvearrowright A$ be a cocycle action of 
a discrete group $G$ on a unital $C^*$-algebra 
such that $u(g,h)\in U(A)_0$ for all $g,h\in G$. 
We introduce the invariant $\kappa^3(\alpha,u)$ as follows. 
Choose a continuous path $\tilde u(g,h):[0,1]\to U(A)_0$ 
such that $\tilde u(g,h)=1$ and $\tilde u(g,h)=u(g,h)$. 
Then 
\[
\omega(g,h,k)
=K_1(\alpha_g(\tilde u(h,k))\tilde u(g,hk)\tilde u(gh,k)^*\tilde u(g,h)^*)
\in K_1(SA)=K_0(A), 
\]
and they form a $3$-cocycle, thanks to the next lemma. 
We denote by $\kappa^3(\alpha,u)$ its cohomology class in $H^3(G,K_0(A))$. 

\begin{lemma}\label{kappa3}
In the setting above, $\omega$ is a $3$-cocycle, 
and its cohomology class does not depend on the choice of 
the continuous paths $(\tilde u(g,h))_{g,h}$. 
\end{lemma}
\begin{proof}
We compute in a similar fashion to the proof of Lemma \ref{kappa2}. 
For $g,h,k,l\in G$, 
\begin{align*}
&g\cdot\omega(h,k,l)-\omega(gh,k,l)+\omega(g,hk,l)
-\omega(g,h,kl)+\omega(g,h,k)\\
&=K_1\left(\alpha_g(\alpha_h(\tilde u(k,l))\tilde u(h,kl)
\tilde u(hk,l)^*\tilde u(h,k)^*)\right)\\
&\quad -\omega(gh,k,l)+\omega(g,hk,l)-\omega(g,h,kl)+\omega(g,h,k)\\
&=K_1\left(\alpha_{gh}(\tilde u(k,l))u(g,h)^*\alpha_g(\tilde u(h,kl))
\alpha_g(\tilde u(hk,l)^*)\alpha_g(\tilde u(h,k)^*)u(g,h)\right)\\
&\quad -K_1\left(\alpha_{gh}(\tilde u(k,l))\tilde u(gh,kl)
\tilde u(ghk,l)^*\tilde u(gh,k)^*\right)
+\omega(g,hk,l)-\omega(g,h,kl)+\omega(g,h,k)\\
&=K_1\left(\tilde u(gh,k)\tilde u(ghk,l)\tilde u(gh,kl)^*
u(g,h)^*\alpha_g(\tilde u(h,kl))
\alpha_g(\tilde u(hk,l)^*)\alpha_g(\tilde u(h,k)^*)u(g,h)\right)\\
&\quad +\omega(g,hk,l)-\omega(g,h,kl)+\omega(g,h,k)\\
&=K_1\left(\alpha_g(\tilde u(h,k)^*)u(g,h)\tilde u(gh,k)\tilde u(ghk,l)
\tilde u(gh,kl)^*u(g,h)^*\alpha_g(\tilde u(h,kl))
\alpha_g(\tilde u(hk,l)^*)\right)\\
&\quad +K_1\left(\alpha_g(\tilde u(hk,l))\tilde u(g,hkl)
\tilde u(ghk,l)^*\tilde u(g,hk)^*\right)
-\omega(g,h,kl)+\omega(g,h,k)\\
&=K_1\left(\alpha_g(\tilde u(h,k)^*)u(g,h)\tilde u(gh,k)\tilde u(ghk,l)
\tilde u(gh,kl)^*u(g,h)^*\alpha_g(\tilde u(h,kl))\right.\\
&\qquad\left.\times\tilde u(g,hkl)\tilde u(ghk,l)^*\tilde u(g,hk)^*\right)
-\omega(g,h,kl)+\omega(g,h,k)\\
&=K_1\left(\alpha_g(\tilde u(h,kl))\tilde u(g,hkl)
\tilde u(ghk,l)^*\tilde u(g,hk)^*\right.\\
&\qquad\left.\times\alpha_g(\tilde u(h,k)^*)u(g,h)\tilde u(gh,k)
\tilde u(ghk,l)\tilde u(gh,kl)^*u(g,h)^*\right)\\
&\quad -K_1\left(\alpha_g(\tilde u(h,kl))\tilde u(g,hkl)
\tilde u(gh,kl)^*\tilde u(g,h)^*\right)+\omega(g,h,k)\\
&=K_1\left(\tilde u(g,h)\tilde u(gh,kl)
\tilde u(ghk,l)^*\tilde u(g,hk)^*\right.\\
&\qquad\left.\times\alpha_g(\tilde u(h,k)^*)u(g,h)\tilde u(gh,k)
\tilde u(ghk,l)\tilde u(gh,kl)^*u(g,h)^*\right)+\omega(g,h,k)\\
&=K_1\left(\tilde u(gh,kl)\tilde u(ghk,l)^*\tilde u(g,hk)^*\right.\\
&\qquad\left.\times\alpha_g(\tilde u(h,k)^*)u(g,h)\tilde u(gh,k)
\tilde u(ghk,l)\tilde u(gh,kl)^*(u(g,h)^*\tilde u(g,h))\right)+\omega(g,h,k)\\
&=K_1\left(\tilde u(gh,kl)\tilde u(ghk,l)^*\tilde u(g,hk)^*\right.\\
&\qquad\left.\times\alpha_g(\tilde u(h,k)^*)u(g,h)(u(g,h)^*\tilde u(g,h))
\tilde u(gh,k)\tilde u(ghk,l)\tilde u(gh,kl)^*\right)+\omega(g,h,k)\\
&=K_1\left(\tilde u(g,hk)^*\alpha_g(\tilde u(h,k)^*)
\tilde u(g,h)\tilde u(gh,k)\right)+\omega(g,h,k)\\
&=0, 
\end{align*}
and so $\omega$ is a $3$-cocycle. 

When $(\hat u(g,h))_{g,h}$ is another family of paths 
from $1$ to $u(g,h)$ in $U(A)_0$, one has 
\begin{align*}
& K_1(\alpha_g(\tilde u(h,k))\tilde u(g,hk)\tilde u(gh,k)^*\tilde u(g,h)^*)
-K_1(\alpha_g(\hat u(h,k))\hat u(g,hk)\hat u(gh,k)^*\hat u(g,h)^*)\\
&=K_1(\alpha_g(\tilde u(h,k))\tilde u(g,hk)\tilde u(gh,k)^*
(\tilde u(g,h)^*\hat u(g,h))
\hat u(gh,k)\hat u(g,hk)^*\alpha_g(\hat u(h,k)^*))\\
&=K_1(\alpha_g(\tilde u(h,k))\tilde u(g,hk)
(\tilde u(gh,k)^*\hat u(gh,k))\hat u(g,hk)^*\alpha_g(\hat u(h,k)^*))
+K_1(\tilde u(g,h)^*\hat u(g,h))\\
&=K_1(\alpha_g(\tilde u(h,k))
(\tilde u(g,hk)\hat u(g,hk)^*)\alpha_g(\hat u(h,k)^*))\\
&\quad +K_1(\tilde u(gh,k)^*\hat u(gh,k))+K_1(\tilde u(g,h)^*\hat u(g,h))\\
&=g\cdot K_1(\tilde u(h,k)\hat u(h,k)^*)-K_1(\tilde u(gh,k)\hat u(gh,k)^*)\\
&\quad +K_1(\tilde u(g,hk)\hat u(g,hk)^*)-K_1(\tilde u(g,h)\hat u(g,h)^*), 
\end{align*}
which is a coboundary. 
\end{proof}

It is easy to show the following. 

\begin{lemma}\label{kappa3_inv}
In the setting above, 
when $(x_g)_g$ is a family of unitaries in $U(A)_0$, 
one has $\kappa^3(\alpha^x,u^x)=\kappa^3(\alpha,u)$. 
\end{lemma}

The following lemma enables us to relate $\kappa^3$ to $\kappa^2$. 

\begin{lemma}\label{kappa3=kappa2}
Let $(\alpha,u):G\curvearrowright A$ be a cocycle action of 
a discrete group $G$ on a unital $C^*$-algebra 
such that $u(g,h)\in U(A)_0$ for all $g,h\in G$. 
Suppose that $N\subset G$ is a normal subgroup such that $G/N\cong\Z$. 
Take $\xi\in G$ so that $G$ is generated by $N$ and $\xi$. 
If 
\[
u(g,h)=1\quad\text{and}\quad u(g\xi^l,\xi^m)=1
\quad\forall g,h\in N,\ l,m\in\Z, 
\]
then the unitaries $\check u_g=u(\xi,\xi^{-1}g\xi)$ 
form an $\alpha|N$-cocycle 
and $\kappa^3(\alpha,u)$ equals $\bar j(\kappa^2(\alpha|N,\check u))$, 
where $\bar j:H^2(N,K_0(A))\to H^3(G,K_0(A))$ is a homomorphism 
defined in Section 7.2. 
\end{lemma}
\begin{proof}
By Lemma \ref{check}, 
$(\check u_g)_{g\in N}$ is an $\alpha|N$-cocycle satisfying 
\[
\alpha_\xi\circ\alpha_{\xi^{-1}g\xi}\circ\alpha_\xi^{-1}
=\Ad \check u_g\circ\alpha_g\quad\forall g\in N. 
\]
For $l\in\Z$ and $g\in N$, 
we put $\check u_{l,g}=u(\xi^l,\xi^{-l}g\xi^l)$. 
Similarly, we can check that $(\check u_{l,g})_g$ is 
an $\alpha|N$-cocycle satisfying 
\[
\alpha_{\xi^l}\circ\alpha_{\xi^{-l}g\xi^l}\circ\alpha_{\xi^l}^{-1}
=\Ad \check u_{l,g}\circ\alpha_g
\]
and 
\[
\alpha_{\xi^l}(\check u_{m,\xi^{-l}g\xi^l})\check u_{l,g}
=\check u_{l+m,g}. 
\]
For $g\in N$, let $\tilde u_g:[0,1]\to U(A)$ be a continuous map 
such that $\tilde u_g(0)=1$ and $\tilde u_g(1)=\check u_g$. 
Then we can construct continuous maps $\tilde u_{l,g}:[0,1]\to U(A)$ 
for $l\in\Z$ and $g\in N$ such that 
$\tilde u_{1,g}=\tilde u_g$, 
$\tilde u_{l,g}(0)=1$, $\tilde u_{l,g}(1)=\check u_{l,g}$ and 
\[
\alpha_{\xi^l}(\tilde u_{m,\xi^{-l}g\xi^l})\tilde u_{l,g}=\tilde u_{l+m,g}, 
\]
where $(\id\otimes\alpha,1\otimes u):G\curvearrowright C([0,1])\otimes A$ is 
abbreviated as $(\alpha,u)$ for simplicity. 
Define $\rho_l:N^2\to K_1(SA)$ by 
\[
\rho_l(g,h)
=K_1\left(\tilde u_{l,g}\alpha_g(\tilde u_{l,h})\tilde u_{l,gh}^*\right), 
\]
which is a $2$-cocycle by Lemma \ref{kappa2} and 
whose cohomology class is $\kappa^2(\alpha|N,\check u_l)$. 
For any $g,h\in N$ and $l,m\in\Z$, we have 
\begin{align*}
&K_1\left(\tilde u_{l+m,g}\alpha_g(\tilde u_{l+m,h})
\tilde u_{l+m,gh}^*\right)\\
&=K_1\left(\alpha_{\xi^l}(\tilde u_{m,\xi^{-l}g\xi^l})\tilde u_{l,g}
\cdot\alpha_g(\alpha_{\xi^l}(\tilde u_{m,\xi^{-l}h\xi^l})\tilde u_{l,h})
\cdot(\alpha_{\xi^l}(\tilde u_{m,\xi^{-l}gh\xi^l})\tilde u_{l,gh})^*\right)\\
&=K_1\left(\alpha_{\xi^l}(\tilde u_{m,\xi^{-l}g\xi^l})\tilde u_{l,g}
\check u_{l,g}^*
\alpha_{\xi^l}(\alpha_{\xi^{-l}g\xi^l}(\tilde u_{m,\xi^{-l}h\xi^l}))
\check u_{l,g}\alpha_g(\tilde u_{l,h})
\tilde u_{l,gh}^*\alpha_{\xi^l}(\tilde u_{m,\xi^{-l}gh\xi^l})^*\right)\\
&=K_1\left(\alpha_{\xi^l}(\tilde u_{m,\xi^{-l}g\xi^l})
\alpha_{\xi^l}(\alpha_{\xi^{-l}g\xi^l}(\tilde u_{m,\xi^{-l}h\xi^l}))
\tilde u_{l,g}\alpha_g(\tilde u_{l,h})\tilde u_{l,gh}^*
\alpha_{\xi^l}(\tilde u_{m,\xi^{-l}gh\xi^l})^*\right)\\
&=K_1\left(
\alpha_{\xi^l}(\tilde u_{m,\xi^{-l}gh\xi^l}^*\tilde u_{m,\xi^{-l}g\xi^l}
\alpha_{\xi^{-l}g\xi^l}(\tilde u_{m,\xi^{-l}h\xi^l}))\right)
+K_1\left(\tilde u_{l,g}\alpha_g(\tilde u_{l,h})\tilde u_{l,gh}^*\right). 
\end{align*}
Hence $\rho_{l+m}=\rho_l+\xi^l\rho_m$ is obtained. 
It follows from Lemma \ref{barj} that 
$\bar j(\kappa^2(\alpha|N,\check u))$ is given 
by the cohomology class of the $3$-cocycle 
\begin{align*}
&(g\xi^l,h\xi^m,k\xi^n)\\
&\mapsto -g\cdot\rho_l(\xi^lh\xi^{-l},\xi^{l+m}k\xi^{-l-m})\\
&=-g\cdot K_1\left(\tilde u_{l,\xi^lh\xi^{-l}}
\cdot\alpha_{\xi^lh\xi^{-l}}(\tilde u_{l,\xi^{l+m}k\xi^{-l-m}})
\cdot\tilde u_{l,\xi^lh\xi^{-l}\xi^{l+m}k\xi^{-l-m}}^*\right)
\in K_1(SA)=K_0(A). 
\end{align*}

Now let us consider $\kappa^3(\alpha,u)\in H^3(G,K_0(A))$. 
Since 
\[
u(g\xi^l,h\xi^m)=\alpha_g(u(\xi^l,h))
=\alpha_g(\check u_{l,\xi^lh\xi^{-l}}), 
\]
$\alpha_g(\tilde u_{l,\xi^lh\xi^{-l}})$ is a continuous path 
connecting $1$ to $u(g\xi^l,h\xi^m)$ in $U(A)_0$. 
Then $\kappa^3(\alpha,u)$ is given by the $3$-cocycle 
\begin{align*}
(g\xi^l,h\xi^m,k\xi^n)
&\mapsto 
K_1\left(\alpha_{g\xi^l}(\alpha_h(\tilde u_{m,\xi^mk\xi^{-m}}))
\cdot\alpha_g(\tilde u_{l,\xi^lh\xi^mk\xi^{-m}\xi^{-l}})\right.\\
&\qquad\left.
\cdot\alpha_{g\xi^lh\xi^{-l}}(\tilde u_{l+m,\xi^{l+m}k\xi^{-l-m}})^*
\cdot\alpha_g(\tilde u_{l,\xi^lh\xi^{-l}})^*\right)\in K_1(SA). 
\end{align*}
One has 
\begin{align*}
& K_1\left(\alpha_{g\xi^l}(\alpha_h(\tilde u_{m,\xi^mk\xi^{-m}}))
\alpha_g(\tilde u_{l,\xi^lh\xi^mk\xi^{-m}\xi^{-l}})
\alpha_{g\xi^lh\xi^{-l}}(\tilde u_{l+m,\xi^{l+m}k\xi^{-l-m}})^*
\alpha_g(\tilde u_{l,\xi^lh\xi^{-l}})^*\right)\\
&=K_1\left(\alpha_g\bigl(
u(\xi^l,h)\alpha_{\xi^lh}(\tilde u_{m,\xi^mk\xi^{-m}})u(\xi^l,h)^*
\tilde u_{l,\xi^lh\xi^{-l}\xi^{l+m}k\xi^{-l-m}}\right.\\
&\qquad\left.
\alpha_{\xi^lh\xi^{-l}}(\tilde u_{l+m,\xi^{l+m}k\xi^{-l-m}})^*
\tilde u_{l,\xi^lh\xi^{-l}}^*\bigl)\right)\\
&=g\cdot K_1\left(
\check u_{l,\xi^lh\xi^{-l}}\alpha_{\xi^lh}(\tilde u_{m,\xi^mk\xi^{-m}})
\check u_{l,\xi^lh\xi^{-l}}^*
\tilde u_{l,\xi^lh\xi^{-l}\xi^{l+m}k\xi^{-l-m}}\right.\\
&\qquad\left.\alpha_{\xi^lh\xi^{-l}}
(\alpha_{\xi^l}(\tilde u_{m,\xi^mk\xi^{-m}})
\tilde u_{l,\xi^{l+m}k\xi^{-l-m}})^*
\tilde u_{l,\xi^lh\xi^{-l}}^*\right)\\
&=g\cdot K_1\left(
\check u_{l,\xi^lh\xi^{-l}}\alpha_{\xi^lh}(\tilde u_{m,\xi^mk\xi^{-m}})
\check u_{l,\xi^lh\xi^{-l}}^*
\tilde u_{l,\xi^lh\xi^{-l}}\tilde u_{l,\xi^lh\xi^{-l}}^*
\tilde u_{l,\xi^lh\xi^{-l}\xi^{l+m}k\xi^{-l-m}}\right.\\
&\qquad\left.\alpha_{\xi^lh\xi^{-l}}(\tilde u_{l,\xi^{l+m}k\xi^{-l-m}})^*
\alpha_{\xi^lh}(\tilde u_{m,\xi^mk\xi^{-m}})^*
\tilde u_{l,\xi^lh\xi^{-l}}^*\right)\\
&=g\cdot K_1\left(\tilde u_{l,\xi^lh\xi^{-l}}
\alpha_{\xi^lh}(\tilde u_{m,\xi^mk\xi^{-m}})\tilde u_{l,\xi^lh\xi^{-l}}^*
\tilde u_{l,\xi^lh\xi^{-l}\xi^{l+m}k\xi^{-l-m}}\right.\\
&\qquad\left.\alpha_{\xi^lh\xi^{-l}}(\tilde u_{l,\xi^{l+m}k\xi^{-l-m}})^*
\alpha_{\xi^lh}(\tilde u_{m,\xi^mk\xi^{-m}})^*
\tilde u_{l,\xi^lh\xi^{-l}}^*\right)\\
&=g\cdot K_1\left(
\tilde u_{l,\xi^lh\xi^{-l}}^*
\tilde u_{l,\xi^lh\xi^{-l}\xi^{l+m}k\xi^{-l-m}}
\alpha_{\xi^lh\xi^{-l}}(\tilde u_{l,\xi^{l+m}k\xi^{-l-m}})^*\right)\\
&=-g\cdot K_1\left(\tilde u_{l,\xi^lh\xi^{-l}}
\cdot\alpha_{\xi^lh\xi^{-l}}(\tilde u_{l,\xi^{l+m}k\xi^{-l-m}})
\cdot\tilde u_{l,\xi^lh\xi^{-l}\xi^{l+m}k\xi^{-l-m}}^*\right). 
\end{align*}
Therefore 
$\kappa^3(\alpha,u)$ is equal to $\bar j(\kappa^2(\alpha|N,\check u))$. 
\end{proof}

\begin{lemma}\label{2cv_Hirsch3}
Let $G$ be a poly-$\Z$ group of Hirsch length three and 
let $(\alpha,u)$ be a cocycle action of $G$ 
belonging to $\AC(\mathcal{O}_\infty,\mu^G)$. 
If $u(g,h)\in U(A)_0$ and $\kappa^3(\alpha,u)=0$, then 
there exists a family of unitaries $(v_g)_{g\in G}$ in $U(A)_0$ 
such that $u(g,h)=\alpha_g(v_h^*)v_g^*v_{gh}$ for all $g,h\in G$. 
\end{lemma}
\begin{proof}
We note that $\kappa^3(\alpha,u)$ is invariant 
under perturbation by unitaries in $U(A)_0$ (see Lemma \ref{kappa3_inv}). 
There exists a normal poly-$\Z$ subgroup $N\subset G$ of Hirsch length two 
and $\xi\in G$ such that $G$ is generated by $N$ and $\xi$. 
By Lemma \ref{2cv_Hirsch2}, 
we may assume that $u(g,h)=1$ for all $g,h\in N$. 
By a cocycle perturbation, 
we may further assume that $u(g\xi^l,\xi^m)=1$ 
for all $g\in N$ and $l,m\in\Z$. 
It follows from Lemma \ref{check} that the unitaries 
\[
\check u_g=u(\xi,\xi^{-1}g\xi)\in U(A)_0
\]
form an $\alpha|N$-cocycle satisfying 
\[
\alpha_\xi\circ\alpha_{\xi^{-1}g\xi}\circ\alpha_\xi^{-1}
=\Ad\check u_g\circ\alpha_g\quad\forall g\in N. 
\]
By Lemma \ref{kappa3=kappa2}, 
we get $\bar j(\kappa^2(\alpha|N,\check u))=\kappa^3(\alpha,u)=0$, 
and so there exists a $2$-cocycle $\rho:N^2\to K_1(SA)$ such that 
\[
\kappa^2(\alpha|N,\check u)=[\rho-\xi\rho]. 
\]
By Lemma \ref{realize:ka2'}, 
we can find an $\alpha|N$-cocycle $(v_g)_{g\in N}$ in $U(A)_0$ 
such that $\kappa^2(\alpha|N,v)=[\rho]$. 
Letting $v_{g\xi^l}=v_g$ for $g\in N$ and $l\in\Z$, 
we obtain a family of unitaries $(v_g)_{g\in G}$ in $U(A)_0$. 
Consider the cocycle action $(\alpha^v,u^v)$. 
We still have 
\[
u^v(g,h)=1\quad\text{and}\quad u^v(g\xi^l,\xi^m)=1
\quad\forall g,h\in N,\ l,m\in\Z. 
\]
Besides, 
\[
\alpha^v_\xi\circ\alpha^v_{\xi^{-1}g\xi}\circ(\alpha^v_\xi)^{-1}
=\Ad(\alpha_\xi(v_{\xi^{-1}g\xi})\check u_gv_g^*)\circ\alpha^v_g
\]
holds true for all $g\in N$, 
where $w_g=\alpha_\xi(v_{\xi^{-1}g\xi})\check u_gv_g^*$ form 
an $\alpha^v|N$-cocycle. 
By means of Lemma \ref{kappa2additive}, one gets 
\[
\kappa^2(\alpha^v|N,w)
=[\xi\rho]+\kappa^2(\alpha|N,\check u)-[\rho]=0. 
\]
Then, Lemma \ref{1cv_Hirsch2} tells us that 
the $\alpha^v|N$-cocycle $(w_g)_g$ can be approximated by coboundaries. 
Therefore, by a suitable perturbation, 
we may further assume that $u^v(g,h)$ is close to $1$ 
on a finite generating subset of $G$. 
Then, by the $H^2$-stability of $G$ (Theorem \ref{H1H2stable}), 
we can conclude that $(u^v(g,h))_{g,h\in G}$ is a coboundary. 
\end{proof}

Let $A$ be a unital $C^*$-algebra such that 
$K_1(A)$ is canonically isomorphic to $U(A)/U(A)_0$. 
Let $(\alpha,u):G\curvearrowright A$ be a cocycle action 
of a countable discrete group $G$. 
Assume that $u(g,h)$ is in $U(A)_0$ for all $g,h\in G$. 
We define a homomorphism 
$h^{1,3}_{(\alpha,u)}:H^1(G,K_1(A))\to H^3(G,K_0(A))$ as follows. 
Let $\eta:G\to K_1(A)$ be a $1$-cocycle. 
Choose $v_g\in U(A)$ so that $K_1(v_g)=\eta(g)$. 
Then $K_1(u^v(g,h))=0$ and 
$\kappa^3(\alpha^v,u^v)\in H^3(G,K_0(A))$ is defined. 
By Lemma \ref{kappa3_inv}, 
$\kappa^3(\alpha^v,u^v)$ does not depend on the choice of $(v_g)_{g\in G}$, 
which also implies that 
it depends only on the cohomology class $[\eta]\in H^1(G,K_1(A))$.  
We denote by $h^{1,3}_{(\alpha,u)}$ the map 
\[
H^1(G,K_1(A))\ni[\eta]\mapsto
\kappa^3(\alpha^v,u^v)-\kappa^3(\alpha,u)\in H^3(G,K_0(A)). 
\]

\begin{lemma}\label{h13}
The map $h^{1,3}_{(\alpha,u)}:H^1(G,K_1(A))\to H^3(G,K_0(A))$ 
is a homomorphism. 
\end{lemma}
\begin{proof}
Consider $(\alpha\otimes\id,u\otimes1):G\curvearrowright A\otimes M_2$. 
Clearly, 
we have $h^{1,3}_{(\alpha,u)}=h^{1,3}_{(\alpha\otimes\id,u\otimes1)}$. 
Let $\eta:G\to K_1(A)$ and $\zeta:G\to K_1(A)$ be $1$-cocycles. 
Choose $v_g\in U(A)$ and $w_g\in U(A)$ 
so that $K_1(v_g)=\eta(g)$ and $K_1(w_g)=\zeta(g)$. 
Let $z_g=\diag(v_g,w_g)\in U(A\otimes M_2)$. 
Then we obtain 
\begin{align*}
h^{1,3}_{(\alpha,u)}([\eta]+[\zeta])
&=h^{1,3}_{(\alpha\otimes\id,u\otimes1)}([\eta]+[\zeta])\\
&=\kappa^3((\alpha\otimes\id)^z,(u\otimes1)^z)
-\kappa^3(\alpha\otimes\id,u\otimes1)\\
&=\kappa^3(\alpha^v,u^v)+\kappa^3(\alpha^w,u^w)-2\kappa^3(\alpha,u)\\
&=h^{1,3}_{(\alpha,u)}([\eta])+h^{1,3}_{(\alpha,u)}([\zeta]), 
\end{align*}
which means that $h^{1,3}_{(\alpha,u)}$ is a homomorphism. 
\end{proof}

When $(\alpha,1):G\curvearrowright A$ is a genuine action, 
we write $h^{1.3}_{(\alpha,1)}=h^{1,3}_\alpha$. 

Let $A$ be a unital Kirchberg algebra and 
let $(\alpha,u):G\curvearrowright A$ be a cocycle action 
such that $u(g,h)\in U(A)_0$. 
Regarding $\alpha$ as an action of $G$ on $A_\flat$, 
we also obtain a homomorphism $H^1(G,K_1(A_\flat))$ to $H^3(G,K_0(A_\flat))$ 
We denote it by $\tilde h^{1,3}_\alpha$. 

\begin{lemma}\label{h13p=ph13}
Let $p\in A$ be a projection such that $K_0(p)\in K_0(A)^G$. 
Then we have $h^{1,3}_{(\alpha,u)}\circ p_*=p_*\circ\tilde h^{1,3}_\alpha$. 
\end{lemma}
\begin{proof}
Notice that $K_*(A)$ is canonically isomorphic to $K_*(A^\flat)$ 
via the embedding $A\to A^\flat$. 
Let $\eta:G\to K_1(A_\flat)$ be a $1$-cocycle. 
Take a family $(v_g)_g$ of unitaries in $U(A_\flat)$ 
satisfying $K_1(v_g)=\eta(g)$ for every $g\in G$. 
Set $\tilde v_g=v_gp+(1{-}p)\in U(A^\flat)$. 
For each $g\in G$, 
we choose a unitary $w_g\in U(A)_0$ so that $w_gpw_g^*=\alpha_g(p)$. 
Under the identification of $K_1(A)$ with $K_1(A^\flat)$, 
one has $p_*(\eta(g))=K_1(\tilde v_gw^*_g)$. 
Then 
\begin{align*}
&\tilde v_gw_g^*\alpha_g(\tilde v_hw_h^*)u(g,h)w_{gh}\tilde v_{gh}^*\\
&=(v_gp+1{-}p)w_g^*\alpha_g(v_hp+1{-}p)w_gw_g^*\alpha_g(w_h^*)u(g,h)
w_{gh}(v_{gh}p+1{-}p)^*\\
&=(v_g\alpha_g(v_h)v_{gh}^*p+1{-}p)\cdot w_g^*\alpha_g(w_h^*)u(g,h)w_{gh}, 
\end{align*}
which implies 
\[
(h^{1,3}_{(\alpha,u)}\circ p_*)([\eta])+\kappa^3(\alpha,u)
=(p_*\circ\tilde h^{1,3}_\alpha)([\eta])+\kappa^3(\alpha,u). 
\]
Therefore 
$h^{1,3}_{(\alpha,u)}\circ p_*=p_*\circ\tilde h^{1,3}_\alpha$ is obtained. 
\end{proof}

\begin{remark}\label{tildeh=0}
Let $A$ be a unital Kirchberg algebra 
which is stably isomorphic to $\mathcal{O}_n$ or $\mathcal{O}_\infty$, 
and let $(\alpha,u):G\curvearrowright A$ be a cocycle action 
such that $KK(\alpha_g)=1_A$ for each $g\in G$. 
As $K_1(A)$ is trivial, $h^{1,3}_{(\alpha,u)}$ is zero. 
There exists a projection $p\in A$ such that 
$K_0(p)$ generates $K_0(A)$, 
and $p_*:K_0(A_\flat)\to K_0(A)$ is an isomorphism. 
It follows from the lemma above that 
$\tilde h^{1,3}_\alpha$ is zero. 
\end{remark}

We would like to introduce 
an obstruction class $\mathfrak{o}^3((\alpha,u),\beta)$. 
Let $(\alpha,u)$ be a cocycle action of a countable discrete group $G$ 
on a unital Kirchberg algebra $A$. 
Let $\beta$ be an action of $G$ on $A$. 
Assume $KK(\alpha_g)=KK(\beta_g)$ for all $g\in G$. 
We choose a family $(v_g)_{g\in G}$ of unitaries in $A^\flat$ 
satisfying $(\Ad v_g\circ\alpha_g)(a)=\beta_g(a)$ for every $a\in A$. 
Define a cocycle action $(\sigma,w):G\curvearrowright A_\flat$ 
by $\sigma_g=\Ad v_g\circ\alpha_g$ and 
$w(g,h)=v_g\alpha_g(v_h)u(g,h)v_{gh}^*$. 
Assume further that $\mathfrak{o}^2((\alpha,u),\beta)=0$ 
(see Section 7.1 for the definition of $\mathfrak{o}^2$). 
Then, we can choose the family $(v_g)_{g\in G}$ 
so that $w(g,h)\in U(A_\flat)_0$ for all $g,h\in G$. 
Hence, $\kappa^3(\sigma,w)\in H^3(G,K_0(A_\flat))$ can be defined. 
When $(v'_g)_g$ is another family in $U(A^\flat)$ with the same properties 
and $(\sigma',w'):G\curvearrowright A_\flat$ is the cocycle action 
arising from $(v'_g)_g$, 
$v'_gv_g^*$ is in $U(A_\flat)$ and 
\[
\kappa^3(\sigma',w')-\kappa^3(\sigma,w)
=h^{1,3}_{(\sigma,w)}([\eta]), 
\]
where $\eta:G\to K_1(A_\flat)$ is the $1$-cocycle 
given by $g\mapsto K_1(v'_gv_g^*)$. 
Remember that $K_i(\alpha_g|A_\flat)=K_i(\beta_g|A_\flat)=K_i(\sigma_g)$ 
holds true for $i=0,1$ and $g\in G$ by Lemma \ref{auto_flat}. 

\begin{lemma}
In the setting above, 
we have $h^{1,3}_{(\sigma,w)}=\tilde h^{1,3}_\alpha$. 
\end{lemma}
\begin{proof}
Choose a path of unitaries $\tilde w(g,h):[0,1]\to U(A_\flat)$ 
such that $\tilde w(g,h)(0)=1$ and $\tilde w(g,h)(1)=w(g,h)$. 

Let $\eta:G\to K_1(A_\flat)$ be a $1$-cocycle. 
We choose $x_g\in K_1(A_\flat)$ so that $K_1(x_g)=\eta(g)$. 
By the rescaling argument, we may assume 
\[
[\sigma_{gh}(x_k),\tilde w(g,h)(t)]=0,\quad 
[\alpha_g(x_h),v_g]=0
\]
hold for all $g,h,k\in G$ and $t\in[0,1]$ 
(here we have used Theorem \ref{refinedNakamura}). 
Furthermore, 
we may assume that there exists a path of unitaries 
$\tilde x(g,h):[0,1]\to U(A_\flat)$ such that 
\[
\tilde x(g,h)(0)=1,\quad 
\tilde x(g,h)(1)=x_g\sigma_g(x_h)x_{gh}^*=x_g\alpha_g(x_h)x_{gh}^*. 
\]
and 
\[
[\alpha_g(\tilde x(h,k)(t)),v_g]=0\quad\forall g,h,k\in G,\ t\in [0,1]. 
\]
Then the concatenation of the paths $t\mapsto\tilde x(g,h)(t)$ and 
$t\mapsto x_g\sigma_g(x_h)\tilde w(g,h)(t)x_{gh}^*$ gives 
a path connecting $1$ to $w^x(g,h)=x_g\sigma_g(x_h)w(g,h)x_{gh}^*$. 
It is easy to see 
\begin{align*}
&\sigma^x_g(x_h\sigma_h(x_k)\tilde w(h,k)(t)x_{hk}^*)
\cdot x_g\sigma_g(x_{hk})\tilde w(g,hk)(t)x_{ghk}^*\\
&\quad \cdot(x_{gh}\sigma_{gh}(x_k)\tilde w(gh,k)(t)x_{ghk}^*)^*
\cdot(x_g\sigma_g(x_h)\tilde w(g,h)(t)x_{gh}^*)^*\\
&=x_g\sigma_g(x_h\sigma_h(x_k))
\sigma_g(\tilde w(h,k)(t))\tilde w(g,hk)(t)
\tilde w(gh,k)(t)^*\sigma_{gh}(x_k^*)
\tilde w(g,h)(t)^*\sigma_g(x_h^*)x_g^*\\
&=x_g\sigma_g(x_h)\sigma_{gh}(x_k)\cdot
\sigma_g(\tilde w(h,k)(t))\tilde w(g,hk)(t)
\tilde w(gh,k)(t)^*\tilde w(g,h)(t)^*
\cdot\sigma_{gh}(x_k^*)\sigma_g(x_h^*)x_g^*
\end{align*}
and 
\begin{align*}
&\sigma^x_g(\tilde x(h,k)(t))\tilde x(g,hk)(t)
\tilde x(gh,k)(t)^*\tilde x(g,h)(t)^*\\
&=\alpha^x_g(\tilde x(h,k)(t))\tilde x(g,hk)(t)
\tilde x(gh,k)(t)^*\tilde x(g,h)(t)^*. 
\end{align*}
Hence we have 
\[
\kappa^3(\sigma^x,w^x)=\kappa^3(\sigma,w)+\kappa^3(\alpha^x,1^x), 
\]
which means $h^{1,3}_{(\sigma,w)}([\eta])=\tilde h^{1,3}_\alpha([\eta])$. 
\end{proof}

\begin{definition}
Let $(\alpha,u):G\curvearrowright A$, $\beta:G\curvearrowright A$ and 
$(\sigma,w):G\curvearrowright A_\flat$ be as above. 
We define 
\[
\mathfrak{o}^3((\alpha,u),\beta)
=\kappa^3(\sigma,w)+\Ima\tilde h^{1,3}_\alpha
\in H^3(G,K_0(A_\flat))/\Ima\tilde h^{1,3}_\alpha. 
\]
By the lemma above, 
$\mathfrak{o}^3((\alpha,u),\beta)$ does not depend on the choice of $(v_g)_g$, 
and $\mathfrak{o}^3((\alpha,u),\beta)=0$ if and only if 
the family $(v_g)_g$ can be chosen so that $\kappa^3(\sigma,w)=0$. 
When $\alpha$ is a genuine action, 
we write $\mathfrak{o}^3(\alpha,\beta)=\mathfrak{o}^3((\alpha,1),\beta)$. 
\end{definition}

Now we are ready to prove the following theorem. 

\begin{theorem}\label{uni_Hirsch3}
Let $A$ be a unital Kirchberg algebra and 
let $G$ be a poly-$\Z$ group of Hirsch length three. 
Let $(\alpha,u):G\curvearrowright A$ be an outer cocycle action 
and let $\beta:G\curvearrowright A$ be an outer action. 
The following are equivalent. 
\begin{enumerate}
\item $(\alpha,u)$ and $\beta$ are $KK$-trivially cocycle conjugate. 
\item $KK(\alpha_g)=KK(\beta_g)$ for all $g\in G$, 
$\mathfrak{o}^2((\alpha,u),\beta)=0$ and 
$\mathfrak{o}^3((\alpha,u),\beta)=0$. 
\end{enumerate}
\end{theorem}
\begin{proof}
(1)$\Rightarrow$(2) is obvious from Lemma \ref{key}. 

Let us show the converse. 
Define the cocycle action $(\sigma,w):G\curvearrowright A_\flat$ as above. 
By $\mathfrak{o}^2((\alpha,u),\beta)=0$ and 
$\mathfrak{o}^3((\alpha,u),\beta)=0$, 
we may assume $K_1(w(g,h))=0$ in $K_1(A_\flat)$ and 
$\kappa^3(\sigma,w)=0$ for every $g,h\in G$. 
It follows from Lemma \ref{2cv_Hirsch3} and Lemma \ref{key} that 
$(\alpha,u)$ and $\beta$ are $KK$-trivially cocycle conjugate. 
\end{proof}

\subsection{Existence}

In this subsection, 
we discuss existence of outer (cocycle) actions of poly-$\Z$ groups 
of Hirsch length three with prescribed $K$-theoretic data. 

\begin{theorem}\label{exi_Hirsch3}
Let $G$ be a poly-$\Z$ group of Hirsch length three. 
Let $A$ be a unital Kirchberg algebra. 
\begin{enumerate}
\item Let $\alpha:G\curvearrowright A$ be an action. 
For any $c\in H^2(G,KK^1(A,A))$, 
there exists an outer cocycle action $(\beta,v):G\curvearrowright A$ 
such that $KK(\alpha_g)=KK(\beta_g)$ for all $g\in G$ and 
$\mathfrak{o}^2((\beta,v),\alpha)=c$. 
\item Let $\alpha:G\curvearrowright A$ be an action. 
For any $c\in H^3(G,KK(A,A))$, 
there exists an outer cocycle action $(\beta,v):G\curvearrowright A$ 
such that $KK(\alpha_g)=KK(\beta_g)$ for all $g\in G$ and 
$\mathfrak{o}^2((\beta,v),\alpha)=0$ and 
$\mathfrak{o}^3((\beta,v),\alpha)=c+\Ima\tilde h^{1,3}_\beta$. 
\end{enumerate}
\end{theorem}
\begin{proof}
There exists a normal poly-$\Z$ subgroup $N\subset G$ 
(of Hirsch length two) 
and $\xi\in G$ such that $G$ is generated by $N$ and $\xi$. 

(1) 
First we claim that 
the statement is true if $A$ is in the Cuntz standard form. 
Let $\alpha:G\curvearrowright A$ be an action and 
let $c\in H^2(G,KK^1(A,A))=H^2(G,K_1(A_\flat))$. 
By \eqref{LHS}, we have the short exact sequence 
\[
0\longrightarrow H^1(\Z,H^1(N,K_1(A_\flat))
\stackrel{j}{\longrightarrow}H^2(G,K_1(A_\flat))
\stackrel{q}{\longrightarrow}H^2(N,K_1(A_\flat))^\Z\longrightarrow0. 
\]
Since $N$ is a poly-$\Z$ group of Hirsch length two and 
$A$ is in the Cuntz standard form, 
by Theorem \ref{exi_Hirsch2} (2) and Theorem \ref{Cuntzstandard1}, 
there exists an outer action $\beta:N\curvearrowright A$ 
such that $KK(\beta_g)=KK(\alpha_g)$ for all $g\in N$ and 
$\mathfrak{o}^2(\beta,\alpha|N)=q(c)$. 
Define $\beta':N\curvearrowright A$ by 
\[
\beta'_g=\alpha_\xi\circ\beta_{\xi^{-1}g\xi}\circ\alpha_\xi^{-1}. 
\]
Clearly $KK(\beta'_g)=KK(\alpha_g)$ for any $g\in N$. 
Because 
$\mathfrak{o}^2(\beta,\alpha|N)=q(c)$ belongs to $H^2(N,K_1(A_\flat))^\Z$, 
we also obtain 
$\mathfrak{o}^2(\beta',\alpha|N)=\mathfrak{o}^2(\beta,\alpha|N)$. 
Therefore $\mathfrak{o}^2(\beta,\beta')=0$. 
It follows from Theorem \ref{uni_Hirsch2} that 
$\beta$ and $\beta'$ are $KK$-trivially cocycle conjugate. 
Thus there exist $\gamma\in\Aut(A)$ and a $\beta$-cocycle $(c_g)_g$ in $A$ 
such that $KK(\gamma)=1_A$ and 
\[
(\gamma\circ\alpha_\xi)\circ\beta_{\xi^{-1}g\xi}
\circ(\gamma\circ\alpha_\xi)^{-1}=\Ad c_g\circ\beta_g
\]
holds true for any $g\in N$. 
Hence $\beta:N\curvearrowright A$ extends to 
a cocycle action $(\beta,v):G\curvearrowright A$ 
such that $KK(\beta_g)=KK(\alpha_g)$ for all $g\in G$ and 
$q(\mathfrak{o}^2((\beta,v),\alpha))=q(c)$. 
We may replace $(\beta,v)$ with a genuine action $\beta''$ 
because $A$ is in the Cuntz standard form. 

Now suppose that a $1$-cocycle $\rho:N\to K_1(A_\flat)$ satisfies 
$\bar j([\rho])=c-\mathfrak{o}^2(\beta'',\alpha)$. 
Since $N$ is a poly-$\Z$ group of Hirsch length two, 
by Lemma \ref{realize:1c}, 
there exists a $\beta''|N$-cocycle $(u_g)_g$ in $A_\flat$ 
such that $K_1(u_g)=\rho(g)$ for $g\in N$. 
In exactly the same way as Theorem \ref{exi_Hirsch2} (2), 
we can find a cocycle action $(\beta''',w):G\curvearrowright A$ 
such that $KK(\beta'''_g)=KK(\beta''_g)$ for all $g\in G$ and 
$\mathfrak{o}^2((\beta''',w),\beta'')=\bar j([\rho])$. 
We may replace $(\beta''',w)$ with a genuine action, 
because $A$ is in the Cuntz standard form. 
Then the proof of the claim is completed by the chain rule.  

Let $A$ be a unital Kirchberg algebra 
which is not necessarily in the Cuntz standard form. 
Let $\alpha:G\curvearrowright A$ be an outer action and 
let $c\in H^2(G,KK^1(A,A))=H^2(G,K_1(A_\flat))$. 
We consider the action 
$\alpha\otimes\id:G\curvearrowright A\otimes\mathcal{O}$. 
By the proof above, 
there exists an outer action $\beta:G\curvearrowright A\otimes\mathcal{O}$ 
such that $KK(\beta_g)=KK(\alpha_g\otimes\id)$ and 
$\mathfrak{o}^2(\beta,\alpha\otimes\id)=c$. 
Take a projection $p\in\mathcal{O}$ 
such that $p\mathcal{O}p\cong\mathcal{O}_\infty$. 
Since $K_0(\beta_g)$ fixes $K_0(1\otimes p)$ for every $g\in G$, 
$\beta$ induces a cocycle action 
$(\beta,1)^{1\otimes p}:G\curvearrowright A\otimes p\mathcal{O}p$ 
(see Section 4.4). 
It is routine to check 
$\mathfrak{o}^2((\beta,1)^{1\otimes p},\alpha\otimes\id)=c$ 
in $H^2(G,KK^1(A\otimes p\mathcal{O}p,A\otimes p\mathcal{O}p))$. 
Then the proof is completed, 
because $\alpha\otimes\id:G\curvearrowright A\otimes p\mathcal{O}p$ is 
cocycle conjugate to $\alpha$ via an isomorphism 
asymptotically unitarily equivalent to the embedding $a\mapsto a\otimes p$. 

(2) 
By \eqref{LHS}, 
$j:H^1(\Z,H^2(N,K_0(A_\flat)))\to H^3(G,K_0(A_\flat))$ is an isomorphism, 
and hence there exists $c'\in H^2(N,K_0(A_\flat))$ such that $\bar j(c')=c$. 
By Lemma \ref{realize:ka2'}, 
we can find an $\alpha|N$-cocycle $(u_g)_{g\in N}$ in $U(A_\flat)_0$ 
satisfying $\kappa^2(\alpha|N,u)=c'$. 
By Lemma \ref{realization}, 
there exist a cocycle action 
$(\beta,v):G\curvearrowright A\otimes\mathcal{O}_\infty$ 
and a family $(x_g)_{g\in G}$ of unitaries 
in $(A\otimes\mathcal{O}_\infty)^\flat$ 
satisfying the following. 
\begin{itemize}
\item $\beta_g=\alpha_g\otimes\mu^G_g$, $v(g,h)=1$, $v(g\xi^l,\xi^m)=1$, 
$x_g=1$ for all $g,h\in N$ and $l,m\in\Z$. 
\item $(\alpha_g\otimes\mu^G_g)(a)=(\Ad x_g\circ\beta_g)(a)$ holds 
for all $g\in G$ and $a\in A$. 
\item When we put $w(g,h)=x_g\beta_g(x_h)v(g,h)x_{gh}^*$ for $g,h\in G$, 
\[
w(g,h)=1,\quad w(g\xi^l,\xi^m)=1
\quad\text{and}\quad w(\xi,g)=u_{\xi g\xi^{-1}}\otimes1
\]
hold for any $g,h\in N$ and $l,m\in\Z$. 
\end{itemize}
Let $\sigma_g=\Ad x_g\circ\beta_g$ and 
consider the cocycle action $(\sigma,w):G\curvearrowright A_\flat$. 
As $K_1(w(g,h))=0$ for all $g,h\in G$, 
$\mathfrak{o}^2((\beta,v),\alpha\otimes\mu^G)=0$. 
Thanks to Lemma \ref{kappa3=kappa2}, we can conclude that 
\[
\kappa^3(\sigma,w)=\bar j(\kappa^2(\sigma|N,\check w))
=\bar j(\kappa^2((\alpha\otimes\mu^G)|N,u\otimes1))=\bar j(c')=c, 
\]
which implies 
$\mathfrak{o}^3((\beta,v),\alpha\otimes\mu^G)=c+\Ima\tilde h^{1,3}_\beta$. 
\end{proof}

\begin{remark}
Let $\phi:G\to KK(A,A)^{-1}_*$ be a homomorphism. 
In general, we do not know whether $\rho$ is realized 
by a cocycle action $G\curvearrowright A$ or not. 
\end{remark}

Recall that a cocycle action $(\alpha,u):G\curvearrowright A$ is 
said to be locally $KK$-trivial if $KK(\alpha_g)=1_A$ for all $g\in G$. 
The following corollary gives a complete classification of 
locally $KK$-trivial outer actions of poly-$\Z$ groups of Hirsch length three 
on the algebras $\mathcal{O}_n\otimes\mathcal{O}$ and $\mathcal{O}$. 

\begin{corollary}\label{OntimesO}
Let $G$ be a poly-$\Z$ group of Hirsch length three. 
Let $A$ be a unital Kirchberg algebra in the Cuntz standard form. 
We regard $KK^i(A,A)$ as trivial $G$-modules. 
\begin{enumerate}
\item For each $c\in H^2(G,KK^1(A,A))$, 
there exists a locally $KK$-trivial outer action 
$\alpha:G\curvearrowright A$ such that $\mathfrak{o}^2(\alpha,\id)=c$. 
\item Assume further that 
$A$ is stably isomorphic to $\mathcal{O}_n$ or $\mathcal{O}_\infty$. 
For each $c\in H^2(G,KK^1(A,A))$, 
we choose and fix an outer action $\alpha^c:G\curvearrowright A$ as above. 
Then the association 
\[
\beta\mapsto(\mathfrak{o}^2(\beta,\id),
\mathfrak{o}^3(\beta,\alpha^{\mathfrak{o}^2(\beta,\id)}))
\]
gives a bijective correspondence 
between the set of $KK$-trivially cocycle conjugacy classes of 
locally $KK$-trivial outer actions $\beta:G\curvearrowright A$ 
and $H^2(G,KK^1(A,A))\times H^3(G,KK(A,A))$. 
\end{enumerate}
\end{corollary}
\begin{proof}
(1) 
The assertion follows from 
Theorem \ref{exi_Hirsch3} (1) and Theorem \ref{Cuntzstandard1}. 

(2) 
By Remark \ref{tildeh=0}, $\tilde h^{1,3}_\beta$ is zero 
for any cocycle action $(\beta,v):G\curvearrowright A$. 
Then the assertion follows from 
Theorem \ref{uni_Hirsch3}, Theorem \ref{exi_Hirsch3} and 
Theorem \ref{Cuntzstandard1}. 
\end{proof}

The following corollary gives a complete classification of 
outer cocycle actions of poly-$\Z$ groups of Hirsch length three 
on the algebras $\mathcal{O}_n$ and $\mathcal{O}_\infty$. 

\begin{corollary}\label{On}
Let $G$ be a poly-$\Z$ group of Hirsch length three. 
Let $A$ be the Cuntz algebra $\mathcal{O}_n$ or $\mathcal{O}_\infty$. 
We regard $KK^i(A,A)$ as trivial $G$-modules. 
For each $c\in H^2(G,KK^1(A,A))$, 
we choose and fix a locally $KK$-trivial outer action 
$\alpha^c:G\curvearrowright A\otimes\mathcal{O}$ 
such that $\mathfrak{o}^2(\alpha^c,\id)=c$. 
Then the association 
\[
(\beta,v)\mapsto\left(\mathfrak{o}^2((\beta,v),\id),
\mathfrak{o}^3((\beta\otimes\id,v\otimes1),
\alpha^{\mathfrak{o}^2((\beta,v),\id)})\right)
\]
gives a bijective correspondence 
between the set of $KK$-trivially cocycle conjugacy classes of 
outer cocycle actions $(\beta,v):G\curvearrowright A$ 
and $H^2(G,KK^1(A,A))\times H^3(G,KK(A,A))$. 
\end{corollary}
\begin{proof}
This is an immediate consequence of 
Corollary \ref{OntimesO} and Theorem \ref{Cuntzstandard2}. 
\end{proof}

In order to discuss classification of genuine actions 
on the Cuntz algebras $\mathcal{O}_n$, 
we would like to introduce a homomorphism 
\[
h^{0,3}_{(\alpha,u)}:H^0(G,K_0(A))\to H^3(G,K_0(A)). 
\]
Let $A$ be a unital Kirchberg algebra with $K_1(A)=0$. 
Let $(\alpha,u):G\curvearrowright A$ be 
a cocycle action of a countable discrete group $G$. 
Let $p\in A\setminus\{0\}$ be a projection such that $K_0(p)\in K_0(A)^G$. 
As in Section 4.4, we choose partial isometries $x_g\in A$ 
so that $x_gx_g^*=p$ and $x_g^*x_g=\alpha_g(p)$, 
and consider $(\alpha^x,u^x):G\curvearrowright pAp$. 
Since $u^x(g,h)$ is in $U(pAp)=U(pAp)_0$ for every $g,h\in G$, 
we can consider $\kappa^3(\alpha^x,u^x)\in H^3(G,K_0(A))$ 
(see Section 8.1 for the definition of $\kappa^3$). 
Let $h^{0,3}_{(\alpha,u)}$ be the map 
\[
H^0(G,K_0(A))\ni K_0(p)\mapsto\kappa^3(\alpha^x,u^x)\in H^3(G,K_0(A)). 
\]
We can show that $h^{0,3}_{(\alpha,u)}$ is a homomorphism. 
Clearly $h^{0,3}_{(\alpha,u)}(K_0(1))$ is equal to $\kappa^3(\alpha,u)$. 
When $(\alpha,1):G\curvearrowright A$ is a genuine action, 
we write $h^{0.3}_{(\alpha,1)}=h^{0,3}_\alpha$. 
Evidently one has $h^{0,3}_\alpha(K_0(1))=\kappa^3(\alpha,1)=0$. 

\begin{lemma}\label{p*ofo3}
Suppose that $A$ is a unital Kirchberg algebra with $K_1(A)$ trivial. 
Let $(\alpha,u)$ be a cocycle action of a countable discrete group $G$ on $A$ 
and let $\beta$ be an action of $G$ on $A$ with $KK(\alpha_g)=KK(\beta_g)$ 
for all $g\in G$. 
Let $p\in A$ be a non-zero projection with $K_0(p)\in K_0(A)^G$. 
Assume that $\mathfrak{o}^2((\alpha,u),\beta)=0$. 
Then one has 
\[
p_*(\mathfrak{o}^3((\alpha,u),\beta))
=h^{0,3}_{(\alpha,u)}(K_0(p))-h^{0,3}_\beta(K_0(p)). 
\]
\end{lemma}
\begin{proof}
Since $K_1(A)$ is trivial, $h^{1,3}_{(\alpha,u)}$ is obviously a zero map. 
Hence, by Lemma \ref{h13p=ph13}, 
$p_*(\mathfrak{o}^3((\alpha,u),\beta))$ is well-defined 
as an element in $H^3(G,K_0(A))$. 

Notice that $K_0(A)$ is canonically isomorphic to $K_0(A^\flat)$ 
via the embedding $A\to A^\flat$. 
Choose a family of unitaries $(v_g)_g$ in $A^\flat$ so that 
\[
(\Ad v_g\circ\alpha_g)(a)=\beta_g(a),\quad 
w(g,h)=v_g\alpha_g(v_h)u(g,h)v_{gh}^*\in U(A_\flat)_0
\]
for all $g,h\in G$ and $a\in A$. 
Let $x_g\in A$ be a partial isometry 
satisfying $x_gx_g^*=p$ and $x_g^*x_g=\beta_g(p)$. 
Then $y_g=x_gv_g\in A^\flat$ is a partial isometry 
satisfying $y_gy_g^*=p$ and $y_g^*y_g=\alpha_g(p)$. 
We have 
\begin{align*}
y_g\alpha_g(y_h)u(g,h)y_{gh}^*
&=x_gv_g\alpha_g(x_hv_h)u(g,h)v_{gh}^*x_{gh}^*\\
&=x_g\beta_g(x_h)v_g\alpha_g(v_h)u(g,h)v_{gh}^*x_{gh}^*\\
&=x_g\beta_g(x_h)x_{gh}^*\cdot pw(g,h). 
\end{align*}
This equation implies 
$h^{0,3}_{(\alpha,u)}(K_0(p))
=h^{0,3}_\beta(K_0(p))+p_*(\mathfrak{o}^3((\alpha,u),\beta))$. 
\end{proof}

\begin{proposition}\label{exi_Hirsch3'}
Let $G$ be a poly-$\Z$ group of Hirsch length three and 
let $A$ be a unital Kirchberg algebra. 
Suppose that $K_1(A)$ is trivial and 
the homomorphism $1_*:K_0(A_\flat)\to K_0(A)$ is surjective. 
Let $(\alpha,u):G\curvearrowright A$ be an outer cocycle action. 
Then there exists an outer action $\beta:G\curvearrowright A$ 
such that $KK(\alpha_g)=KK(\beta_g)$ for every $g\in G$ and 
$\mathfrak{o}^2((\alpha,u),\beta)=0$. 
\end{proposition}
\begin{proof}
Let $p\in\mathcal{O}$ be a projection 
such that $p\mathcal{O}p\cong\mathcal{O}_\infty$. 
By Theorem \ref{Cuntzstandard1}, 
$(\alpha\otimes\id,u\otimes1):G\curvearrowright A\otimes\mathcal{O}$ 
is $KK$-trivially cocycle conjugate to 
an outer action $\alpha':G\curvearrowright A\otimes\mathcal{O}$. 
Put 
\[
c=h^{0,3}_{\alpha'}(K_0(1\otimes p))\in H^3(G,K_0(A\otimes\mathcal{O})). 
\]
Since $1_*:K_0(A_\flat)\to K_0(A)$ is surjective, 
$(1\otimes p)_*:K_0((A\otimes\mathcal{O})_\flat)
\to K_0(A\otimes\mathcal{O})$ is surjective. 
Hence 
\[
(1\otimes p)_*:H^3(G,K_0((A\otimes\mathcal{O})_\flat))
\to H^3(G,K_0(A\otimes\mathcal{O}))
\]
is also surjective, because $G$ is a poly-$\Z$ group of Hirsch length three. 
Therefore we can find $c'\in H^3(G,K_0((A\otimes\mathcal{O})_\flat))$ 
such that $(1\otimes p)_*(c')=c$. 
It follows from 
Theorem \ref{exi_Hirsch3} (2) and Theorem \ref{Cuntzstandard1} that 
there exists an outer action $\beta:G\curvearrowright A\otimes\mathcal{O}$ 
such that $KK(\alpha'_g)=KK(\beta_g)$, $\mathfrak{o}^2(\beta,\alpha')=0$ 
and $\mathfrak{o}^3(\beta,\alpha')=-c'$. 
Thus, 
\[
(1\otimes p)_*(\mathfrak{o}^3(\beta,\alpha'))=(1\otimes p)_*(-c')
=-c=-h^{0,3}_{\alpha'}(K_0(1\otimes p)). 
\]
On the other hand, by Lemma \ref{p*ofo3}, 
\[
(1\otimes p)_*(\mathfrak{o}^3(\beta,\alpha'))
=h^{0,3}_\beta(K_0(1\otimes p))-h^{0,3}_{\alpha'}(K_0(1\otimes p)), 
\]
and so $h^{0,3}_\beta(K_0(1\otimes p))=0$. 
Moreover, 
\[
h^{0,3}_\beta(K_0(1-1\otimes p))
=h^{0,3}_\beta(K_0(1))-h^{0,3}_\beta(K_0(1\otimes p))=0-0=0. 
\]
Hence, by Lemma \ref{2cv_Hirsch3}, 
both $(\beta,1)^{1\otimes p}$ and $(\beta,1)^{1-1\otimes p}$ 
are cocycle conjugate to genuine actions. 
Therefore we can perturb $\beta$ by a $\beta$-cocycle 
so that $\beta_g(1\otimes p)=1\otimes p$ for every $g\in G$. 
Then from Lemma \ref{reduction} (2), we get 
\[
\mathfrak{o}^2((\alpha\otimes\id,u\otimes1)^{1\otimes p},\beta^{1\otimes p})
=\mathfrak{o}^2((\alpha',1)^{1\otimes p},\beta^{1\otimes p})
=\mathfrak{o}^2(\alpha',\beta)=0. 
\]
By Remark \ref{Oinfty_absorb'}, 
$(\alpha\otimes\id,u\otimes1)^{1\otimes p}:
G\curvearrowright A\otimes p\mathcal{O}p$ is 
cocycle conjugate to $(\alpha,u)$, which completes the proof. 
\end{proof}

\begin{remark}
Suppose that $A$ satisfies the UCT and $K_1(A)$ is trivial. 
Then $K_0(A_\flat)=KK(A,A)$ is isomorphic to $\Hom(K_0(A),K_0(A))$, 
and the homomorphism $1_*:K_0(A_\flat)\to K_0(A)$ is equal to 
\[
\Hom(K_0(A),K_0(A))\ni\psi\mapsto\psi(K_0(1))\in K_0(A). 
\]
In particular, the Cuntz algebras $\mathcal{O}_n$ 
satisfy the hypothesis of Proposition \ref{exi_Hirsch3'}. 
\end{remark}

The following corollary gives a complete classification of 
genuine actions of poly-$\Z$ groups of Hirsch length three 
on the algebras $\mathcal{O}_n$. 

\begin{corollary}\label{On'}
Let $G$ be a poly-$\Z$ group of Hirsch length three and 
let $A$ be the Cuntz algebra $\mathcal{O}_n$. 
The association $\alpha\mapsto\mathfrak{o}^2(\alpha,\id)$ 
gives a bijective correspondence 
between the set of $KK$-trivially cocycle conjugacy classes of 
outer actions $\alpha:G\curvearrowright A$ and $H^2(G,KK^1(A,A))$. 
\end{corollary}
\begin{proof}
By the remark above, $1_*:K_0(A_\flat)\to K_0(A)$ is an isomorphism. 
As $G$ is a poly-$\Z$ group of Hirsch length three, 
$1_*:H^3(G,K_0(A_\flat))\to H^3(G,K_0(A))$ is also an isomorphism. 
When $\alpha$ and $\beta$ are actions of $G$ on $A$ 
such that $\mathfrak{o}^2(\alpha,\beta)=0$, by Lemma \ref{p*ofo3}, 
we have 
\[
1_*(\mathfrak{o}^3(\alpha,\beta))
=h^{0,3}_\alpha(K_0(1))-h^{0,3}_\beta(K_0(1))=0-0=0. 
\]
Thus $\mathfrak{o}^3(\alpha,\beta)=0$. 
Therefore, 
the injectivity of $\alpha\mapsto\mathfrak{o}^2(\alpha,\id)$ follows 
from Theorem \ref{uni_Hirsch3}. 

The surjectivity of $\alpha\mapsto\mathfrak{o}^2(\alpha,\id)$ follows 
from Theorem \ref{exi_Hirsch3} (1) and Proposition \ref{exi_Hirsch3'}. 
\end{proof}

\begin{example}
\begin{enumerate}
\item For $G=\Z^3$, we have 
\[
H^2(\Z^3,KK^1(\mathcal{O}_n,\mathcal{O}_n))
\cong H^2(\Z^3,\Z_{n-1})
\cong(\Z_{n-1})^3
\]
and 
\[
H^3(\Z^3,KK(\mathcal{O}_n,\mathcal{O}_n))
\cong H^3(\Z^3,\Z_{n-1})
\cong\Z_{n-1}. 
\]
Hence, there exist exactly $(n{-}1)^3$ cocycle conjugacy classes 
of outer $\Z^3$-actions on $\mathcal{O}_n$, 
and there exist exactly $(n{-}1)^4$ cocycle conjugacy classes 
of outer cocycle actions of $\Z^3$ on $\mathcal{O}_n$. 
\item Let $G$ be the discrete Heisenberg group. 
We have 
\[
H^2(G,KK^1(\mathcal{O}_n,\mathcal{O}_n))
\cong H^2(G,\Z_{n-1})
\cong(\Z_{n-1})^2
\]
and 
\[
H^3(G,KK(\mathcal{O}_n,\mathcal{O}_n))
\cong H^3(G,\Z_{n-1})
\cong\Z_{n-1}. 
\]
Hence, there exist exactly $(n{-}1)^2$ cocycle conjugacy classes 
of outer $G$-actions on $\mathcal{O}_n$, 
and there exist exactly $(n{-}1)^3$ cocycle conjugacy classes 
of outer cocycle actions of $G$ on $\mathcal{O}_n$. 
\end{enumerate}
\end{example}

\begin{example}\label{On''}
Let $G$ be a poly-$\Z$ group of Hirsch length three. 
By Corollary \ref{On}, 
there exists a bijective correspondence between 
the set of cocycle conjugacy classes of outer cocycle actions 
$G\curvearrowright\mathcal{O}_\infty$ and $H^3(G,\Z)$, 
while outer actions $G\curvearrowright\mathcal{O}_\infty$ 
are unique up to cocycle conjugacy. 
\end{example}

\newcommand{\noopsort}[1]{}

\end{document}